\documentclass[reqno]{amsart}

\usepackage{amscd, amssymb, amsthm, mathrsfs, amsmath, amsrefs}
\usepackage[all]{xy}

\usepackage{graphicx}
\usepackage{txfonts}
\usepackage[all]{xy}
\usepackage{accents}
\usepackage{bbm}
\usepackage{lscape}
\usepackage{tikz}
\usetikzlibrary{matrix,arrows}
\usepackage[bookmarks]{hyperref}
\usepackage{enumitem}

\theoremstyle{plain}
\newtheorem{thm}[subsection]{Theorem}
\newtheorem{ax}[subsection]{Axiom}
\newtheorem{prop}[subsection]{Proposition}
\newtheorem{coro}[subsection]{Corollary}
\newtheorem{lem}[subsection]{Lemma}
\newcommand{\comment}[1]{}

\theoremstyle{definition}
\newtheorem{defi}[subsection]{Definition}

\theoremstyle{remark}
\newtheorem{rem}[subsection]{Remark}

\theoremstyle{definition}
\newtheorem{exa}[subsection]{Example}

\newcommand{\Alg}{\mathrm{Alg}}
\newcommand{\Mod}{\mathrm{Mod}}
\renewcommand{\S}{\mathbb{S}}
\newcommand{\Z}{\mathbb{Z}}
\newcommand{\N}{\mathbb{N}}
\newcommand{\Hom}{\mathrm{Hom}}

\newcommand{\colim}{\operatornamewithlimits{colim}}
\newcommand{\Zo}{{\mathbb{Z}_{\geq 0}}}
\newcommand{\sd}{\mathrm{sd}}
\newcommand{\Ex}{\mathit{Ex}}
\newcommand{\ind}{\mathrm{ind}}
\newcommand{\op}{\mathrm{op}}

\newcommand{\id}{\mathrm{id}}
\newcommand{\fS}{\mathfrak{S}}
\newcommand{\ev}{\mathrm{ev}}

\newcommand{\cE}{\mathcal{E}}
\newcommand{\GrAb}{\mathrm{GrAb}}
\newcommand{\pr}{\mathrm{pr}}

\newcommand{\gra}{\mathscr{A}}

\newcommand{\scrZ}{\mathscr{Z}}

\newcommand{\fk}{\mathfrak{K}}

\newcommand{\inc}{\mathrm{incl}}
\newcommand{\tilT}{\widetilde{T}}

\newcommand{\SigmaHo}{{\Sigma\mathrm{Ho}}}

\usepackage{tensor}

\newcommand{\fC}{{\mathfrak{C}}}
\newcommand{\scrU}{\mathscr{U}}
\newcommand{\scrE}{\mathscr{E}}
\renewcommand{\bul}{\bullet}

\newcommand{\Minf}{{M_\infty}}

\newcommand{\lar}[1]{{#1}^{\prec}}

\newcommand{\tria}{\mathscr{T}}
\newcommand{\scrP}{\mathscr{P}}

\newcommand{\Algl}{{\mathrm{Alg}_\ell}}
\newcommand{\Set}{\mathrm{Set}}

\newcommand{\kk}{{kk}}

\newcommand{\fR}{\mathfrak{R}}
\newcommand{\fF}{\mathfrak{F}}
\newcommand{\Fspl}{\mathfrak{F}_{\mathrm{spl}}}
\newcommand{\bbK}{\mathbb{K}}

\begin{document}

\title{The loop-stable homotopy category of algebras}
\author{Emanuel Rodr\'\i guez Cirone}
\email{ercirone@dm.uba.ar}
\address{Dep. Matem\'atica-IMAS, FCEyN-UBA\\ Ciudad Universitaria Pab 1\\
	1428 Buenos Aires\\ Argentina}

\begin{abstract}
Let $\ell$ be a commutative ring with unit. Garkusha constructed a functor from the category of $\ell$-algebras into a triangulated category $D$, that is a universal excisive and homotopy invariant homology theory. Later on, he provided different descriptions of $D$, as an application of his motivic homotopy theory of algebras. Using these, it  can be shown that $D$ is triangulated equivalent to a category, denote it by $\fk$, whose objects are pairs $(A,m)$ with $A$ an $\ell$-algebra and $m$ an integer, and whose Hom-sets can be described in terms of homotopy classes of morphisms. All these computations, however, require a heavy machinery of homotopy theory. In this paper, we give a more explicit construction of the triangulated category $\fk$ and prove its universal property, avoiding the homotopy-theoretic methods and using instead the ones developed by Corti\~nas-Thom for defining $\kk$-theory. Moreover, we give a new description of the composition law in $\fk$, mimicking the one in the suspension-stable homotopy category of bornological algebras defined by Cuntz-Meyer-Rosenberg. We also prove that the triangulated structure in $\fk$ can be defined using either extension or mapping path triangles.
\end{abstract}


\maketitle

\section{Introduction}
Throughout this text, $\ell$ is a commutative ring with unit and $\Algl$ stands for the category of associative and not necessarily unital $\ell$-algebras and $\ell$-algebra homomorphisms.

Algebraic $kk$-theory is a bivariant homology theory on $\Algl$, defined by Corti\~nas-Thom in \cite{cortho} as a completely algebraic analogue of Kasparov's $KK$-theory. It consists of a triangulated category $\kk$ endowed with a functor $j:\Algl\to\kk$ that is excisive (E), polinomial homotopy invariant (H) and $\Minf$-stable (M). The excision property (E) means that every short exact sequence in $\Algl$ that splits as a sequence of $\ell$-modules gives rise to a distinguished triangle upon applying $j$. The property (M) is a kind of matrix-stability. The functor $j$ is moreover universal with the above properties: any other functor from $\Algl$ into a triangulated category satisfying (E), (H) and (M) factors uniquely through $j$. A remarkable property of $kk$-theory is that it recovers Weibel's homotopy $K$-theory \cite{cortho}*{Theorem 8.2.1}. The work of Corti\~nas-Thom was motivated by the work of Cuntz on the bivariant $K$-theory of locally convex algebras \cite{cuntz}. In particular, the Hom-sets and the composition law in $\kk$ were defined explicitely in terms of polynomial homotopy classes of $\ell$-algebra homomorphisms.

With completely different methods, Garkusha constructed in \cite{garkuuni} various universal bivariant homology theories of algebras. Starting with an admissible category of $\ell$-algebras $\fR$ and a class $\fF$ of fibrations on it, he inverted certain weak equivalences to obtain a triangulated category $D(\fR,\fF)$. The latter is endowed with a functor $\fR\to D(\fR,\fF)$ that is $\fF$-excisive, polynomial homotopy invariant and universal with these properties \cite{garkuuni}*{Theorem 2.6}. At this stage, Garkusha did not consider any matrix-stability, motivated by the fact that some interesting admissible categories ---such as the one of commutative algebras--- are not closed under matrices. He showed, moreover, that different matrix-stabilities can be added later on. In particular, he gave a new construction of $\kk$-theory, taking $\fF$ to be the class $\Fspl$ of $\ell$-split surjections and adding $\Minf$-stability.

This paper is concerned with the universal excisive and homotopy invariant homology theory of $\ell$-algebras ---$D(\Algl,\Fspl)$ in the notation of \cite{garkuuni}. Garkusha provided different descriptions of $D(\Algl,\Fspl)$ as an application of his motivic homotopy theory of algebras; we proceed to briefly recall some details of this. Garkusha showed in \cite{garku} that the universal homology theories defined in \cite{garkuuni} are represented by spectra. In particular, for every pair of algebras $(A,B)$, he explicitely constructed an $\Omega$-spectrum $\bbK(A,B)$ such that
\[\Hom_{D(\Algl,\Fspl)}(A,\Omega^nB)\cong \pi_n\bbK(A,B)\]
for all $n\in\Z$ \cite{garku}*{Comparison Theorem B}; see also \cite{garku2}*{Theorem 4.4 and Corollary 4.3}. The homotopy groups of $\bbK(A,B)$ can be computed as follows; see \cite{garku}*{Corollary 7.1} and \cite{htpysimp}*{Theorem 4.11}:
\[\pi_n\bbK(A,B)\cong \colim_v[J^{v}A, B^{\fS_{n+v}}]\]
Here, the square brackets stand for homotopy classes, $JA$ is the kernel of the universal extension of $A$ (see Section \ref{subsec:extensions}), and $B^{\fS_N}$ is the ind-algebra of polynomial functions on the $N$-dimensional cube that vanish at the boundary of the cube (see Example \ref{exa:fSn} and Section \ref{sec:vanish}). It is easily deduced from the above that $D(\Algl,\Fspl)$ is equivalent to a category, denote it by $\fk$, whose objects are pairs $\{(A,m)\}_{A\in\Algl,\,m\in\Z}$ and whose morphisms are defined by:
\begin{equation}\label{eq:introhom}\Hom_{\fk}((A,m),(B,n))=\colim_v[J^{m+v}A, B^{\fS_{n+v}}]\end{equation}
We call $\fk$ the \emph{loop-stable homotopy category} of algebras. We emphasize that the cited computations in \cite{garku} and \cite{garku2} all require a heavy machinery of homotopy theory. In this paper, we avoid this kind of methods and use instead the ones developed by Corti\~nas-Thom in \cite{cortho} to give a more explicit construction of $\fk$ and prove its universal property. In particular, we provide a new explicit formula for the composition law in $\fk$ in terms of \eqref{eq:introhom}. The main results of the paper can be summarized as follows.

\begin{thm}[Theorem \ref{thm:compowd}, Proposition \ref{prop:trianglesinkk} and Theorem \ref{thm:unitria}] Let $\fk$ be the category whose objects are pairs $(A,m)$ with $A$ an $\ell$-algebra and $m$ an integer, whose morphisms are defined by \eqref{eq:introhom} and whose composition law is induced by the operation $\star$ of Definition \ref{defi:kkc}. Then $\fk$ admits a triangulated structure such that the obvious functor $j:\fk\to\Algl$, $j(A)=(A,0)$, is a universal excisive and homotopy invariant homology theory. Moreover, this triangulated structure can be defined using either extension or mapping path triangles.
\end{thm}

If we regard $\kk$-theory as the algebraic counterpart of Kasparov's $K$-theory, then we should think of $\fk$ as the algebraic counterpart of the category $\SigmaHo$ defined by Cuntz-Meyer-Rosenberg in the context of bornological algebras \cite{ralf}*{Section 6.3}. Notice that the Hom-sets in $\fk$ \eqref{eq:introhom} and the Hom-sets in $\SigmaHo$ are defined in a similar fashion. However, two technical issues arise when we want to mimic in $\fk$ the definition of the composition law in $\SigmaHo$. The first one is how to make sense of a sign that appears when permuting coordinates. This is easily handled, since $B^{\fS_N}$ is a group object in the category of ind-algebras if $N\geq 1$ and this group is abelian if $N\geq 2$ \cite{htpysimp}*{Theorem 3.10}. The second difficulty is that $(B^{\fS_M})^{\fS_N}\not\cong B^{\fS_{M+N}}$ in $\Algl$ due to the failure of the exponential law \cite{cortho}*{Remark 3.1.4}. To deal with this, we use certain morphisms $\mu^{M,N}:(B^{\fS_M})^{\fS_N}\to B^{\fS_{M+N}}$ that induce isomorphisms in any excisive and homotopy invariant homology theory \cite{htpysimp}*{Section 3.1}. Modulo these technicalities, we can define a composition law in $\fk$ essentially as in $\SigmaHo$; this is done in Section \ref{sec:defik}. Section \ref{sec:well} is devoted to the proof of Theorem \ref{thm:compowd}, which asserts that the composition law in $\fk$ is well-defined.

The rest of this paper is organized as follows. In Section \ref{sec:add} we prove that $\fk$ is an additive category. In section \ref{sec:exci} we begin to establish the excision properties of $\fk$. In Theorem \ref{thm:excision} we show that there is an exact sequence of abelian groups as follows, for every $D\in\Algl$ and every short exact sequence $A\to B\to C$ that splits as a sequence of $\ell$-modules.
\begin{equation}\label{intro:excision}
\xymatrix@C=1.7em{\fk(D,B^{\fS_1}_0)\ar[r] & \fk(D,C^{\fS_1}_0)\ar[r]^-{\partial} & \fk(D,A)\ar[r] & \fk(D,B)\ar[r] & \fk(D,C)}
\end{equation}
In Section \ref{sec:trans} we define the translation endofunctor of $\fk$ and in Section \ref{sec:long} we extend  \eqref{intro:excision} to a long exact sequence (Lemma \ref{lem:longexact}). In section \ref{sec:triangu} we prove that $\fk$ is a triangulated category. We show that the triangulated structure can be defined by either extension or mapping path triangles (Proposition \ref{prop:trianglesinkk}). In section \ref{sec:uniprop} we prove that the functor $j:\Algl\to\fk$ is a universal excisive and homotopy invariant homology theory (Theorem \ref{thm:unitria}). The Appendix \ref{sec:app} contains the proof of Proposition \ref{prop:bundle}; this is a technical result that is closely related to the proof of \cite{htpysimp}*{Theorem 3.10}.

\section{Preliminaries}\label{sec:conv}We regard simplicial $\ell$-algebras as simplicial sets via the forgetful functor $\Algl\to \Set$. The symbol $\otimes$ indicates tensor product over $\Z$. If $\fC$ is a category and $X$, $Y$ are two of its objects, we may write $\fC(X,Y)$ instead of $\Hom_\fC(X,Y)$.

\subsection{Directed diagrams}\label{sec:directed}

Let $\fC$ be a category. A \emph{directed diagram} in $\fC$ is a functor $X:I\to\fC$, where $I$ is a filtering poset; we often write $(X,I)$ or $X_\bul$ for such a functor. We shall consider different notions of morphisms between directed diagrams:

\subsubsection{Fixing the filtering poset} Let $I$ be a filtering poset. We write $\fC^I$ for the category of functors $X:I\to\fC$ with natural transformations as morphisms.

\subsubsection{Varying the filtering poset} We write $\lar{\fC}$ for the category whose objects are the directed diagrams in $\fC$ and whose morphisms are defined as follows. Let $(X,I)$ and $(Y,J)$ be two directed diagrams. A morphism from $(X,I)$ to $(Y,J)$ consists of a pair $(f,\theta)$ where $\theta:I\to J$ is a functor and $f:X\to Y\circ\theta$ is a natural transformation.

For fixed $I$, there is a functor $a_I:\fC^I\to \lar{\fC}$ that is the identity on objects and sends a natural transformation $f$ to the morphism $(f,\id_I)$.

\subsubsection{The ind-category} We write $\fC^\ind$ for the category whose objects are the directed diagrams in $\fC$ 	and whose hom-sets are defined by:
\[\Hom_{\fC^\ind}\left( (X,I),(Y,J)\right):=\lim_{i\in I}\colim_{j\in J}\Hom_\fC(X_i,Y_j)\]

There is a functor $\lar{\fC}\to\fC^\ind$ that acts as the identity on objects and that sends a morphism $(f,\theta):(X,I)\to(Y,J)$ to the morphism:
\[\left\{f_i:X_i\to Y_{\theta(i)}\right\}_{i\in I}\in \lim_{i\in I}\colim_{j\in J}\Hom_\fC(X_i,Y_j)\]

\begin{lem}\label{lem:indI}
	Let $I$ be a filtering poset and let $\fC$ be a category. Then there exist functors $b_I:\lar{(\fC^I)}\to \lar{\fC}$ and $c_I:(\fC^I)^\ind\to \fC^\ind$ such that, for every filtering poset $J$, the following diagram commutes:
	\begin{equation}\label{eq:indI}\begin{gathered}\xymatrix{(\fC^I)^J\ar[r]^-{a_J}\ar[d]_-{\cong} & \lar{(\fC^I)}\ar[d]^-{b_I}\ar[r] & (\fC^I)^\ind\ar[d]^-{c_I} \\
		\fC^{I\times J}\ar[r]^-{a_{I\times J}} & \lar{\fC}\ar[r] & \fC^\ind \\}\end{gathered}\end{equation}
\end{lem}
\begin{proof}
If $X:J\to \fC^I$ is a directed diagram in $\fC^I$, define both $b_I(X)$ and $c_I(X)$ to be the functor $I\times J\to \fC$ obtained by the exponential law. If $(f,\theta):(X,J)\to (Y,K)$ is a morphism in $\lar{(\fC^I)}$, then $f:b_I(X)\to b_I(Y)\circ(\id_I\times\theta)$ is a natural transformation of functors $I\times J\to \fC$; this defines $b_I$ on morphisms. Let us define $c_I$ on morphisms. If $X:J\to\fC^I$ and $Y:K\to \fC^I$ are directed diagrams in $\fC^I$, we should define a function:
\[\xymatrix{\Hom_{(\fC^I)^\ind}(X,Y)\ar[r] & \Hom_{\fC^\ind}(c_I(X),c_I(Y))=\displaystyle\lim_{(i,j)\in I\times J}\Hom_{\fC^\ind}(X_j(i),c_I(Y))}\]
This is equivalent to defining compatible functions:
\[\xymatrix{\pi_{(i,j)}:\Hom_{(\fC^I)^\ind}(X,Y)\ar[r] & \Hom_{\fC^\ind}(X_j(i),c_I(Y))}\]
Fix $(i,j)\in I\times J$. We let $\pi_{(i,j)}$ be the following composite, that we proceed to explain:
\[\xymatrix@C=1em{\Hom_{(\fC^I)^\ind}(X,Y)\ar[r] & \Hom_{\fC^\ind}(X_\bul(i), Y_\bul(i))\ar[r] & \Hom_{\fC^\ind}(X_j(i), Y_\bul(i))\ar[r] & \Hom_{\fC^\ind}(X_j(i), c_I(Y))}\]
The function on the left is induced by the functor $\fC^I\to \fC$ that evaluates a diagram at the object $i\in I$. The function in the middle is precomposition with the obvious morphism $X_j(i)\to X_\bul(i)$ in $\fC^\ind$. The function on the right is composition with the obvious morphism $Y_\bul(i)\to c_I(Y)$ in $\fC^\ind$. It is straightforward but tedious to verify that the definitions above determine functors $b_I$  and $c_I$ making the diagram \eqref{eq:indI} commute.
\end{proof}

\subsection{Simplicial enrichment of algebras}\label{subsec:simpenrich} Let $\S$ be the category of simplicial sets. We briefly recall the simplicial enrichment of $\Algl$ introduced by Corti\~nas-Thom in \cite{cortho}*{Section 3}. Let $p\geq 0$ and let $\Z^{\Delta^p}:=\Z[t_0,\dots ,t_p]/\langle 1-\textstyle\sum t_i\rangle$. We shall think of $\Z^{\Delta^p}$ as the ring of polynomial functions on the $p$-dimensional simplex. A non-decreasing function $\varphi:[p]\to[q]$ induces a ring homomorphism $\Z^{\Delta^q}\to\Z^{\Delta^p}$ by the formula:
\[t_i\mapsto \sum_{\varphi(j)=i}t_j\]
This defines a simplicial ring $\Z^\Delta$. For $B\in\Algl$, put $B^\Delta:=\Z^\Delta\otimes B$; this is a simplicial $\ell$-algebra. For $X\in \S$, put $B^X:=\Hom_\S(X,B^\Delta)$. It is easily verified that $B^X$ is an $\ell$-algebra with the pointwise operations; we shall think of it as the $\ell$-algebra of polynomial functions on $X$ with coefficients in $B$.

\begin{rem}\label{rem:simpliberreta}In general $(B^X)^Y\not\cong B^{X\times Y}$ ---this already fails when $X$ and $Y$ are standard simplices; see \cite{cortho}*{Remark 3.1.4}.
\end{rem}

Let $A$ and $B$ be two $\ell$-algebras. The \emph{simplicial mapping space} from $A$ to $B$ is the simplicial set $\Hom_\Algl(A,B^\Delta)$. For $A,B\in \Algl$ and $X\in \S$ we have an adjunction isomorphism:
\[\Hom_\S(X,\Hom_\Algl(A,B^\Delta))\cong\Hom_\Algl(A,B^X)\]

\subsection{Simplicial pairs} A \emph{simplicial pair} is a pair $(K,L)$ where $K$ is a simplicial set and $L\subseteq K$ is a simplicial subset. A \emph{morphism of pairs} from $(K',L')$ to $(K,L)$ is a morphism of simplicial sets $f:K'\to K$ such that $f(L')\subseteq L$. A pair $(K,L)$ is \emph{finite} if $K$ is a finite simplicial set.

\begin{exa}\label{exa:fSn} Let $I:=\Delta^1$ and $\partial I:=\{0,1\}\subset I$. For $n\geq 1$, let $I^n:=I\times\cdots\times I$ be the $n$--fold direct product and let $\partial I^n$ be the boundary of $I^n$:
	\[\partial I^n:=\left[(\partial I)\times I\times\cdots\times I\right]\cup\left[I\times (\partial I)\times\cdots\times I\right]\cup\cdots\cup\left[I\times\cdots\times I\times (\partial I)\right]\]
Put $I^0:=\Delta^0$ and $\partial I^0:=\emptyset$. We will write $\fS_n$ for the simplicial pair $(I^n,\partial I^n)$.
\end{exa}

\begin{exa}We shall think of a simplicial set $K$ as the pair $(K,\emptyset)$. This gives an embedding of the category of (finite) simplicial sets as a full subcategory of the category of (finite) simplicial pairs. We usually make this identification without further mention.
\end{exa}

\begin{defi}
	Let $(K, L)$ and $(K', L')$ be simplicial pairs. The \emph{box product} of $(K, L)$ with $(K', L')$ is defined as the simplicial pair:
	\[(K,L)\square(K', L'):=(K\times K', (K\times L')\cup(L\times K'))\]
\end{defi}

It is easily verified that $\square$ is a symmetric monoidal operation with unit $\Delta^0$ on (finite) simplicial pairs. It equals the cartesian product when restricted to simplicial sets.

\begin{exa}
For $m.n\geq 0$, we have $\fS_m\square\fS_n=\fS_{m+n}$.
\end{exa}

In the rest of the paper we will only consider \emph{finite} simplicial pairs, omitting the word `finite' from now on.

\subsection{Functions vanishing on a subset}\label{sec:vanish} Let $\sd:\S\to\S$ be the subdivision functor; see \cite{goja}*{Section III.4}. We have a natural transformation $\gamma:\sd\to\id_\S$ called the \emph{last vertex map}. For a simplicial pair $(K,L)$, an $\ell$-algebra $B$ and $r\geq 0$, put:
\[B_r^{(K,L)}:=\ker\left(\xymatrix@C=6em{B^{\sd^rK}\ar[r]^-{\mbox{\small restriction}} & B^{\sd^rL}}\right)\]
The last vertex map induces morphisms $B^{(K,L)}_r\to B^{(K,L)}_{r+1}$ and we usually regard $B^{(K,L)}_\bul$ as a  directed diagram in $\Algl$:
\[\xymatrix{B^{(K,L)}_\bul:B^{(K,L)}_0\ar[r] & B^{(K,L)}_1\ar[r] & B^{(K,L)}_2\ar[r] & \cdots}\]


\begin{exa}
$B^{\fS_0}_\bul=B^{\Delta^0}_\bul$ is the constant $\Zo$-diagram $B$.
\end{exa}

\begin{lem}[\cite{htpysimp}*{Lemma 2.10}; cf. \cite{cortho}*{Proposition 3.1.3}]\label{lem:simppair}
	Let $(K,L)$ be a (finite) simplicial pair, let $B$ be an $\ell$-algebra and let $r\geq 0$. Then $\Z^{(K,L)}_r$ is a free abelian group and there is a natural $\ell$-algebra isomorphism $B\otimes \Z^{(K,L)}_r\cong B^{(K,L)}_r$.
\end{lem}

\subsection{Multiplication morphisms}\label{sec:multi}

Let $(K,L)$ and $(K',L')$ be simplicial pairs, let $B\in\Algl$ and let $r,s\geq 0$. Recall from \cite{htpysimp}*{Section 3.1} that the multiplication in the simplicial commutative ring $\Z^\Delta$ induces an $\ell$-algebra homomorphism
\begin{equation}\label{eq:mu}\xymatrix{\mu^{(K,L),(K',L')}_B:\left(B^{(K,L)}_r\right)^{(K',L')}_s\ar[r] & B^{(K,L)\square(K',L')}_{r+s}}\end{equation}
such that the following statements hold:
\begin{enumerate}
	\item The morphisms \eqref{eq:mu} are covariant in $B$ and contravariant in $(K,L)$ and in $(K',L')$.
	\item Define $\theta:\Zo\times\Zo\to\Zo$ by $\theta(r,s)=r+s$. Then, for varying $r$ and $s$, the morphisms \eqref{eq:mu} assemble into a morphism in $\lar{\Algl}$:
	\[\xymatrix{\left(\mu^{(K,L),(K',L')}_B,\theta\right):\left(B^{(K,L)}_\bul\right)^{(K',L')}_\bul\ar[r] & B^{(K,L)\square(K',L')}_\bul}\]
	\item The morphisms \eqref{eq:mu} are associative in the obvious way.
	\item The composite
	\[\xymatrix@C=4em{B^{(K,L)}_r\cong\left(B^{(K,L)}_r\right)^{\Delta^0}_s\ar[r]^-{\mu^{(K,L),\Delta^0}_B} &  B^{(K,L)\square \Delta^0}_{r+s}\cong B^{(K,L)}_{r+s}}\]
	is the transition morphism of $B^{(K,L)}_\bul$. An analogous statement holds for $\mu^{\Delta^0,(K,L)}_B$.
\end{enumerate}
To alleviate notation when there is no risk of confusion, we often write $\mu^{(K,L),(K',L')}$ or even just $\mu$ instead of $\mu^{(K,L),(K',L')}_B$. We also write $\mu^{n,m}$ instead of $\mu^{\fS_n,\fS_m}$.

\subsection{Polynomial homotopy}\label{sec:alghtp} Two morphisms $f_0,f_1:A\to B$ in $\Algl$ are \emph{elementary homotopic} if there exists an $\ell$-algebra homomorphism $f:A\to B[t]$ such that $\ev_0\circ f=f_1$ and $\ev_1\circ f=f_0$. Here, $\ev_i$ stands for the evaluation at $i$.  Elementary homotopy $\sim_e$ is a reflexive and symmetric relation, but it is not transitive. We let $\sim$ be its transitive closure and we call $f_0$ and $f_1$ \emph{homotopic} if $f_0\sim f_1$. It is easily shown that $f_0\sim f_1$ iff there exist $r\in\N$ and $f:A\to B^{\Delta^1}_r$ such that the following diagrams commute:
\[\xymatrix@C=3em@R=2em{A\ar[d]_-{f_i}\ar[r]^-{f} & B^{\Delta^1}_r\ar[d]^-{d_i} \\
	B\ar[r]^-{\cong} & B^{\Delta^0}_r \\}\]
Homotopy is compatible with composition, and we have a category $[\Algl]$ whose objects are $\ell$-algebras and whose morphisms are polynomial homotopy classes of $\ell$-algebra homomorphisms \cite{gersten}*{Lemma 1.1}. There is an obvious functor $\Algl\to [\Algl]$.

Now let $A_\bul$ and $B_\bul$ be two directed diagrams in $\Algl$ and let $f,g:A_\bul\to B_\bul$ be two morphisms in $\Algl^\ind$. Following \cite{cortho}*{Definition 3.1.1}, we say that $f$ and $g$ are  \emph{homotopic} if they become equal upon applying the functor $\Algl^\ind\to [\Algl]^\ind$. We will write:
	\[[A_\bul,B_\bul]:=\Hom_{[\Algl]^\ind}\left((A,I),(B,J)\right)\]

It turns out that the homotopy groups of the simplicial mapping space between algebras can be described in terms of polynomial homotopy classes of morphisms:

\begin{thm}[\cite{htpysimp}*{Theorem 3.10}]\label{thm:bij}
	For any pair of $\ell$-algebras $A$ and $B$ and any $n\geq 0$, there is a natural bijection:
	\begin{equation}\label{eq:bij}\pi_n\Hom_\Algl(A,B^\Delta)\cong [A,B^{\fS_n}_\bul]\end{equation}
\end{thm}


\begin{rem}\label{rem:groupstructure}
	Let $A$ and $B$ be two $\ell$-algebras and let $n\geq 1$. Endow the set $[A,B^{\fS_n}_\bul]$ with the group structure for which \eqref{eq:bij} is a group isomorphism. This group structure is abelian if $n\geq 2$. Moreover, if $f:A\to A'$ and $g:B\to B'$ are morphisms in $[\Algl]$, then the following functions are group homomorphisms:
	\[f^*:[A', B^{\fS_n}_\bul]\to [A, B^{\fS_n}_\bul]\]
	\[g_*:[A, B^{\fS_n}_\bul]\to [A, (B')^{\fS_n}_\bul]\]
\end{rem}

\begin{exa}\label{exa:groupinverse}
	Let $\omega$ be the automorphism of $B^{\Delta^1}$ defined by $\omega(t_0)=t_1$ and $\omega(t_1)=t_0$; then $\omega$ induces an automorphism of $B^{\fS_1}_0$. It is easily shown that if $g:A\to B^{\fS_1}_0$ is an $\ell$-algebra homomorphism and $[g]$ represents its class in $[A,B^{\fS_1}_\bul]$, then $[\omega\circ g]=[g]^{-1}$ in $[A,B^{\fS_1}_\bul]$; see \cite{htpysimp}*{Example 3.12} for details.
\end{exa}

\begin{exa}\label{exa:concat}
Fix $r\geq 1$. Note that there is an automorphism of $\sd^r I$ exchanging both endpoints of $I$. We have a pushout square
\[\xymatrix@C=2em@R=2em{\Delta^0\ar[d]_-{d^0}\ar[r]^-{d^1} & \sd^rI\ar[d] \\ \sd^rI\ar[r] & \sd^{r+1}I }\]
that shows that $\sd^{r+1}I$ can be obtained by concatenating two copies of $\sd^rI$; the endpoint $1$ in the first copy gets identified with the $0$ in the second one. Since $B^?:\S^\op\to\Algl$ preserves products, it follows that $B^{\sd^{r+1}I}\cong B^{\sd^rI}\times_BB^{\sd^rI}$. Thus, two morphisms $f,g:A\to B^{\sd^rI}$ such that $d_0f=d_1g$ determine a morphism $f\bullet g:A \to B^{\sd^{r+1}I}$, that we call the \emph{concatenation} of $f$ and $g$. In particular, we can always concatenate $f,g:A\to B^{\fS_1}_r$ to obtain $f\bullet g:A\to B^{\fS_1}_{r+1}$. It is not difficult to show that the group structure in $[A,B^{\fS_1}_\bul]$ is induced by concatenation; we will not use this fact, however.
\end{exa}

\begin{exa}\label{exa:cmn}
	Let $m,n\geq 1$ and let $c:I^m\times I^n\overset{\cong}\to I^n\times I^m$ be the commutativity isomorphism. Then $c$ induces an isomorphism $c^*:B^{\fS_{n+m}}_\bul\to B^{\fS_{m+n}}_\bul$, that in turn induces a bijection $c^*:[A,B^{\fS_{n+m}}_\bul]\to [A,B^{\fS_{m+n}}_\bul]$. We claim that the latter is multiplication by $(-1)^{mn}$. Indeed, this follows from Theorem \ref{thm:bij} and the well known fact that permuting two coordinates in $\Omega^k$ induces multiplication by $(-1)$ upon taking $\pi_0$.
\end{exa}

\subsection{Extensions and classifying maps}\label{subsec:extensions}

An \emph{extension} is a short exact sequence
\begin{equation}\label{eq:extension}\scrE:\xymatrix{A\ar[r] & B\ar[r] & C\\}\end{equation}
of $\ell$-algebras that splits as a sequence of $\ell$-modules. A \emph{morphism of extensions} is a morphism of diagrams in $\Algl$. We often write $(\scrE,s)$ to indicate that we consider $\scrE$ together with a particular $\ell$-linear splitting $s$. A \emph{strong morphism of extensions} $(\scrE',s')\to (\scrE,s)$ is a morphism of extensions $(a,b,c):\scrE'\to \scrE$ that is compatible with the specified splittings; i.e. such that the following square of modules commutes:
\[\xymatrix@R=2em{B'\ar[d]_{b} & C'\ar[l]_-{s'}\ar[d]^-{c} \\
	B & C\ar[l]_-{s} \\}\]

\begin{rem}
If $\scrE$ is an extension and $G$ is a ring, then $\scrE\otimes G$ is an extension. In particular, if $(K,L)$ is a simplicial pair, then $\scrE^{(K,L)}_r$ is an extension by Lemma \ref{lem:simppair}.
\end{rem}

Let $\Mod_\ell$ be the category of $\ell$-modules. The forgetful functor $F:\Algl\to\Mod_\ell$ has a right adjoint $\tilT$; put $T:=\tilT\circ F$. For an $\ell$-algebra $A$, let $\eta_A:TA\to A$ be the counit of this adjunction and put $JA:=\ker \eta_A$. The \emph{universal extension} of $A$ is the sequence
\[\scrU_A:\xymatrix{JA\ar[r] & TA\ar[r]^-{\eta_A} & A\\},\]
split by the unit map $\sigma_A:FA\to F\tilT(FA)$. We will always use $\sigma_A$ as a splitting for $\scrU_A$.

\begin{prop}[cf. \cite{cortho}*{Proposition 4.4.1}]\label{lem:classexists}
	Let \eqref{eq:extension} be an extension with splitting $s$ and let $f:D\to C$ be an $\ell$-algebra homomorphism. Then there exists a unique strong morphism of extensions $\scrU_{D}\to(\scrE,s)$ extending $f$:
	\[\xymatrix{\scrU_D\ar@{..>}[d]_-{\exists !} & JD\ar[r]\ar@{..>}[d]_-{\xi} & TD\ar@{..>}[d]\ar[r]^-{\eta_D} & D\ar[d]^-{f} \\
		(\scrE,s) & A\ar[r] & B\ar[r] & C \\}\]
\end{prop}

The morphism $\xi$ is called the \emph{classifying map of $f$ with respect to $(\scrE,s)$}. When $D=C$ and $f=\id_C$ we call $\xi$ the \emph{classifying map of $(\scrE,s)$}.

\begin{prop}[cf. \cite{cortho}*{Proposition 4.4.1}]\label{lem:classhmtp}
	In the hypothesis of Proposition \ref{lem:classexists}, the homotopy class of the classifying map $\xi$ does not depend upon the splitting $s$.
\end{prop}

Thus, we may refer to the classifying map of \eqref{eq:extension} as a homotopy class $JC\to A$ without specifying a splitting for \eqref{eq:extension}.

\begin{prop}[\cite{cortho}*{Proposition 4.4.2}]\label{lem:classcommute}Let $\scrE_i:A_i\to B_i\to C_i$ be an extension with classifying map $\xi_i$ ($i=1,2$). Let $(a, b,c):\scrE_1\to\scrE_2$ be a morphism of extensions. Then the following square af algebras commutes up to homotopy:
	\[\xymatrix{JC_1\ar[d]_-{\xi_1}\ar[r]^-{J(c)} & JC_2\ar[d]^-{\xi_2} \\
		A_1\ar[r]^-{a} & A_2 \\}\]
	Moreover, if we consider specific splittings for the $\scrE_i$ and $(a,b,c)$ is a strong morphism of extensions then the square above commutes on the nose.
\end{prop}

\begin{lem}[cf. \cite{cortho}*{Corollary 4.4.4}]\label{lem:Jhtp}
	The functor $J:\Algl\to \Algl$ preserves homotopy, and thus defines a functor $J:[\Algl]\to[\Algl]$.
\end{lem}

\subsection{Path extensions}\label{sec:pathext} For $B\in\Algl$ and $n\geq 0$, put:
\[P(n,B)_\bul:=B^{\fS_n\square(I,\{1\})}_\bul\]
We will write $(PB)_\bul$ instead of $P(0,B)_\bul$. The inclusions
\[\fS_{n+1}=\fS_n\square\fS_1\supseteq \fS_n\square(I,\{1\})\supseteq \fS_n\square \{0\}\cong \fS_n\]
induce a sequence of $\Zo$-diagrams:
\begin{equation}\label{eq:pathsequencenB}\xymatrix@C=3em{\scrP_{n,B}: B^{\fS_{n+1}}_r\ar[r] & P(n,B)_r\ar[r] & B^{\fS_n}_r}\end{equation}
We claim that \eqref{eq:pathsequencenB} is an extension. Exactness at $P(n,B)_r$ holds because the functor $B^{\sd^r(-)}:\S\to \Algl^\op$ preserves pushouts and we have:
\[\partial I^{n+1}=\left[ (\partial I^n\times I)\cup (I^n\times \{1\})\right]\cup (I^n\times\{0\})\]
Exactness at $B^{\fS_{n+1}}_r$ follows from the fact that both $B^{\fS_{n+1}}_r$ and $P(n,B)_r$ are subalgebras of $B^{\sd^rI^{n+1}}$. Consider the element $t_0\in\Z^{\Delta^1}$; $t_0$ is actually in $\Z^{(I,\{1\})}_0$ since $d_0(t_0)=0$. It is easily verified that the composite
\[\xymatrix{B^{\fS_n}_r\ar[r]^-{?\otimes t_0} & B^{\fS_n}_r\otimes\Z^{(I,\{1\})}_0\cong\left(B^{\fS_n}_r\right)^{(I,\{1\})}_0\ar[r]^-{\mu} & B^{\fS_n\square(I,\{1\})}_r \\}\]
is a splitting for \eqref{eq:pathsequencenB}; we will always use this as a splitting for \eqref{eq:pathsequencenB}.

\begin{exa}\label{exa:muSm}
	By naturality of $\mu$, there is a strong morphism of extensions:
	\[\xymatrix@C=3em{\scrP_{m,B^{\fS_n}_s}\ar[d] & \left(B^{\fS_n}_s\right)^{\fS_{m+1}}_r\ar[d]_-{\mu}\ar[r] & P(m,B^{\fS_n}_s)_r\ar[r]\ar[d] & \left(B^{\fS_n}_s\right)^{\fS_m}_r\ar[d]^-{\mu} \\
		\scrP_{n+m,B} & B^{\fS_{n+m+1}}_{s+r}\ar[r] & P(n+m,B)_{s+r}\ar[r] & B^{\fS_{n+m}}_{s+r}}\]
\end{exa}

\begin{exa}\label{exa:pathb1}
	For $n=0$, the extension \eqref{eq:pathsequencenB} takes the form:
	\begin{equation}\label{eq:pathb1}\xymatrix{\scrP_{0,B}: B^{\fS_1}_r\ar[r] & (PB)_r\ar[r] & B^{\fS_0}_r\cong B}\end{equation}
	This is the \emph{loop extension} of \cite{cortho}*{Section 4.5} and we will write $\lambda_B$ for its classifying map.
\end{exa}

\begin{exa}\label{exa:pathmonho} Define $\widetilde{P}(n,B)_\bul:=B^{(I,\{1\})\square\fS_n}_\bul$. The inclusions
	\[\fS_{1+n}=\fS_1\square\fS_n\supseteq (I,\{1\})\square \fS_n\supseteq \{0\}\square\fS_n\]
	induce a sequence that turns out to be an extension:
	\begin{equation}\label{eq:pathmonho}\xymatrix@C=3em{\widetilde{\scrP}_{n,B}: B^{\fS_{1+n}}_r \ar[r] & \widetilde{P}(n,B)_r\ar[r] & B^{\fS_n}_r \\}\end{equation}
	The commutativity isomorphism $c:I^n\times I\to I\times I^n$ induces an isomorphism of extensions $c^*:\widetilde{\scrP}_{n,B}\to\scrP_{n,B}$.
\end{exa}

\begin{exa}\label{exa:subdi1}
	It will be useful to have a more explicit description of \eqref{eq:pathb1}. We have:
	\begin{equation}\label{eq:identi}\left\{\begin{array}{l}B^{\Delta^0}=B[t_0]/\langle 1-t_0\rangle\cong B; \\
	B^{\Delta^1}=B[t_0,t_1]/\langle 1-t_0-t_1\rangle\cong B[t],\hspace{1em} t_1\leftrightarrow t.\end{array}\right.\end{equation}
	Under these identifications, the face morphisms $B^{\Delta^1}\to B^{\Delta^0}$ coincide with the evaluations $\ev_i:B[t]\to B$ for $i=0,1$. We have:
	\[(PB)_0=\ker\left(B^{\Delta^1}\overset{d_0}\longrightarrow B^{\Delta^0}\right)\cong \ker\left(B[t]\overset{\ev_1}\longrightarrow B\right)=(t-1)B[t]\]
	\[B^{\fS_1}_0=\ker\left((PB)_0\overset{d_1}\longrightarrow B^{\Delta^0}\right)\cong \ker\left((t-1)B[t]\overset{\ev_0}\longrightarrow B\right)=(t^2-t)B[t]\]
	Hence, the extension \eqref{eq:pathb1} is isomorphic to:
	\begin{equation}\label{eq:pathcop}\xymatrix@C=3em{(t^2-t)B[t]\ar[r]^-{\inc} & (t-1)B[t]\ar[r]^-{\ev_0} & B }\end{equation}
	The section in \eqref{eq:pathb1} identifies with the morphism $B\to (t-1)B[t]$, $b\mapsto b(1-t)$.

	We now want a description of \eqref{eq:pathb1} once a subdivision has been made. Recall that $\sd\Delta^1$ fits into the following pushout:
	\[\xymatrix{\Delta^0\ar[r]^-{d^0}\ar[d]_-{d^0} & \Delta^1\ar[d] \\
		\Delta^1\ar[r] & \sd\Delta^1 \\}\]
	Since the functor $B^{?}:\S^\op\to \Algl$ preserves limits, we have:
	\[B^{\sd\Delta^1}\cong B[t]\tensor[_{\ev_1}]{\times}{_{\ev_1}} B[t]= \{(p,q)\in B[t]\times B[t]: p(1)=q(1)\}\]
	The endpoints of $\sd\Delta^1$ are the images of the coface maps $d^1:\Delta^0\to\Delta^1$ whose codomains are each of the two copies of $\Delta^1$ contained in $\sd\Delta^1$. We get:
	\[(PB)_1=\ker\left(B^{\sd\Delta^1}\longrightarrow B^{\sd\{1\}}\right)\cong B[t]\tensor[_{\ev_1}]{\times}{_{\ev_1}} tB[t]\]
	\[B^{\fS_1}_1=\ker\left(B^{\sd\Delta^1}\longrightarrow B^{\sd(\partial\Delta^1)}\right) \cong tB[t]\tensor[_{\ev_1}]{\times}{_{\ev_1}} tB[t]\]
	In this description of $(PB)_1$ a choice has been made, since the two endpoints of $\sd\Delta^1$ are indistinguishable. The extension \eqref{eq:pathb1} is isomorphic to:
	\begin{equation}\label{eq:path1cop}\xymatrix@C=3em{tB[t]\tensor[_{\ev_1}]{\times}{_{\ev_1}} tB[t]\ar[r]^-{\inc} & B[t]\tensor[_{\ev_1}]{\times}{_{\ev_1}}tB[t]\ar[r]^-{\ev_0\circ\pr_1} & B}\end{equation}
	Here $\pr_1:B[t]\tensor[_{\ev_1}]{\times}{_{\ev_1}} tB[t]\longrightarrow B[t]$ is the projection into the first factor. The section in \eqref{eq:pathb1} identifies with the morphism:
	\[B\to B[t]\tensor[_{\ev_1}]{\times}{_{\ev_1}}tB[t],\hspace{1em}b\mapsto (b(1-t),0).\]
	The last vertex map induces a strong morphism of extensions from \eqref{eq:pathcop} to \eqref{eq:path1cop}; this morphism has the following components:
	\[(t^2-t)B[t]\to tB[t]\tensor[_{\ev_1}]{\times}{_{\ev_1}} tB[t], \hspace{1em} p\mapsto (p,0);\]
	\[(t-1)B[t]\to B[t]\tensor[_{\ev_1}]{\times}{_{\ev_1}} tB[t], \hspace{1em} p\mapsto (p,0).\]
\end{exa}

\begin{lem}\label{lem:pathcontr}
	Let $B\in\Algl$, $n\geq 1$ and $r\geq 0$. Then $P(n,B)_r$ is contractible.
\end{lem}
\begin{proof}
	Let $\vartheta:I\times I\to I$ be the unique morphism of simplicial sets that satisfies:
	\[\xymatrix@R=0em{(0,0)\ar@{|->}[r]^-{\vartheta} & 0\\
		(0,1),(1,0),(1,1)\ar@{|->}[r]^-{\vartheta} & 1}\]
	It is easily verified that $\vartheta$ induces a morphism of pairs $\vartheta:(I,\{1\})\square I\to (I,\{1\})$. The latter, in turn, induces an $\ell$-algebra homomorphism:
	\[\xymatrix{f:=(\fS_n\square\vartheta)^*:P(n,B)_r=B^{\fS_n\square(I,\{1\})}_r\ar[r] & B^{\fS_n\square (I,\{1\})\square I}_r}\]
	The coface maps $d^i:\Delta^0\to I$ induce $\ell$-algebra homomorphisms:
	\[\xymatrix{\delta_i:=(\fS_n\square (I,\{1\})\square d^i)^*:B^{\fS_n\square(I,\{1\})\square I}_r\ar[r] & B^{\fS_n\square(I,\{1\})\square \Delta^0}_r\cong P(n,B)_r}\]
	Note that $\delta_0\circ f=0$ and $\delta_1\circ f=\id_{P(n,B)_r}$. By \cite{garku}*{Hauptlemma (2)}, $\delta_0\circ f$ and $\delta_1\circ f$ represent the same class in $[P(n,B)_r, P(n,B)_r]$. Indeed, both morphisms represent the same class in $[P(n,B)_r, B^{\sd^rI^{n+1}}]$ but the polynomial homotopies constructed in op. cit. preserve our boundary conditions.
\end{proof}

\subsection{Exchanging loop functors}\label{sec:kappa} Let $B\in \Algl$ and let $m,n\geq 0$. We proceed to define, by induction on $n$, a natural transformation:
\[\xymatrix{\kappa^{n,m}_B:J^n(B^{\fS_m}_r)\ar[r] & (J^nB)^{\fS_m}_r}\]
Let $\kappa^{1,m}_B:J(B^{\fS_m}_r)\to (JB)^{\fS_m}_r$ be the classifying map of the extension $\left(\scrU_B\right)^{\fS_m}_r$.
For $n\geq 1$, define inductively $\kappa^{n+1,m}_B$ as the composite:
\[\xymatrix@C=3em{J^{n+1}(B^{\fS_m}_r)\ar[r]^-{J(\kappa^{n,m}_B)} & J\left((J^nB)^{\fS_m}_r\right)\ar[r]^-{\kappa^{1,m}_{J^nB}} & (J^{n+1}B)^{\fS_m}_r\\}\]
Define $J^0$ as the identity functor of $\Algl$ and let $\kappa^{0,m}_B$ be the identity of $B_\bul^{\fS_m}$. When there is no risk of confusion, we may write $\kappa^{n,m}$ instead of $\kappa^{n,m}_B$ to alleviate notation. The next result follows from an easy induction on $n=p+q$.

\begin{lem}\label{lem:kappapq}
Let $p,q,r\geq 0$ and let $B$ be an $\ell$-algebra. Then $\kappa^{p+q,m}_B=\kappa^{q,m}_{J^pB}\circ J^q\left(\kappa^{p,m}_B\right)$. That is, the following diagram in $\Algl$ commutes:
\[\xymatrix{J^{p+q}(B^{\fS_m}_r)\ar[rr]^-{\kappa^{p+q,m}}\ar@/_/[dr]_-{J^q(\kappa^{p,m})} & & (J^{p+q}B)^{\fS_m}_r \\
	& J^q\left((J^pB)_r^{\fS_m}\right)\ar@/_/[ur]_-{\kappa^{q,m}} & }\]
\end{lem}

This lemma is explained as follows. We have $J^{p+q}=J^q\circ J^p$. Thus, in order to exchange $J^{p+q}$ and $(?)^{\fS_m}$, we may first exchange $J^p$ and $(?)^{\fS_m}$ and then $J^q$ and $(?)^{\fS_m}$.

\begin{rem}\label{rem:analogtokappapq}
Replacing $\fS_n$ with any simplicial pair $(K,L)$, we can define morphisms $J^n(B^{(K,L)})\to (J^nB)^{(K,L)}$ and prove Lemma \ref{lem:kappapq} in this setting.
\end{rem}

The following result is an analog of Lemma \ref{lem:kappapq}. Its statement is, however, more complicated since $(B^{\fS_p})^{\fS_q}\not\cong B^{\fS_{p+q}}$.

\begin{lem}\label{lem:penta}
Let $p,q,r,s\geq 0$ and let $B$ be an $\ell$-algebra. Then the following diagram in $\Algl$ commutes:
\[\xymatrix@C=4em{J^n\left(\left(B^{\fS_p}_r\right)^{\fS_q}_s\right)\ar[r]^-{\kappa^{n,q}_{B^{\fS_p}}}\ar[d]_-{J^n\left(\mu^{p,q}_B\right)} & \left(J^n\left(B^{\fS_p}_r\right)\right)^{\fS_q}_s\ar[r]^-{\left(\kappa^{n,p}_B\right)^{\fS_q}} & \left((J^nB)^{\fS_p}_r\right)^{\fS_q}_s\ar[d]^-{\mu^{p,q}_{J^nB}} \\
	J^n\left(B^{\fS_{p+q}}_{r+s}\right)\ar[rr]^-{\kappa^{n,p+q}_B} & & (J^nB)^{\fS_{p+q}}_{r+s} \\}\]
\end{lem}
\begin{proof}
We proceed by induction on $n$. The case $n=1$ follows from Proposition \ref{lem:classcommute} applied to the following strong morphism of extensions:
\[\xymatrix@C=3em{\left(\scrU_{B^{\fS_p}_r}\right)^{\fS_q}_s\ar[d] & \left(J\left(B^{\fS_p}_r\right)\right)^{\fS_q}_s\ar[r]\ar[d]_-{\left(\kappa^{1,p}_B\right)^{\fS_q}} & \left(T\left(B^{\fS_p}_r\right)\right)^{\fS_q}_s\ar[r]\ar[d] & \left(B^{\fS_p}_r\right)^{\fS_q}_s\ar[d]^-{\id} \\
	\left((\scrU_B)^{\fS_p}_r\right)^{\fS_q}_s\ar[d] & \left((JB)^{\fS_p}_r\right)^{\fS_q}_s\ar[r]\ar[d]_-{\mu^{p,q}_{JB}} & \left((TB)^{\fS_p}_r\right)^{\fS_q}_s\ar[r]\ar[d]^-{\mu^{p,q}_{TB}} & \left(B^{\fS_p}_r\right)^{\fS_q}_s\ar[d]^-{\mu^{p,q}_{B}} \\
	(\scrU_B)^{\fS_{p+q}}_{r+s} & (JB)^{\fS_{p+q}}_{r+s}\ar[r] & (TB)^{\fS_{p+q}}_{r+s}\ar[r] & B^{\fS_{p+q}}_{r+s} \\}\]
Now suppose that the diagram commutes for $n$; we will show it also commutes for $n+1$. The following diagram commutes by inductive hypothesis and naturality of $\kappa^{n,1}_?$; we omit the subindices $r$ and $s$ to alleviate notation:
\[\xymatrix@C=4em{J^{n+1}\left( \left(B^{\fS_p}\right)^{\fS_q}\right)\ar[rrd]^-{J^{n+1}\left(\mu^{p,q}_B\right)}\ar[d]_-{J^n\left(\kappa^{1,q}_{B^{\fS_p}}\right)} & & \\
	J^n\left(\left(J\left(B^{\fS_p}\right)\right)^{\fS_q}\right)\ar[d]_-{\kappa^{n,q}_{J\left(B^{\fS_p}\right)}}\ar@/_/[dr]^-{J^n\left(\left(\kappa^{1,p}_B\right)^{\fS_q}\right)} & & J^{n+1}\left(B^{\fS_{p+q}}\right)\ar[d]^-{J^n\left(\kappa^{1,p+q}_B\right)} \\
	\left(J^{n+1}\left(B^{\fS_p}\right)\right)^{\fS_q}\ar@/_1pc/[dr]_-{\left(J^n\left(\kappa^{1,p}_B\right)\right)^{\fS_q}} & J^n\left(\left((JB)^{\fS_p}\right)^{\fS_q}\right)\ar[d]_-{\kappa^{n,q}_{(JB)^{\fS_p}}}\ar[r]^-{J^n\left(\mu^{p,q}_{JB}\right)} & J^n\left((JB)^{\fS_{p+q}}\right)\ar[dd]^-{\kappa^{n,p+q}_{JB}} \\
	& \left(J^n\left((JB)^{\fS_p}\right)\right)^{\fS_q}\ar[d]_-{\left(\kappa^{n,p}_{JB}\right)^{\fS_q}} & \\
	& \left((J^{n+1}B)^{\fS_p}\right)^{\fS_q}\ar[r]^-{\mu^{p,q}_{J^{n+1}B}} & (J^{n+1}B)^{\fS_{p+q}} \\}\]
Moreover, the following equalities hold by Lemma \ref{lem:kappapq}, proving the result:
\[\kappa^{n,p+q}_{JB}\circ J^n\left(\kappa^{1,p+q}_B\right)=\kappa^{n+1,p+q}_B\]
\[\left(\kappa^{n,p}_{JB}\right)^{\fS_q}\circ \left(J^n\left(\kappa^{1,p}_B\right)\right)^{\fS_q}=\left(\kappa^{n+1,p}_B\right)^{\fS_q}\]
\[\kappa^{n,q}_{J\left(B^{\fS_p}\right)}\circ J^n\left(\kappa^{1,q}_{B^{\fS_p}}\right)=\kappa^{n+1,q}_{B^{\fS_p}}\qedhere\]
\end{proof}

\subsection[Extending constructions to the ind-homotopy category]{Extending constructions to the ind-homotopy category}\label{subsec:extending}

Let $I$ be a filtering poset and let $F:\Algl\to\Algl^I$ be a functor. Then $F$ induces $F^\ind:\Algl^\ind\to(\Algl^I)^\ind$;  composing this with the functor $c_I$ of Lemma \ref{lem:indI} we get a functor $\Algl^\ind\to \Algl^\ind$ that we still denote $F$. This happens, for example, in the following cases:
\begin{enumerate}[label=(\roman*)]
	\item\label{item:functora} $I=\{*\}$ and $F=J:\Algl\to \Algl$;
	\item $I=\{*\}$ and $F=(?)^X:\Algl\to\Algl$ for any $X\in \S$;
	\item $I=\Zo$ and $F=(?)^{(K,L)}_\bul:\Algl\to \Algl^\Zo$ for any simplicial pair $(K,L)$;
	\item\label{item:functorz} $I$ any poset and $F=?\otimes C_\bul:\Algl\to \Algl^I$, with $C_\bul\in (\Alg_\Z)^I$.
\end{enumerate}
In all these examples, $F$ has the aditional property of being \emph{homotopy invariant}: if $f$ and $g$ are two homotopic $\ell$-algebra homomorphisms then, for all $i\in I$, $F(f)_i$ and $F(g)_i$ are homotopic. Because of this, $F$ induces a functor $F:[\Algl]\to [\Algl]^I$ and thus a functor $F:[\Algl]^\ind\to [\Algl]^\ind$; here we are using Lemma \ref{lem:indI} once more. It is easily seen that the following diagram commutes:
\[\xymatrix{\Algl^\ind \ar[r]^-{F}\ar[d] & \Algl^\ind\ar[d] \\
	[\Algl]^\ind\ar[r]^-{F} & [\Algl]^\ind \\}\]
Thus, we can apply functors \ref{item:functora}-\ref{item:functorz} to objects and morphisms in $[\Algl]^\ind$ and, if we start at $\Algl^\ind$, we may apply these functors and take homotopy classes in any order we like.

By the previous discussion, we can regard $((?)^{\fS_m}_\bul)^{\fS_n}_\bul$ and $(?)^{\fS_{m+n}}_\bul$ as endofunctors of $[\Algl]^\ind$; we would like to consider $\mu^{m,n}_?$ as a natural transformation between them. We now proceed to explain how certain morphisms from $F:\Algl\to\Algl^I$ to $G:\Algl\to\Algl^J$ induce natural transformations between the associated endofunctors of $[\Algl]^\ind$ ---here, $I$ and $J$ may be different filtering posets, and $F$ and $G$ are homotopy invariant functors.

Let $F:\Algl\to\Algl^I$ and $G:\Algl\to\Algl^J$ be two homotopy invariant functors. Consider a pair $(\nu,\theta)$ where $\theta:I\to J$ is a functor and $\nu:F\to \theta^*G$ is a natural transformation of functors $\Algl\to \Algl^I$. This means that:
\begin{enumerate}[label=(\alph*)]
	\item For each $A\in\Algl$ we have $\nu_A:F(A)\to G(A)\circ\theta\in\Algl^I$;
	\item For each morphism $f:A\to A'$ in $\Algl$, the following diagram in $\Algl^I$ commutes:
	\[\xymatrix{F(A)\ar[r]^-{\nu_A}\ar[d]_-{F(f)} & G(A)\circ\theta\ar[d]^-{(\theta^*G)(f)} \\
		F(A')\ar[r]^-{\nu_{A'}} & G(A')\circ\theta}\]
\end{enumerate}
Let $(C,K)\in[\Algl]^\ind$. We will define $\nu_{C_\bul}\in[F(C_\bul)_\bul, G(C_\bul)_\bul]$. For each pair $(i,k)\in I\times K$, let $\left(\nu_{C_\bul}\right)_{(i,k)}$ be the class of the morphism $\left(\nu_{C_k}\right)_i:F(C_k)_i\to G(C_k)_{\theta(i)}$ in $\left[F(C_k)_i,G(C_\bul)_\bul\right]$.
It is easily verified that the $\left(\nu_{C_\bul}\right)_{(i,k)}$ are compatible and assemble into a morphism:
\[\nu_{C_\bul}=\left\{\left(\nu_{C_\bul}\right)_{(i,k)}\right\}\in \lim_{(i,k)}\left[F(C_k)_i,G(C_\bul)_\bul\right]=\left[F(C_\bul)_\bul, G(C_\bul)_\bul\right]\]

\begin{lem}\label{lem:extendingnat}
	The construction above determines a natural transformation $\nu:F\to G$ of functors $[\Algl]^\ind\to[\Algl]^\ind$. That is, for every morphism $f\in [C_\bul, D_\bul]$, the following diagram in $[\Algl]^\ind$ commutes:
	\[\xymatrix{F(C_\bul)_\bul\ar[d]_-{F(f)}\ar[r]^-{\nu_{C_\bul}} & G(C_\bul)_\bul\ar[d]^-{G(f)} \\
		F(D_\bul)_\bul\ar[r]^-{\nu_{D_\bul}} & G(D_\bul)_\bul}\]
\end{lem}
\begin{proof}
	It is a straightforward verification.
\end{proof}

\begin{exa}
	Regard $\kappa^{n,m}_?:J^n((?)^{\fS_m}_\bul)\to (J^n(?))^{\fS_m}_\bul$ as a natural transformation between (homotopy invariant) functors $\Algl\to\Algl^\Zo$. By Lemma \ref{lem:extendingnat}, we can also regard $\kappa^{n,m}_?$ a natural transformation:
	\[\xymatrix{\kappa^{n,m}_?:J^n((?)^{\fS_m}_\bul)\ar[r] & (J^n(?))^{\fS_m}_\bul:[\Algl]^\ind\ar[r] & [\Algl]^\ind}\]
\end{exa}

\begin{exa}
	Consider the (homotopy invariant) functors $F:\Algl\to \Algl^{\Zo\times\Zo}$, $F(B)=(B^{\fS_m}_\bul)^{\fS_n}_\bul$, and $G:\Algl\to\Algl^\Zo$, $G(B)=B^{\fS_{m+n}}_\bul$. Define $\theta:\Zo\times\Zo\to\Zo$ by $\theta(r,s)=r+s$. Then $\theta$ is a functor and $\mu^{m,n}_?:F\to \theta^*G$ is a natural transformation. By Lemma \ref{lem:extendingnat}, the pair $(\mu^{m,n}_?,\theta)$ induces a natural transformation:
	\[\xymatrix{\mu^{m,n}_?:((?)^{\fS_m}_\bul)^{\fS_n}_\bul\ar[r] & (?)^{\fS_{m+n}}_\bul:[\Algl]^\ind\ar[r] & [\Algl]^\ind}\]
\end{exa}

\begin{rem}\label{rem:extendingtoind}
	We have just seen we can regard $J$ and $(?)^{\fS_n}_\bul$ as endofunctors of $[\Algl]^\ind$. Moreover, we can consider $\kappa^{n,m}_?$ and $\mu^{n,m}_?$ as natural transformations between endofunctors of  $[\Algl]^\ind$. In the sequel, we shall do this without further mention.
\end{rem}

\subsection{Group structure}

Let $n\geq 1$. For $A,B\in\Algl$, the set $[A,B^{\fS_n}_\bul]$ has a natural group structure, that is abelian if $n\geq 2$; see Remark \ref{rem:groupstructure}. We claim that this assertion remains true if we replace $A$ and $B$ by arbitrary ind-objects in $[\Algl]$. Indeed, for $(A,I),(B,J)\in[\Algl]^\ind$, we have:
\begin{equation}\label{eq:groupstr}[A_\bul, (B_\bul)^{\fS_n}_\bul]\cong \lim_i\colim_j\, [A_i,(B_j)^{\fS_n}_\bul]\end{equation}
Since limits and filtered colimits of groups are computed respectively as limits and filtered colimits of sets, the right hand side of \eqref{eq:groupstr} has a natural group structure, that is abelian if $n\geq 2$. This can be summarized as follows.

\begin{lem}\label{lem:groupstr}
	Let $B_\bul\in[\Algl]^\ind$ and let $n\geq 1$. Then $(B_\bul)^{\fS_n}_\bul$ is a group object in $[\Algl]^\ind$, which is abelian if $n\geq 2$. Moreover, a morphism $g\in[B_\bul, B'_\bul]$ induces a morphism of group objects $g_*\in[(B_\bul)^{\fS_n}_\bul, (B'_\bul)^{\fS_n}_\bul]$.
\end{lem}

We conclude the section with a proposition that relates the group structure in $[A_\bul, (B_\bul)^{\fS_n}_\bul]$ with the loop functors $J$ and $(?)^{\fS_m}$. We postpone its proof to the Appendix \ref{sec:app}.

\begin{prop}\label{prop:bundle}
Let $A,B\in\Algl$, let $K$ be a filtering poset, let $C_\bul\in(\Alg_\Z)^K$ and let $m,n\geq 1$. Then the following composite functions are group homomorphisms:
\begin{enumerate}[label=(\roman*)]
	\item\label{lem:Jkappa} $\xymatrix@C=3em{[A,B^{\fS_m}_\bul] \ar[r]^-{J^n} & [J^nA, J^n(B^{\fS_m}_\bul)]\ar[r]^-{\left(\kappa^{n,m}_{B}\right)_*} & [J^nA, (J^nB)^{\fS_m}_\bul]}$
	\item\label{lem:tensorgrouphomo} $\xymatrix@C=4em{[A,B^{\fS_m}_\bul]\ar[r]^-{?\otimes C_\bul} & [A\otimes C_\bul,B^{\fS_m}_\bul\otimes C_\bul]\cong[A\otimes C_\bul,(B\otimes C_\bul)^{\fS_m}_\bul]}$
	\item\label{lem:mugrouphomo} $\xymatrix@C=4em{[A,(B^{\fS_m}_\bul)^{\fS_n}_\bul]\ar[r]^-{\left(\mu^{m,n}_{B}\right)_*} & [A,B^{\fS_{m+n}}_\bul]}$
	\item\label{lem:mu2grouphomo} $\xymatrix@C=3em{[A,B^{\fS_m}_\bul]\ar[r]^-{(?)^{\fS_n}} & [A^{\fS_n}_\bul,(B^{\fS_m}_\bul)^{\fS_n}_\bul]\ar[r]^-{\left(\mu^{m,n}_{B}\right)_*} & [A^{\fS_n}_\bul,B^{\fS_{m+n}}_\bul]}$
\end{enumerate}
In \ref{lem:tensorgrouphomo}, the bijection on the right is induced by the obvious isomorphism of $K\times\Zo$-diagrams $B^{\fS_m}_\bul\otimes C_\bul\cong(B\otimes C_\bul)^{\fS_m}_\bul$.
\end{prop}

\section{The loop-stable homotopy category}\label{sec:defik}

Let $f:A\to B^{\fS_n}_r$ be an $\ell$-algebra homomorphism. By Proposition \ref{lem:classexists}, there exists a unique strong morphism of extensions $\scrU_A\to\scrP_{n,B}$ extending $f$:
	\[\xymatrix@C=3em@R=2em{\scrU_A\ar@{..>}[d]_-{\exists !} & JA\ar[r]\ar@{..>}[d]_-{\Lambda^n(f)} & TA\ar@{..>}[d]\ar[r] & A\ar[d]^-{f} \\
		\scrP_{n,B} & B^{\fS_{n+1}}_r\ar[r] & P(n,B)_r\ar[r] & B^{\fS_n}_r \\}\]
	We will write $\Lambda^n(f)$ for the classifying map of $f$ with respect to $\scrP_{n,B}$.

	\begin{rem}\label{rem:Lambda-J}
		We have $\Lambda^n(f)=\Lambda^n(\id_{B^{\fS_n}})\circ J(f)$. Indeed, this follows from the uniqueness statement in Proposition \ref{lem:classexists} and the fact that the following diagram exhibits a strong morphism of extensions $\scrU_A\to \scrP_{n,B}$ that extends $f$:
		\[\xymatrix@C=3em@R=2em{\scrU_A\ar[d] & JA\ar[d]_-{J(f)}\ar[r] & TA\ar[r]\ar[d]^-{T(f)} & A\ar[d]^-{f} \\
			\scrU_{B^{\fS_n}_r}\ar[d] & J(B^{\fS_n}_r)\ar[d]_-{\Lambda^n\left(\id_{B^{\fS_n}}\right)}\ar[r] & T(B^{\fS_n}_r)\ar[r]\ar[d] & B^{\fS_n}_r\ar[d]^-{\id} \\
			\scrP_{n,B} & B^{\fS_{n+1}}_r\ar[r] & P(n,B)_r\ar[r] & B^{\fS_n}_r \\}\]
	\end{rem}

	\begin{rem}\label{rem:Lambda-fS}
		We have $\Lambda^n(f)=\mu^{n,1}_B\circ f^{\fS_1}_0\circ \lambda_A$. Indeed, this follows from the uniqueness statement in Proposition \ref{lem:classexists} and the fact that the following diagram exhibits a strong morphism of extensions $\scrU_A\to \scrP_{n,B}$ that extends $f$:
		\[\xymatrix@C=3em@R=2em{\scrU_A\ar[d] & JA\ar[d]_-{\lambda_{A}}\ar[r] & TA\ar[r]\ar[d] & A\ar[d]^-{\id} \\
			\scrP_{0,A}\ar[d] & A^{\fS_1}_0\ar[d]_-{f^{\fS_1}_0}\ar[r] & P(0,A)_0\ar[r]\ar[d]^-{P(0,f)_0} & A\ar[d]^-{f} \\
			\scrP_{0,B^{\fS_n}_r}\ar[d]|-{\mbox{\scriptsize Example \ref{exa:muSm}}} & (B^{\fS_n}_r)^{\fS_1}_0\ar[d]_-{\mu^{n,1}}\ar[r] & P(0,B^{\fS_n}_r)_0\ar[r]\ar[d] & B^{\fS_n}_r\ar[d]^-{\id} \\
			\scrP_{n,B} & B^{\fS_{n+1}}_r\ar[r] & P(n,B)_r\ar[r] & B^{\fS_n}_r \\}\]
	\end{rem}

	If $f,g:A\to B^{\fS_n}_r$ are homotopic morphisms, then $\Lambda^n(f)$ and $\Lambda^n(g)$ are homotopic too. Thus, we can regard $\Lambda^n$ as a function $\Lambda^n_{A,B}:[A, B^{\fS_n}_r]\to [JA,B^{\fS_{n+1}}_r]$. For ind-algebras $(A,I)$ and $(B,J)$, define $\Lambda^{n}_{A_\bul,B_\bul}:[A_\bul,(B_\bul)_\bul^{\fS_n}]\to [J(A_\bul),(B_\bul)_\bul^{\fS_{n+1}}]$ as the function:
	\[\xymatrix@C=2em{\displaystyle\lim\colim \Lambda^n_{A_i,B_j}:\lim_i\colim_{(r,j)}\,[A_i, (B_j)^{\fS_n}_r]\ar[r] & \displaystyle\lim_i\colim_{(r,j)}\,[JA_i, (B_j)^{\fS_{n+1}}_r]}\]

	\begin{lem}\label{lem:Lambda}
		Let $A_\bul, B_\bul\in\Algl^\ind$ and let $n\geq 1$. Then the functions
		\[\xymatrix{\Lambda^n_{A_\bul, B_\bul}:[A_\bul,(B_\bul)_\bul^{\fS_n}]\ar[r] & [J(A_\bul),(B_\bul)_\bul^{\fS_{n+1}}]}\]
		are group homomorphisms.
	\end{lem}
	\begin{proof}
		It suffices to prove the case $A_\bul=A\in\Algl$ and $B_\bul=B\in\Algl$. Consider the following strong morphism of extensions $\scrU_{B^{\fS_n}_r}\to \scrP_{n,B}$ extending the identity:
		\[\xymatrix@R=2em{\scrU_{B^{\fS_n}_r}\ar[d] & J(B^{\fS_n}_r)\ar[d]_-{\kappa^{1,n}}\ar[r] & T(B^{\fS_n}_r)\ar[d]\ar[r] & B^{\fS_n}_r\ar[d]^-{\id} \\
			(\scrU_B)^{\fS_n}\ar[d] & (JB)^{\fS_n}_r\ar[d]_-{(\lambda_B)^{\fS_n}}\ar[r] & (TB)^{\fS_n}_r\ar[d]\ar[r] & B^{\fS_n}_r\ar[d]^-{\id} \\
			(\scrP_{0,B})^{\fS_n}\ar[d] & (B^{\fS_1}_0)^{\fS_n}_r\ar[d]_-{\mu^{1,n}}\ar[r] & (PB)^{\fS_n}_r\ar[d]\ar[r] & B^{\fS_n}_r\ar[d]^-{\id} \\
			\widetilde{\scrP}_{n,B}\ar[d]|-{\mbox{\scriptsize Example \ref{exa:pathmonho}}} & B^{\fS_{1+n}}_r\ar[d]_-{c^*}^-{\cong}\ar[r] & \widetilde{P}(n,B)_r\ar[d]_-{c^*}^-{\cong}\ar[r] & B^{\fS_n}_r\ar[d]^-{\id} \\
			\scrP_{n,B} & B^{\fS_{n+1}}_r\ar[r] & P(n,B)_r\ar[r] & B^{\fS_n}_r }\]
		By the uniqueness statement in Proposition \ref{lem:classexists}, we have:
		\begin{equation}\label{eq:Lambdaid}\Lambda^n(\id_{B^{\fS_n}})=c^*\circ\mu^{1,n}_B\circ (\lambda_B)^{\fS_n}\circ \kappa^{1,n}_B\end{equation}
		Then, by Remark \ref{rem:Lambda-J}, $\Lambda^n_{A,B}$ equals the composite:
		\[\xymatrix@C=2.5em{[A,B^{\fS_n}_\bul]\ar[r]^-{\kappa^{1,n}_B\circ J(?)} & [JA, (JB)^{\fS_n}_\bul]\ar[r]^-{(\lambda_B)_*} & [JA, (B^{\fS_1}_0)^{\fS_n}_\bul]\ar[r]^-{(\mu^{1,n}_B)_*} & [JA, B^{\fS_{1+n}}_\bul]\ar[r]^-{c^*} & [JA, B^{\fS_{n+1}}_\bul] }\]
		The result now follows from Proposition \ref{prop:bundle} \ref{lem:Jkappa} and \ref{lem:mugrouphomo} and Example \ref{exa:cmn}.
	\end{proof}

	\begin{defi}[cf. \cite{ralf}*{Section 6.3}]\label{defi:kkc} We proceed to define a category $\fk$, that we will call the \emph{loop-stable homotopy category}. The objects of $\fk$ are the pairs $(A,m)$ where $A$ is an object of $\Algl$ and $m\in\Z$. For two objects $(A,m)$ and $(B,n)$,  put:
		\[\fk\left((A,m),(B.n)\right):=\colim_v\,[J^{m+v}A,B^{\fS_{n+v}}_\bul]\]
		Here, the colimit is taken over the morphisms $\Lambda^{n+v}$ of Lemma \ref{lem:Lambda} and $v$ runs over the integers such that both $m+v\geq 0$ and $n+v\geq 0$. We often write $\langle f\rangle$ for the element of $\fk\left((A,m),(B.n)\right)$ represented by $f\in[J^{m+v}A,B^{\fS_{n+v}}_\bul]$. The composition in $\fk$ is defined as follows. Represent elements of $\fk\left((A,m),(B,n)\right)$ and $\fk\left((B,n),(C,k)\right)$ by $f\in[J^{m+v}A,B^{\fS_{n+v}}_\bul]$ and $g\in[J^{n+w}B,C^{\fS_{k+w}}_\bul]$ respectively. To simplify notation, write:
		\[N_1:=m+v,\hspace{1em}N_2:=n+v,\hspace{1em}N_3:=n+w\hspace{1em}\mbox{and}\hspace{1em}N_4:=k+w.\]
		Let $g\star f$ be the following composite in $[\Algl]^\ind$:
		\[\xymatrix@C=3.7em{J^{N_1+N_3}A \ar[r]^-{J^{N_3}(f)} &
			J^{N_3}(B^{\fS_{N_2}}_\bul)\ar[r]^-{(-1)^{N_2N_3}\kappa^{N_3,N_2}} &			(J^{N_3}B)^{\fS_{N_2}}_\bul\ar[r]^-{g^{\fS_{N_2}}} & (C^{\fS_{N_4}}_\bul)^{\fS_{N_2}}_\bul\ar[r]^-{\mu^{N_4,N_2}} & C_\bul^{\fS_{N_4+N_2}} \\}\]
		Then define $\langle g\rangle\circ\langle f\rangle:=\langle g\star f\rangle\in \fk((A,m),(C,k))$. We will show in Lemma \ref{lem:compwelldef} that $\langle g\star f\rangle$ does not depend upon the choice of the representatives $f$ and $g$. Then, in Theorem \ref{thm:compowd}, we will prove that the composition just defined makes $\fk$ into a category.
	\end{defi}

\section{Well-definedness of the composition}\label{sec:well}

We closely follow \cite{ralf}*{Section 6.3}, making appropiate changes to translate the proof into the algebraic setting. The following two lemmas are straightforward verifications.

\begin{lem}\label{lem:pegote}
	Let $A_\bul, B_\bul\in[\Algl]^\ind$ and let $g\in[A_\bul, (B_\bul)^{\fS_n}_\bul]$.
	\begin{enumerate}[label=(\roman*)]
		\item\label{lem:Lambda-J} If $f\in[A'_\bul, A_\bul]$, then $\Lambda^n(g\circ f)=\Lambda^n(g)\circ J(f)\in[J(A'_\bul), (B_\bul)^{\fS_{n+1}}_\bul]$.
		\item\label{lem:Lambda-fS} If $h\in[B_\bul, B'_\bul]$, then $\Lambda^n(h^{\fS_n}\circ g)=h^{\fS_{n+1}}\circ \Lambda^n(g)\in[J(A_\bul), (B'_\bul)^{\fS_{n+1}}_\bul]$.
	\end{enumerate}
\end{lem}

\begin{lem}\label{lem:Lambda-mu}
	Let $A_\bul, B_\bul\in[\Algl]^\ind$ and let $f\in[A_\bul, ((B_\bul)^{\fS_n}_\bul)^{\fS_m}_\bul]$. Then:
	\[\Lambda^{n+m}\left(\mu^{n,m}_{B_\bul}\circ f\right)=\mu^{n,m+1}_{B_\bul}\circ \Lambda^m(f)\in[J(A_\bul), (B_\bul)^{\fS_{n+m+1}}_\bul]\]
\end{lem}

\begin{lem}[cf. \cite{ralf}*{Lemma 6.30}]\label{lem:clascon}
	Let $B\in\Algl$. Then the following diagram in $[\Algl]^\ind$ commutes:
	\begin{equation}\label{eq:clascon}\begin{gathered}\xymatrix@C=3em{J^2B\ar[r]^-{J(\lambda_B)}\ar@/_/[dr]_-{(\lambda_{JB})^{-1}} & J(B^{\fS_1}_\bul)\ar[d]^-{\kappa^{1,1}_B} \\
		& (JB)^{\fS_1}_\bul \\}\end{gathered}\end{equation}
	Here, $(\lambda_{JB})^{-1}$ is the inverse of $\lambda_{JB}$ in the group $[J^2B,(JB)^{\fS_1}_\bul]$.
\end{lem}
\begin{proof}
	Recall from Example \ref{exa:subdi1} that, for $A\in\Algl$, the extension $\scrP_{0,A}$ is isomorphic to:
	\[\begin{array}{cr}\xymatrix@C=3em{(t^2-t)A[t]\ar[r]^-{\inc} & (t-1)A[t]\ar[r]^-{\ev_0} & A} & \text{($r=0$)}\\
	\xymatrix@C=3em{tA[t]\tensor[_{\ev_1}]{\times}{_{\ev_1}} tA[t]\ar[r]^-{\inc} & A[t]\tensor[_{\ev_1}]{\times}{_{\ev_1}}tA[t]\ar[r]^-{\ev_0\circ\pr_1} & A} & \text{($r=1$)}
	\end{array}\]
	In the rest of the proof we will make these identifications without further mention.

Define
\[I:=\ker\left(\xymatrix{t(TB)[t]\ar[r]^-{\ev_1} & TB\ar[r]^-{\eta_B} & B\\}\right)\]
and let $s:B\to t(TB)[t]$ be given by $s(b)=\sigma_B(b)t$ ---recall that $\sigma_B$ is the section of the universal extension $\scrU_B$.	We have an extension:
	\[(\scrE,s):\xymatrix@C=4em{I\ar[r]^-{\inc} & t(TB)[t]\ar[r]^-{\eta_B\circ \ev_1} & B}\]
Now put
\[E:=\left\{(p,q)\in t(TB)[t]\times t(JB)[t]:p(1)=q(1)\right\}\]
and let $s':I\to E$ be given  by $s'(p)=(p,p(1)t)$ ---here we use that $\ev_1:I\to TB$ factors through $JB$. It is easily seen that we have an extension:
\[\xymatrix@C=4em@R=-0.3em{(\scrE',s'): (t^2-t)(JB)[t]\ar[r]^-{(0,\inc)} & E\ar[r]^-{\pr_1} & I\\}\]

Let $\chi:JB\to I$ be the classifying map of $(\scrE,s)$. The following diagram exhibits a strong morphism of extensions $\scrU_B\to\scrU_B$ extending $\id_B$:
	\[\xymatrix@C=4em@R=2em{\scrU_B\ar[d] & JB\ar[r]\ar[d]_-{\chi} & TB\ar[r]\ar[d] & B\ar[d]^-{\id} \\
		(\scrE,s)\ar[d] & I\ar[r]^-{\inc}\ar[d]_-{\ev_1} & t(TB)[t]\ar[r]^-{\eta_B\circ \ev_1}\ar[d]^-{\ev_1} & B\ar[d]^-{\id} \\
		\scrU_B & JB\ar[r]^-{\inc} & TB\ar[r]^-{\eta_B} & B \\}\]
	It follows that $\ev_1\circ\chi=\id_{JB}$. Now let $\omega$ be the automorphism of $(JB)[t]$ given by $\omega(t)=1-t$ and consider the following strong morphism of extensions $\scrU_{JB}\to\scrP_{0,JB}$ extending $\id_{JB}$:
	\[\xymatrix@C=4em@R=2em{\scrU_{JB}\ar[d] & J^2B\ar[r]\ar[d] & T(JB)\ar[r]\ar[d] & JB\ar[d]^-{\chi}\ar@/^2pc/[dd]^-{\id} \\
		(\scrE',s')\ar[d] & (t^2-t)(JB)[t]\ar[r]^-{(0,\inc)}\ar[d]_-{\id} & E\ar[r]^-{\pr_1}\ar[d]^-{\pr_2} & I\ar[d]^-{\ev_1} \\
		\scrE''\ar[d] & (t^2-t)(JB)[t]\ar[r]^-{\inc}\ar[d]_-{\omega} & t(JB)[t]\ar[r]^-{\ev_1}\ar[d]^-{\omega} & JB\ar[d]^-{\id} \\
		\scrP_{0,JB} & (t^2-t)(JB)[t]\ar[r]^-{\inc} & (t-1)(JB)[t]\ar[r]^-{\ev_0} & JB \\}\]
	Since $\omega^{-1}=\omega$, we get that the classifying map of $\chi$ with respect to $(\scrE',s')$ equals the composite:
	\[\xymatrix@C=3em{J^2B\ar[r]^-{\lambda_{JB}} & (t^2-t)(JB)[t]\ar[r]^-{\omega} & (t^2-t)(JB)[t] }\]
Define $\theta:E\to (TB)^{\fS_1}_1=t(TB)[t]\tensor[_{\ev_1}]{\times}{_{\ev_1}}t(TB)[t]$ by $\theta(p,q)=(q,p)$ and consider the following morphisms of extensions:
	\begin{equation}\label{eq:post}\begin{gathered}\xymatrix@C=4em@R=2em{\scrU_{JB}\ar[d] & J^2B\ar[r]\ar[d]_-{\omega\circ\lambda_{JB}} & T(JB)\ar[r]\ar[d] & JB\ar[d]^-{\chi} \\
		\scrE'\ar[d] & (JB)^{\fS_1}_0\ar[r]^-{(0,\inc)}\ar[d]_-{(\inc,0)} & E\ar[r]^-{\pr_1}\ar[d]^-{\theta} & I\ar[d]^-{(0,\eta_B)} \\
		(\scrU_B)^{\fS_1} & (JB)^{\fS_1}_1\ar[r]^-{(\inc)^{\fS_1}} & (TB)^{\fS_1}_1\ar[r]^{(\eta_B)^{\fS_1}} & B^{\fS_1}_1 \\}\end{gathered}\end{equation}
	Note that the morphism $\scrE'\to (\scrU_B)^{\fS_1}$ is not strong. Put $\psi:=(0,\eta_B)\circ\chi:JB\to B^{\fS_1}_1$. By Proposition \ref{lem:classcommute} applied to \eqref{eq:post}, the following diagram commutes in $[\Algl]$:
	\[\xymatrix@C=2em@R=2em{J^2B\ar[d]_-{\id}\ar[rrr]^{J(\psi)} & & & J(B^{\fS_1}_1)\ar[d]^-{\kappa^{1,1}_B} \\
		J^2B\ar[r]^-{\lambda_{JB}} & (JB)^{\fS_1}_0\ar[r]^-{\omega} & (JB)^{\fS_1}_0\ar[r]^-{\gamma^*} & (JB)^{\fS_1}_1 \\}\]
	Here $\gamma^*=(\inc,0):(t^2-t)(JB)[t]\to t(JB)[t]\tensor[_{\ev_1}]{\times}{_{\ev_1}}t(JB)[t]$ is the morphism induced by the last vertex map; see Example \ref{exa:subdi1}. The proof will be finished if we show that $J(\psi)$ equals $J(\lambda_B)$ in $[\Algl]$. By Lemma \ref{lem:Jhtp} it suffices to show that $\psi$ equals $\lambda_B$ in $[\Algl]$. Define $\beta:t(TB)[t]\to (PB)_1$ by $\beta(p)=(\eta_B(p)(1-t),0)$. The following diagram exhibits a strong morphism of extensions $\scrU_B\to \scrP_{0,B}$ extending $\id_B$:
	\[\xymatrix@C=4em@R=2em{\scrU_B\ar[d] & JB\ar[d]_-{\chi}\ar[r] & TB\ar[d]\ar[r] & B\ar[d]^-{\id} \\
		(\scrE,s)\ar[d] & I\ar[d]_-{\beta}\ar[r]^-{\inc} & t(TB)[t]\ar[d]^-{\beta}\ar[r]^-{\eta_B\circ \ev_1} & B\ar[d]^-{\id} \\
		\scrP_{0,B} & B^{\fS_1}_1\ar[r]_-{\inc} & (PB)_1\ar[r]_-{\ev_0\circ\pr_1} & B}\]
	It follows that $\lambda_B$ equals the composite $JB\overset{\chi}\to I\overset{\beta}\to B^{\fS_1}_1$,	which is easily seen to be homotopic to $\psi$.
\end{proof}

\begin{lem}\label{lem:precompo2}
	Let $B\in\Algl$ and let $\varepsilon_n=(-1)^n$. Then the following diagram in $[\Algl]^\ind$ commutes:
	\[\xymatrix@C=3em{J^{n+1}B\ar[r]^-{J^n(\lambda_B)}\ar@/_/[dr]_-{(\lambda_{J^nB})^{\varepsilon_n}} & J^n(B^{\fS_1}_\bul)\ar[d]^-{\kappa^{n,1}_B} \\
		& (J^nB)^{\fS_1}_\bul \\}\]
\end{lem}
\begin{proof}
	We prove the result by induction on $n$. The case $n=1$ is Lemma \ref{lem:clascon}. Suppose that the result holds for $n\geq 1$. We have:
	\begin{align*}\kappa^{n+1,1}_B\circ J^{n+1}(\lambda_B)&=\kappa^{1,1}_{J^nB}\circ J(\kappa^{n,1}_B)\circ J^{n+1}(\lambda_B) && \text{(by Lemma \ref{lem:kappapq})}\\
	&=\kappa^{1,1}_{J^nB}\circ J(\kappa^{n,1}_B\circ J^n(\lambda_B)) \\
	&=\kappa^{1,1}_{J^nB}\circ J\left((\lambda_{J^nB})^{\varepsilon_n}\right) && \text{(by hypothesis)} \\
	&=\left[\kappa^{1,1}_{J^nB}\circ J(\lambda_{J^nB})\right]^{\varepsilon_n} && \text{(by Proposition \ref{prop:bundle} \ref{lem:Jkappa})}\\
	&=\left[(\lambda_{J^{n+1}B})^{-1}\right]^{\varepsilon_n}= (\lambda_{J^{n+1}B})^{\varepsilon_{n+1}} && \text{(by the case $n=1$)}\end{align*}
	Then the result holds for $n+1$.
\end{proof}

\begin{lem}[cf. \cite{ralf}*{Lemma 6.29}]\label{prop:Lambda1} Let $B\in\Algl$ and let $n\geq 0$. Then the following diagram in $[\Algl]^\ind$ commutes:
	\[\xymatrix@C=2em@R=2em{J(B^{\fS_n}_\bul)\ar[d]_-{(-1)^n\Lambda^n(\id_{B^{\fS_n}})}\ar[r]^-{\kappa^{1,n}_B} & (JB)^{\fS_n}_\bul \ar[d]^-{(\lambda_B)^{\fS_n}} \\
		B^{\fS_{n+1}}_\bul & (B^{\fS_1}_\bul)^{\fS_n}_\bul\ar[l]^-{\mu^{1,n}} \\}\]
\end{lem}
\begin{proof}
	We have to prove the equality of two morphisms in $[\Algl]^\ind$; since $[J(B^{\fS_n}_\bul),B^{\fS_{n+1}}_\bul]=\lim_r[J(B^{\fS_n}_r),B^{\fS_{n+1}}_\bul]$, it will be enough to show that both morphisms are equal when projected to $[J(B^{\fS_n}_r),B^{\fS_{n+1}}_\bul]$, for every $r$. The result now follows from \eqref{eq:Lambdaid}; the appearance of the sign $(-1)^n$ is explained in Example \ref{exa:cmn}.
\end{proof}

\begin{lem}\label{lem:compo2}
	Let $B\in\Algl$ and let $m,n\geq 0$. Then, we have:
	\[(-1)^n\Lambda^m(\kappa^{n,m}_B)=\kappa^{n,m+1}_B\circ J^n\Lambda^m(\id_{B^{\fS_m}})\in [J^{n+1}(B^{\fS_m}_\bul),(J^nB)^{\fS_{m+1}}_\bul]\]
\end{lem}
\begin{proof}
On one hand, the diagram below commutes in $[\Algl]^\ind$ by Lemmas \ref{lem:penta} and \ref{lem:precompo2}:
	\[\xymatrix@C=5em{J^{n+1}(B^{\fS_m}_\bul)\ar[r]^-{J^n\left(\lambda_{B^{\fS_m}}\right)}\ar@/_1pc/[dr]_-{\left(\lambda_{J^n(B^{\fS_m})}\right)^{\varepsilon_n}} & J^n\left((B^{\fS_m}_\bul\right)^{\fS_1}_\bul)\ar[d]^-{\kappa^{n,1}_{B^{\fS_m}}}\ar[r]^-{J^n\left(\mu^{m,1}_B\right)} & J^n(B^{\fS_{m+1}}_\bul)\ar[dd]^-{\kappa^{n,m+1}_B} \\
		& (J^n(B^{\fS_m}_\bul))^{\fS_1}_\bul\ar[d]_-{(\kappa^{n,m}_B)^{\fS_1}} &  \\
		& ((J^nB)^{\fS_m}_\bul)^{\fS_1}_\bul\ar[r]^-{\mu^{m,1}_{J^nB}} & (J^nB)^{\fS_{m+1}}_\bul}\]
	On the other hand, by Remark \ref{rem:Lambda-fS}, we have:
	\[\kappa^{n,m+1}_B\circ J^n(\mu^{m,1}_B)\circ J^n(\lambda_{B^{\fS_m}})=\kappa^{n,m+1}_B\circ J^n\Lambda^m(\id_{B^{\fS_m}})\]
	\[\mu^{m,1}_{J^nB}\circ (\kappa^{n,m}_B)^{\fS_1}\circ (\lambda_{J^n(B^{\fS_m})})^{\varepsilon_n}=(-1)^n\Lambda^m(\kappa^{n,m}_B)\]
	Note that we use Proposition \ref{prop:bundle} \ref{lem:mugrouphomo} to handle the sign $(-1)^n$ in the latter equation.
\end{proof}

\begin{lem}[cf. \cite{ralf}*{Lemma 6.32}]\label{lem:compwelldef}
	Let $f\in [J^{N_1}A,B^{\fS_{N_2}}_\bul]$ and $g\in[J^{N_3}B,C^{\fS_{N_4}}_\bul]$. Then:
	\[\Lambda^{N_4}(g)\star f=\Lambda^{N_2+N_4}(g\star f)=g\star\Lambda^{N_2}(f)\in[J^{N_1+N_3+1}A,C^{\fS_{N_2+N_4+1}}_\bul]\]
\end{lem}
\begin{proof}
	First, we have:
	\begin{align*}(-1)^{N_2N_3}\Lambda^{N_2+N_4}(g\star f)&=\Lambda^{N_2+N_4}\left(\mu^{N_4,N_2}_C\circ g^{\fS_{N_2}}\circ\kappa^{N_3,N_2}_B\circ J^{N_3}(f)\right)\\
	&=\Lambda^{N_2+N_4}\left(\mu^{N_4,N_2}_C\circ g^{\fS_{N_2}}\circ\kappa^{N_3,N_2}_B\right)\circ J^{N_3+1}(f) \\
		&=\mu^{N_4,N_2+1}_C\circ \Lambda^{N_2}\left(g^{\fS_{N_2}}\circ\kappa^{N_3,N_2}_B\right)\circ J^{N_3+1}(f)\\
		&=\mu^{N_4,N_2+1}_C\circ g^{\fS_{N_2+1}}\circ\Lambda^{N_2}\left(\kappa^{N_3,N_2}_B\right)\circ J^{N_3+1}(f)
	\end{align*}
	Here the equalities follow from the definition of $g\star f$, Lemma \ref{lem:pegote} \ref{lem:Lambda-J}, Lemma \ref{lem:Lambda-mu} and Lemma \ref{lem:pegote} \ref{lem:Lambda-fS} ---in that order. We are also using Proposition \ref{prop:bundle} \ref{lem:mu2grouphomo} and Lemma \ref{lem:Lambda} to handle the sign $(-1)^{N_2N_3}$.

	Secondly, we have:
	\begin{align*}(-&1)^{N_2N_3}\Lambda^{N_4}(g)\star f=\\
	&=(-1)^{N_2}\mu_C^{N_4+1,N_2}\circ \left(\mu_C^{N_4,1}\circ g^{\fS_1}\circ \lambda_{J^{N_3}B}\right)^{\fS_{N_2}} \circ \kappa^{N_3+1,N_2}_B\circ J^{N_3+1}(f)\\
	&=(-1)^{N_2}\mu_C^{N_4+1,N_2}\circ \left(\mu_C^{N_4,1}\right)^{\fS_{N_2}}\circ \left(g^{\fS_1}\right)^{\fS_{N_2}}\circ \left(\lambda_{J^{N_3}B}\right)^{\fS_{N_2}} \circ \kappa^{N_3+1,N_2}_B\circ J^{N_3+1}(f)\\
	&=(-1)^{N_2}\mu_C^{N_4,1+N_2}\circ \mu_{C^{\fS_{N_4}}}^{1,N_2}\circ \left(g^{\fS_1}\right)^{\fS_{N_2}}\circ \left(\lambda_{J^{N_3}B}\right)^{\fS_{N_2}} \circ \kappa^{N_3+1,N_2}_B\circ J^{N_3+1}(f)\\
	&=(-1)^{N_2}\mu^{N_4,1+N_2}_C\circ g^{\fS_{1+N_2}}\circ \mu^{1,N_2}_{J^{N_3}B}\circ \left(\lambda_{J^{N_3}B}\right)^{\fS_{N_2}}\circ \kappa^{N_3+1,N_2}_B\circ J^{N_3+1}(f)\end{align*}
	Here the equalities follow from the definition of $\star$ and Remark \ref{rem:Lambda-fS}, the functoriality of $(?)^{\fS_{N_2}}$, the associativity of $\mu$ and the naturality of $\mu$ ---in that order.

	Finally, we have:
	\begin{align*}(-&1)^{N_2N_3}g\star\Lambda^{N_2}(f)=\\
	&=(-1)^{N_3}\mu^{N_4,N_2+1}_C\circ g^{\fS_{N_2+1}}\circ\kappa^{N_3,N_2+1}_B\circ J^{N_3}\Lambda^{N_2}(f) \\
	&=(-1)^{N_3}\mu^{N_4,N_2+1}_C\circ g^{\fS_{N_2+1}}\circ\kappa^{N_3,N_2+1}_B\circ J^{N_3}\left(\Lambda^{N_2}(\id_{B^{\fS_{N_2}}})\circ J(f)\right) \\
	&=(-1)^{N_3}\mu^{N_4,N_2+1}_C\circ g^{\fS_{N_2+1}}\circ\kappa^{N_3,N_2+1}_B\circ J^{N_3}\Lambda^{N_2}(\id_{B^{\fS_{N_2}}})\circ J^{N_3+1}(f)\end{align*}
	Here the equalities follow from the definition of $\star$, Remark \ref{rem:Lambda-J} and the functoriality of $J^{N_3}$ ---in that order.

	Thus, $\Lambda^{N_2+N_4}(g\star f)=g\star \Lambda^{N_2}(f)$ by Lemma \ref{lem:compo2} and to prove that $\Lambda^{N_2+N_4}(g\star f)=\Lambda^{N_4}(g)\star f$ it is enough to show that:
	\[\Lambda^{N_2}\left(\kappa^{N_3,N_2}_B\right)=(-1)^{N_2}\mu^{1,N_2}_{J^{N_3}B}\circ \left(\lambda_{J^{N_3}B}\right)^{\fS_{N_2}}\circ \kappa^{N_3+1,N_2}_B\]
	We have:
	\begin{align*}
	\Lambda^{N_2}\left(\kappa^{N_3,N_2}_B\right)&=\Lambda^{N_2}\left(1_{(J^{N_3}B)^{\fS_{N_2}}}\circ\kappa^{N_3,N_2}_B\right)\\
	&=\Lambda^{N_2}\left(1_{(J^{N_3}B)^{\fS_{N_2}}}\right)\circ J\left(\kappa^{N_3,N_2}_B\right) \\
	&=(-1)^{N_2}\mu^{1,N_2}_{J^{N_3}B}\circ \left(\lambda_{J^{N_3}B}\right)^{\fS_{N_2}}\circ \kappa^{1,N_2}_{J^{N_3}B}\circ J\left(\kappa^{N_3,N_2}_B\right) \\
	&=(-1)^{N_2}\mu^{1,N_2}_{J^{N_3}B}\circ \left(\lambda_{J^{N_3}B}\right)^{\fS_{N_2}}\circ \kappa^{N_3+1,N_2}_B
	\end{align*}
	The first equality is trivial and the others follow from Lemma \ref{lem:pegote} \ref{lem:Lambda-J}, Lemma \ref{prop:Lambda1} and Lemma \ref{lem:kappapq} ---in that order.
\end{proof}

\begin{thm}\label{thm:compowd}
	The composition described in Definition \ref{defi:kkc} is well-defined and makes $\fk$ into a category.
\end{thm}
\begin{proof}
	The composition is well-defined by Lemma \ref{lem:compwelldef}. The associativity is a straightforward but lengthy verification.
\end{proof}

\begin{exa}\label{exa:j}
	There is a functor $j:\Algl\to\fk$ defined by $A\mapsto(A,0)$ on objects, that sends a morphism $f:A\to B$ to its class in $\fk((A,0),(B,0))$. This functor clearly factors through $\Algl\to[\Algl]$. We often write $A$ and $f$ instead of $j(A)$ and $j(f)$.
\end{exa}

\section{Additivity}\label{sec:add}

The hom-sets in $\fk$ are abelian groups; indeed, they are defined as the filtered colimit of a diagram of abelian groups. Next we show that the composition is bilinear.

\begin{lem}\label{lem:bilinearcomp}
	The composition in $\fk$ is bilinear.
\end{lem}
\begin{proof}
	Let $g\in[J^{n+w}B,C^{\fS_{k+w}}_\bul]$ represent an element $\beta\in\fk((B,n),(C,k))$. Let us show that $\beta_*:\fk((A,m),(B,n))\to \fk((A,m),(C,k))$ is a group homomorphism. Represent elements $\alpha,\alpha'\in\fk((A,m),(B,n))$ by $f,f'\in[J^{m+v}A,B^{\fS_{n+v}}_\bul]$ ---we may assume that $n+v\geq 2$ by choosing $v$ large enough. To alleviate notation, write
	\[N_1:=m+v,\hspace{1em}N_2:=n+v,\hspace{1em}N_3:=n+w\hspace{1em}\mbox{and}\hspace{1em}N_4:=k+w.\]
	By definition of the composition in $\fk$, the following diagram of sets commutes:
	\[\xymatrix@R=2em{[J^{N_1}A,B^{\fS_{N_2}}_\bul]\ar[r]^-{g\star ?}\ar[d] & [J^{N_1+N_3}A,C^{\fS_{N_2+N_4}}_\bul]\ar[d] \\
		\fk((A,m),(B,n))\ar[r]^-{\beta_*} & \fk((A,m),(C,k)) \\}\]
	Here, the vertical arrows are structural morphisms into the colimits ---hence they are group homomorphisms. Since $\beta_*(\alpha+\alpha')=\langle g\star(f+f')\rangle$, to prove that $\beta_*(\alpha+\alpha')=\beta_*(\alpha)+\beta_*(\alpha')$ it suffices to show that $g\star ?$ is a group homomorphism. But $g\star ?$ can be factored as the following composite of group homomorphisms, as we proceed to explain:
\[\xymatrix@C=1.4em{[J^{N_1}A,B^{\fS_{N_2}}_\bul]\ar[r] & [J^{N_1+N_3}A,(J^{N_3}B)^{\fS_{N_2}}_\bul]\ar[r] &
[J^{N_1+N_3}A,(C^{\fS_{N_4}}_\bul)^{\fS_{N_2}}_\bul]\ar[r] &
[J^{N_1+N_3}A,C^{\fS_{N_2+N_4}}_\bul]}\]
The morphism on the left is the one in Proposition \ref{prop:bundle} \ref{lem:Jkappa}. The morphism in the middle is $(-1)^{N_2N_3}g_*$. The morphism on the right is the one in Proposition \ref{prop:bundle} \ref{lem:mugrouphomo}.

Now let $f\in[J^{m+v}A,B^{\fS_{n+v}}_\bul]$ represent an element $\alpha\in\fk((A,m),(B,n))$. Let us show that $\alpha^*:\fk((B,n),(C,k))\to \fk((A,m),(C,k))$ is a group homomorphism. Represent elements $\beta,\beta'\in\fk((B,n),(C,k))$ by $g,g'\in[J^{n+w}B,C^{\fS_{k+w}}_\bul]$ ---we may assume that $k+w\geq 2$ by choosing $w$ large enough. As before, write
\[N_1:=m+v,\hspace{1em}N_2:=n+v,\hspace{1em}N_3:=n+w\hspace{1em}\mbox{and}\hspace{1em}N_4:=k+w.\]
By definition of the composition in $\fk$, the following diagram of sets commutes:
\[\xymatrix@R=2em{[J^{N_3}B,C^{\fS_{N_4}}_\bul]\ar[r]^-{?\star f}\ar[d] & [J^{N_1+N_3}A,C^{\fS_{N_2+N_4}}_\bul]\ar[d] \\
		\fk((B,n),(C,k))\ar[r]^-{\alpha^*} & \fk((A,m),(C,k))}\]
Since $\alpha^*(\beta+\beta')=\langle (g+g')\star f\rangle$, to prove that $\alpha^*(\beta+\beta')=\alpha^*(\beta)+\alpha^*(\beta')$ it suffices to show that $?\star f$ is a group homomorphism. But $?\star f$ and the following composite of group homomorphisms differ only by the sign $(-1)^{N_2N_3}$:
\[\xymatrix@C=6em{[J^{N_3}B,C^{\fS_{N_4}}_\bul]\ar[r]^-{\mbox{\scriptsize Proposition \ref{prop:bundle} \ref{lem:mu2grouphomo}}} & [(J^{N_3}B)^{\fS_{N_2}},C^{\fS_{N_4+N_2}}_\bul]\ar[r]^-{\left(\kappa^{N_3,N_2}_B\circ J^{N_3}f\right)^*} &
[J^{N_1+N_3}A,C^{\fS_{N_4+N_2}}_\bul]}\qedhere\]
\end{proof}

\begin{prop}
	The category $\fk$ is additive.
\end{prop}
\begin{proof}
	By Lemma \ref{lem:bilinearcomp}, it suffices to show that $\fk$ has finite products.

	Let $B,C\in\Algl$ and let $n\in\Z$. Let us first show that $(B\times C,n)$ is a product of $(B,n)$ and $(C,n)$ in $\fk$. For any $(A,m)\in\fk$, we have:
	\begin{align*}
	\fk((A,m),(B\times C,n))&=\colim_{v,r}\,[J^{m+v}A, (B\times C)^{\fS_{n+v}}_r] \\
	&\cong\colim_{v,r}\,[J^{m+v}A, B^{\fS_{n+v}}_r\times C^{\fS_{n+v}}_r] \\
	&\cong\colim_{v,r}\left\{[J^{m+v}A, B^{\fS_{n+v}}_r]\times[J^{m+v}A, C^{\fS_{n+v}}_r]\right\} \\
	&\cong\colim_{v,r}\,[J^{m+v}A, B^{\fS_{n+v}}_r]\times\colim_{v,r}\,[J^{m+v}A, C^{\fS_{n+v}}_r] \\
	&=\fk((A,m), (B,n))\times \fk((A,m), (C,n))
	\end{align*}
	Here we use that the functors $(?)^{\fS_N}_r:\Algl\to\Algl$ and $\Algl\to[\Algl]$ commute with finite products, and that filtered colimits of sets commute with finite products.

To prove that $\fk$ has finite products, we reduce to the special case above. We will show in Lemma \ref{lem:lambdaidinverses1} that, for any $(B,n)\in\fk$ and any $p\geq 1$, we have an isomorphism $(B,n)\cong (J^pB, n-p)$. Using this, any pair of objects of $\fk$ can be replaced by a new pair of objects ---each of them isomorphic to one of the original ones--- with equal second coordinate. The proof of Lemma \ref{lem:lambdaidinverses1} relies only on Lemma \ref{lem:clascon} and the definition of $\star$.
\end{proof}

\section{Excision}\label{sec:exci}

In this section we closely follow \cite{cortho}*{Section 6.3}. Let $f:A\to B$ be a morphism in $\Algl$. The \emph{mapping path} $(P_f)_\bul$ is the $\Zo$-diagram in $\Algl$ defined by the pullbacks:
\begin{equation}\label{eq:mappath}\begin{gathered}\xymatrix@C=2em@R=2em{(P_f)_r\ar[d]\ar[r]^-{\pi_f} & A\ar[d]^-{f} \\
	(PB)_r\ar[r]^-{d_1} & B \\}\end{gathered}\end{equation}
Note that $\pi_f$ is a split surjection in $\Mod_\ell$, since so is $d_1$. Define $\iota_f$ as the unique morphism that makes the following diagram commute:
\begin{equation}\label{eq:defiiotaf}\begin{gathered}\xymatrix{B^{\fS_1}_r\ar@/_/[ddr]_-{\mbox{\scriptsize inc}}\ar@/^/[drr]^-{0}\ar[dr]|-{\iota_f} & & \\
 & (P_f)_r\ar[d]\ar[r]^-{\pi_f} & A\ar[d]^-{f} \\
 & (PB)_r\ar[r]^-{d_1} & B }\end{gathered}\end{equation}

\begin{lem}[cf. \cite{cortho}*{Lemma 6.3.1}]\label{lem:pfexact}Let $f:A\to B$ be a morphism in $\Algl$ and let $C\in\Algl$. Then the following sequence is exact:
	\begin{equation}\label{eq:exc}\xymatrix@C=4em{\displaystyle\colim_{s} \fk(C,(P_f)_s)\ar[r]^-{(\pi_f)_*} & \fk(C,A)\ar[r]^-{f_*} & \fk(C,B)}\end{equation}
\end{lem}
\begin{proof}Let $s\geq 0$ and note that $\fk(C,(PB)_s)=0$ since $(PB)_s$ is contractible. Then the following composite is zero, because it factors through $\fk(C,(PB)_s)$:
	\[\xymatrix@C=4em{\fk(C,(P_f)_s)\ar[r]^-{(\pi_f)_*} & \fk(C,A)\ar[r]^-{f_*} & \fk(C,B)}\]
	This shows that the composite in \eqref{eq:exc} is zero.

	Now let $g:J^mC\to A^{\fS_m}_s$ be a morphism in $\Algl$ that represents an element $\alpha\in \fk(C,A)$ such that $f_*(\alpha)=0\in \fk(C,B)$. Increasing $m$ and $s$ if necessary, we may assume that the following composite is nullhomotopic:
	\[\xymatrix{J^mC\ar[r]^-{g} & A^{\fS_m}_s\ar[r]^-{f^{\fS_m}_s} & B^{\fS_m}_s}\]
This implies the existence of a commutative square of algebras:
	\begin{equation}\label{eq:exc2}\begin{gathered}\xymatrix@C=2em@R=2em{J^mC\ar[r]^-{g}\ar[d] & A^{\fS_m}_s\ar[d]^-{f^{\fS_m}_s} \\
		(B^{\fS_m}_s)^{(I,\{1\})}_r\ar[r]^-{d_1} & B^{\fS_m}_s \\}\end{gathered}\end{equation}
	Since $(?)^{\fS_m}_s:\Algl\to\Algl$ commutes with finite limits, we get the following pullback upon applying this functor to \eqref{eq:mappath}:
	\[\xymatrix@C=2em@R=2em{((P_f)_r)^{\fS_m}_s\ar[r]\ar[d] & A^{\fS_m}_s\ar[d]^-{f^{\fS_m}_s} \\
		(B^{(I,\{1\})}_r)^{\fS_m}_s\ar[r]^-{d_1} & B^{\fS_m}_s }\]
	Note that $(B^{(I,\{1\})}_r)^{\fS_m}_s\cong (B^{\fS_m}_s)^{(I,\{1\})}_r$. Then the commutativity of \eqref{eq:exc2} determines a morphism $J^mC\longrightarrow ((P_f)_r)^{\fS_m}_s$, that in turn gives an element $\beta\in \fk(C, (P_f)_r)$ mapping to $\alpha$.
\end{proof}

\begin{defi}
	Let $f:A\to B$ be a morphism in $\Algl$. We call $f$ a \emph{$\fk$-equivalence} if it becomes invertible upon applying $j:\Algl\to\fk$.
\end{defi}

\begin{lem}[c.f. \cite{cortho}*{Lemma 6.3.2}]\label{lem:kerkkequiv}
	Let $f$ be a morphism in $\Algl$ that is a split surjection in $\Mod_\ell$. Then the natural inclusions $\ker f\to (P_f)_r$ are $\fk$-equivalences for all $r$.
\end{lem}
\begin{proof}
	The proof is like that of \cite{cortho}*{Lemma 6.3.2}.
\end{proof}

Let $f:A\to B$ be a morphism in $\Algl$ that is a split surjection in $\Mod_\ell$. As explained in the discussion following \cite{cortho}*{Lemma 6.3.2}, Lemma \ref{lem:kerkkequiv} implies that the morphisms $(P_f)_r\to (P_f)_{r+1}$ are $\fk$-equivalences for all $r\geq 0$. Indeed, this follows from the `two out of three' property of $\fk$-equivalences. Combining this fact with Lemma \ref{lem:pfexact} we get:

\begin{coro}\label{coro:exap0}
	Let $f:A\to B$ be a morphism in $\Algl$ that is a split surjection in $\Mod_\ell$ and let $C\in\Algl$. The following sequence is exact:
	\[\xymatrix@C=4em{\fk(C,(P_f)_0)\ar[r]^-{(\pi_f)_*} & \fk(C,A)\ar[r]^-{f_*} & \fk(C,B)}\]
\end{coro}

\begin{coro}[\cite{cortho}*{Corollary 6.3.3}]\label{coro:phi}
	Let $f:A\to B$ be a morphism in $\Algl$. Recall the definitions of $\pi_f$ and $\iota_f$ from \eqref{eq:mappath} and \eqref{eq:defiiotaf} respectively. Let $\phi_f:B^{\fS_1}_0\to (P_{\pi_f})_0$ be the unique morphism that makes the following diagram commute:
	\[\xymatrix{B^{\fS_1}_0\ar@/_/[ddr]_-{0}\ar@/^/[drr]^-{\iota_f}\ar[dr]|-{\phi_f} & & \\
		& (P_{\pi_f})_0\ar[d]\ar[r]^-{\pi_{\pi_f}} & (P_f)_0\ar[d]^-{\pi_f} \\
		& (PA)_0\ar[r]^-{d_1} & A \\}\]
	Then $\phi_f$ is a $\fk$-equivalence.
\end{coro}
\begin{proof}
	The morphism $\pi_f$ is a split surjection in $\Mod_\ell$. The result then follows from Lemma \ref{lem:kerkkequiv} if we show that $\iota_f:B^{\fS_1}_0\to (P_f)_0$ is a kernel of $\pi_f$, and the latter is easily verified.
\end{proof}

\begin{coro}[\cite{cortho}*{Corollary 6.3.4}]\label{coro:kk-split-exact} Let $D\in\Algl$. Then the functor $\fk(D,?)$ sends split short exact sequences in $\Algl$ to short exact sequences of abelian groups.\end{coro}
\begin{proof}
	The proof of \cite{cortho}*{Corollary 6.3.4} carries over verbatim in this setting.
\end{proof}

\begin{lem}\label{lem:tr2}
	Let $f:A\to B$ be a morphism in $\Algl$ and let $\phi_f$ be the morphism defined in Corollary \ref{coro:phi}. Then the following diagram in $[\Algl]$ commutes:
	\[\xymatrix{A^{\fS_1}_0\ar[d]_-{\id}\ar[r]^-{(f^{\fS_1})^{-1}} & B^{\fS_1}_0\ar[d]^-{\phi_f}\ar[r]^-{\iota_f} & (P_f)_0\ar[d]^-{\id}\ar[r]^-{\pi_f} & A\ar[d]^-{\id} \\
		A^{\fS_1}_0\ar[r]_-{\iota_{\pi_f}} & (P_{\pi_f})_0\ar[r]_-{\pi_{\pi_f}} & (P_f)_0\ar[r]_-{\pi_f} & A}\]
\end{lem}
\begin{proof}
	The square in the middle commutes by definition of $\phi_f$; let us show that the square on the left commutes. Throughout the proof, we will omit the subscript $0$ and write $A^{\fS_1}$ instead of $A^{\fS_1}_0$, $PB$ instead of $(PB)_0$, $P_f$ instead of $(P_f)_0$, etc. We will identify the extensions \eqref{eq:pathb1} and \eqref{eq:pathcop} as explained in Example \ref{exa:subdi1} without further mention. We have:
	\[PB=(t-1)B[t]\]
	\[B^{\fS_1}=(t^2-t)B[t]\]
	\[P_f=\left\{(p(t),a)\in PB\times A: p(0)=f(a)\right\}\]
	The map $\pi_f:P_f\to A$ is given by $(p(t),a)\mapsto a$. The map $\iota_f:B^{\fS_1}\to P_f$ is given by $p(t)\mapsto (p(t),0)$. We also have:
	\[PA=(t-1)A[t]\]
	\[A^{\fS_1}=(t^2-t)A[t]\]
	\begin{align*}P_{\pi_f}&=\left\{(p(t),a,q(t))\in PB\times A\times PA: p(0)=f(a), a=q(0)\right\}\\
	&=\left\{(p(t),q(t))\in PB\times PA: p(0)=f(q(0))\right\}\end{align*}
	The map $\phi_f:B^{\fS_1}\to P_{\pi_f}$ is given by $p(t)\mapsto (p(t),0)$. The map $\iota_{\pi_f}:A^{\fS_1}\to P_{\pi_f}$ is given by $q(t)\mapsto (0,q(t))$. To prove the result, it suffices to show that the following morphisms $A^{\fS_1}\to P_{\pi_f}$ are homotopic:
	\begin{align}\label{elemh}\begin{split}\phi_f\circ \left(f^{\fS_1}\right)^{-1}&: q(t)\mapsto (f(q(1-t)),0)\\
	\iota_{\pi_f}&: q(t)\mapsto (0,q(t))\end{split}\end{align}
	We have:
	\[P_{\pi_f}[u]=\left\{(p(t,u),q(t,u))\in (t-1)B[t,u]\times (t-1)A[t,u]: p(0,u)=f(q(0,u))\right\}\]
	Let $H:A^{\fS_1}\to P_{\pi_f}[u]$ be the homotopy defined by:
	\[H(q(t))=(f(q((1-t)u)), q(1-(1-t)(1-u)))\]
	Then $H$ is an elementary homotopy between the morphisms in \eqref{elemh}.
\end{proof}

\begin{thm}[\cite{cortho}*{Theorem 6.3.6}]\label{thm:excision}
	Let $A\overset{f}\to B\overset{g}\to C$ be an extension in $\Algl$. Then, for any $D\in \Algl$, the following sequence is exact:
	\[\xymatrix@C=1.7em{\fk(D,B^{\fS_1}_0)\ar[r]^-{(g^{\fS_1})_*} & \fk(D,C^{\fS_1}_0)\ar[r]^-{\partial} & \fk(D,A)\ar[r]^-{f_*} & \fk(D,B)\ar[r]^-{g_*} & \fk(D,C)}\]
	Here, the morphism $\partial$ is the composite:
	\[\xymatrix@C=5em{\fk(D,C^{\fS_1}_0)\ar[r]^-{(\iota_g)_*} & \fk(D,(P_g)_0) & \fk(D,A)\ar[l]_-{\mbox{\scriptsize Lemma \ref{lem:kerkkequiv}}}^-{\cong}}\]
\end{thm}
\begin{proof}
	Both $g$ and $\pi_g$ are split surjections in $\Mod_\ell$; then, the following sequence is exact by Corollary \ref{coro:exap0}:
	\begin{equation}\label{eq:secex}\xymatrix{\fk(D,(P_{\pi_g})_0)\ar[r]^-{(\pi_{\pi_g})_*} & \fk(D,(P_g)_0)\ar[r]^-{(\pi_g)_*} & \fk(D,B)\ar[r]^-{g_*} & \fk(D,C)}\end{equation}
	We have a commutative diagram:
	\[\xymatrix{A\ar@/_/[ddr]_-{0}\ar@/^/[drr]^-{f}\ar[dr] & & \\
		& (P_g)_0\ar[r]^-{\pi_g}\ar[d] & B\ar[d]^-{g} \\
		& (PC)_0\ar[r] & C }\]
	The morphism $A\to (P_g)_0$ is a $\fk$-equivalence by Lemma \ref{lem:kerkkequiv}. Thus, we can replace $(\pi_g)_*$ in \eqref{eq:secex} by the composite
	\[\xymatrix{\fk(D,(P_g)_0) & \fk(D,A)\ar[l]_-{\cong}\ar[r]^-{f_*} & \fk(D,B)}\]
	and we get an exact sequence:
	\begin{equation}\label{eq:corophitermine}\xymatrix{\fk(D,(P_{\pi_g})_0)\ar[r] & \fk(D,A)\ar[r]^-{f_*} & \fk(D,B)\ar[r]^-{g_*} & \fk(D,C)}\end{equation}
	By Corollary \ref{coro:phi}, we can identify $(\phi_g)_*:\fk(D,C^{\fS_1}_0)\overset{\cong}\to \fk(D,(P_{\pi_g})_0)$ and \eqref{eq:corophitermine} becomes:
	\begin{equation}\label{eq:secex2}\xymatrix{\fk(D,C^{\fS_1}_0)\ar[r] & \fk(D,A)\ar[r]^-{f_*} & \fk(D,B)\ar[r]^-{g_*} & \fk(D,C)}\end{equation}
	It is easily verified that $\partial$ is the leftmost morphism in \eqref{eq:secex2}; indeed, this follows from the equality $\pi_{\pi_g}\circ\phi_g=\iota_g$.

	By Lemma \ref{lem:tr2}, we have a commutative diagram as follows:
	\begin{equation}\label{eq:diagram-6.3.6}\begin{gathered}
	\xymatrix{\fk(D,B^{\fS_1}_0)\ar[d]_-{\id}\ar[r]^-{-(g^{\fS_1})_*} & \fk(D,C^{\fS_1}_0)\ar[r]^-{(\iota_g)_*}\ar[d]^-{(\phi_g)_*} & \fk(D,(P_g)_0)\ar[d]^-{\id} \\
		\fk(D,B^{\fS_1}_0)\ar[d]_-{(\phi_{\pi_g})_*}\ar[r]^-{(\iota_{\pi_g})_*} & \fk(D,(P_{\pi_g})_0)\ar[d]^-{\id}\ar[r]^-{(\pi_{\pi_g})_*} & \fk(D,(P_g)_0)\ar[d]^-{\id} \\
		\fk(D,(P_{\pi_{\pi_g}})_0)\ar[r]^-{(\pi_{\pi_{\pi_g}})_*} & \fk(D,(P_{\pi_g})_0)\ar[r]^-{(\pi_{\pi_g})_*} & \fk(D,(P_g)_0) \\}\end{gathered}
	\end{equation}
	Notice that $\ker\partial=\ker\left((\iota_g)_*:\fk(D,C^{\fS_1}_0)\to\fk(D,(P_g)_0)\right)$. Thus, to finish the proof, it suffices to show that the top row in \eqref{eq:diagram-6.3.6} is exact. Since $(\phi_g)_*$ and $(\phi_{\pi_g})_*$ are isomorphisms by Corollary \ref{coro:phi}, the top row in \eqref{eq:diagram-6.3.6} is exact if and only if the bottom one is. But $\pi_{\pi_g}$ is a split surjection in $\Mod_\ell$, and so the bottom row in \eqref{eq:diagram-6.3.6} is exact by Corollary \ref{coro:exap0}.
\end{proof}

\section{The translation functor}\label{sec:trans} Define a functor $L:\fk\to\fk$ as follows. For $(A,m)\in\fk$, put $L(A,m):=(A,m+1)$. Let $L$ act as the identity on morphisms. It is clear that $L$ is an automorphism of $\fk$.

Recall the definition of $\lambda_B:JB\to B^{\fS_1}_r$ from Example \ref{exa:pathb1}. For $m\in\Z$ we can consider:
\begin{align}\label{eq:trans1}\begin{split}\langle\lambda_B\rangle &\in\fk((JB,m),(B,1+m))\\
\langle\id_{JB}\rangle &\in\fk((B,1+m),(JB,m))\end{split}\end{align}
These two morphisms are mutually inverses in $\fk$, as we prove below. From now on, each time we identify $(B,1+m)\cong(JB,m)$ it will be through these isomorphisms. Using this identification $n$ times we get an isomorphism $(B,n+m)\overset{\cong}\to(J^nB,m)$. It is easily verified that the latter is represented by $\id_{J^nB}$.

\begin{lem}\label{lem:lambdaidinverses1}
	The morphisms in \eqref{eq:trans1} are mutually inverses.
\end{lem}
\begin{proof}
	We have that $\langle\id_{JB}\rangle\circ\langle\lambda_B\rangle=\langle \id_{JB}\star\lambda_B\rangle$, where $\id_{JB}\star\lambda_B$ equals the following composite in $[\Algl]^\ind$:
	\[\xymatrix@C=3em{J(JB)\ar[r]^-{J(\lambda_B)} & J(B^{\fS_1}_\bul)\ar[r]^-{-\kappa^{1,1}_B} & (JB)^{\fS_1}_\bul}\]
	By Lemma \ref{lem:clascon}, $\id_{JB}\star\lambda_B=\lambda_{JB}$ and thus $\langle\id_{JB}\rangle\circ\langle\lambda_B\rangle=\langle\lambda_{JB}\rangle=\id_{(JB,m)}$.

	We have that $\langle\lambda_B\rangle\circ\langle\id_{JB}\rangle=\langle\lambda_B\star\id_{JB}\rangle$. It is easily verified that $\lambda_B\star\id_{JB}=\lambda_B$ and thus $\langle\lambda_B\rangle\circ\langle\id_{JB}\rangle=\langle\lambda_B\rangle=\id_{(B,1+m)}$.
\end{proof}

We can also consider:
\begin{align}\label{eq:trans2}\begin{split}\langle\lambda_B\rangle&\in\fk((B,1+m),(B^{\fS_1}_0,m))\\
\langle\id_{B^{\fS_1}}\rangle &\in\fk((B^{\fS_1}_0,m),(B,1+m))\end{split}\end{align}
We will show below that these morphisms are mutually inverses. Each time we identify $(B,1+m)\cong(B^{\fS_1}_0,m)$ it will be through these isomorphisms.

\begin{lem}[\cite{cortho}*{Lemma 6.3.10}]\label{lem:lambdaiso}
	The morphism $\lambda_B:JB\to B^{\fS_1}_0$ is a $\fk$-equivalence.
\end{lem}
\begin{proof}
	The proof of \cite{cortho}*{Lemma 6.3.10} works verbatim, using Theorem \ref{thm:excision}.
\end{proof}

\begin{lem}\label{lem:lambdaidinverses2}
	The morphisms in \eqref{eq:trans2} are mutually inverses.
\end{lem}
\begin{proof}
We claim that $\langle\lambda_B\rangle$ is an isomorphism. To see this, notice that it equals the following composite:
\[\xymatrix@C=5em{(B,1+m)\ar[r]^-{\langle\id_{JB}\rangle} & (JB,m)\ar[r]^-{L^m(\lambda_B)} & (B^{\fS_1}_0,m)}\]
Here, $\langle\id_{JB}\rangle$ is an isomorphism by Lemma \ref{lem:lambdaidinverses1} and $L^m(\lambda_B)$ is an isomorphism by Lemma \ref{lem:lambdaiso}. This proves the claim. To prove the lemma it suffices to show that $\langle\id_{B^{\fS_1}}\rangle\circ \langle\lambda_B\rangle=\id_{(B,1+m)}$. This follows immediately from the definitions.
\end{proof}

\begin{rem}\label{rem:fS1L}
	By Lemma \ref{lem:lambdaidinverses2}, there is a natural isomorphism $j\circ(?)^{\fS_1}_0\cong L\circ j$ of functors $\Algl\to\fk$. Hence, a morphism $f$ in $\Algl$ is a $\fk$-equivalence if and only if $f^{\fS_1}_0$ is.
\end{rem}

\begin{lem}\label{lem:lastvertexkkequiv}
	The morphisms $B^{\fS_n}_r\to B^{\fS_n}_{r+1}$ are $\fk$-equivalences.
\end{lem}
\begin{proof}
	We proceed by induction on $n$. The result holds for $n=0$ as in this case $B^{\fS_0}_r\to B^{\fS_0}_{r+1}$ is the identity of $B$. For the inductive step, consider the morphism of extensions:
	\begin{equation}\label{eq:lastvertexkkequiv}\begin{gathered}\xymatrix@R=2em{B^{\fS_{n+1}}_r\ar[d]\ar[r] & P(n,B)_r\ar[d]\ar[r] & B^{\fS_n}_r\ar[d] \\
		B^{\fS_{n+1}}_{r+1}\ar[r] & P(n,B)_{r+1}\ar[r] & B^{\fS_n}_{r+1}}\end{gathered}\end{equation}
	The middle vertical morphism is a $\fk$-equivalence since both its source and target are contractible by Lemma \ref{lem:pathcontr}. The right vertical morphism is a $\fk$-equivalence by induction hypothesis. Then the left vertical morphism is a $\fk$-equivalence too by Theorem \ref{thm:excision}.
\end{proof}

\begin{lem}\label{lem:mukkequiv}
	The morphisms $\mu^{m,n}_B:(B^{\fS_m}_r)^{\fS_n}_s\to B^{\fS_{m+n}}_{r+s}$ are $\fk$-equivalences for all $m$ and $n$.
\end{lem}
\begin{proof}Let us start with the case $n=1$. Consider the following morphism of extensions:
	\begin{equation}\label{eq:mukkequiv}\begin{gathered}\xymatrix@C=4em@R=2em{(B^{\fS_m}_r)^{\fS_1}_s\ar[d]_{\mu^{m,1}}\ar[r] & P(B^{\fS_m}_r)_s\ar[d]^-{\mu^{\fS_m,(I,\{1\})}}\ar[r] & B^{\fS_m}_r\ar[d]^-{(\gamma^s)^*} \\
		B^{\fS_{m+1}}_{r+s}\ar[r] & P(m,B)_{r+s}\ar[r] & B^{\fS_m}_{r+s}}\end{gathered}\end{equation}
	The result follows from Theorem \ref{thm:excision} since $P(B^{\fS_m}_r)_s$ and $P(m,B)_{r+s}$ are contractible by Lemma \ref{lem:pathcontr} and $(\gamma^s)^*$ is a $\fk$-equivalence by Lemma \ref{lem:lastvertexkkequiv}.

	We will prove the general case by induction on $n$. Suppose that $\mu^{m,n}_B$ is a $\fk$-equivalence for all $m$. The following square commutes by the associativity of $\mu$:
	\[\xymatrix@C=4em{\left(\left(B^{\fS_m}_r\right)^{\fS_n}_s\right)^{\fS_1}_0\ar[r]^-{\mu^{n,1}_{B^{\fS_m}_r}}\ar[d]_-{\left(\mu^{m,n}_B\right)^{\fS_1}_0} & \left(B^{\fS_m}_r\right)^{\fS_{n+1}}_{s}\ar[d]^-{\mu^{m,n+1}_B} \\
		\left(B^{\fS_{m+n}}_{r+s}\right)^{\fS_1}_0\ar[r]^-{\mu^{m+n,1}_B} & B^{\fS_{m+n+1}}_{r+s}}\]
The horizontal morphisms are $\fk$-equivalences by the case $n=1$ and the left vertical morphism is a $\fk$-equivalence since $\mu^{m,n}_B$ is; see Remark \ref{rem:fS1L}. Then $\mu^{m,n+1}_B$ is a $\fk$-equivalence.
\end{proof}

\begin{lem}\label{lem:idBfSkkequiv}
The identity of $B^{\fS_n}_r$ induces an isomorphism $\langle\id_{B^{\fS_n}_r}\rangle\in\fk((B^{\fS_n}_r,m), (B,n+m))$.
\end{lem}
\begin{proof}
	By Lemma \ref{lem:lastvertexkkequiv} we may assume that $r=0$. We will proceed by induction on $n$. The case $n=1$ holds by Lemma \ref{lem:lambdaidinverses2}. Suppose now that the result holds for $n\geq 1$. It is easily verified that the following diagram in $\fk$ commutes:
	\[\xymatrix@C=5em{((B^{\fS_n}_0)^{\fS_1}_0,m)\ar[d]_-{L^m\left(\mu^{n,1}_B\right)}\ar[r] & (B^{\fS_n}_0,1+m)\ar[d] \\
		(B^{\fS_{n+1}}_0,m)\ar[r]^-{\langle\id_{B^{\fS_{n+1}}}\rangle} & (B,n+1+m)}\]
	The top horizontal morphism is an isomorphism by Lemma \ref{lem:lambdaidinverses2}, the right vertical one is an isomorphism by induction hypothesis, and $L^m(\mu^{n,1}_B)$ is an isomorphism by Lemma \ref{lem:mukkequiv}.
\end{proof}

\begin{rem}\label{rem:natisos}
	By Lemma \ref{lem:idBfSkkequiv} we have natural isomorphisms $j\circ (?)^{\fS_n}_0\cong L^n\circ j$ and $j\circ ((?)^{\fS_n}_0)^{\fS_1}_0\cong L^{n+1}\circ j$. Indeed, it is easily seen that the following squares commute for every morphism $f:A\to B$ in $\Algl$:
	\[\xymatrix@C=4em{(A^{\fS_n}_0,0)\ar[d]_-{f}\ar[r]^-{\langle\id_{A^{\fS_n}}\rangle}_-{\cong} & (A,n)\ar[d]^-{L^n(f)} \\
		(B^{\fS_n}_0,0)\ar[r]^-{\langle\id_{B^{\fS_n}}\rangle}_-{\cong} & (B,n) }\hspace{1em}\xymatrix{((A^{\fS_n}_0)^{\fS_1}_0,0)\ar[d]_-{f}\ar[r]^-{\langle\mu^{n,1}_A\rangle}_-{\cong} & (A,n+1)\ar[d]^-{L^{n+1}(f)} \\
		((B^{\fS_n}_0)^{\fS_1}_0,0)\ar[r]^-{\langle\mu^{n,1}_B\rangle}_-{\cong} & (B,n+1)}\]
\end{rem}

\begin{lem}\label{lem:arrowiso}
	Let $\alpha\in\fk((A,m),(B,n))$ be represented by $f:J^{m+u}A\to B^{\fS_{n+u}}_r$. Then the following diagram commutes:
	\[\xymatrix@C=4em{(A,m)\ar[d]^-{\cong}_-{\langle\id_{J^{m+u}A}\rangle}\ar[r]^-{\alpha} & (B,n) \\
		(J^{m+u}A,-u)\ar[r]^-{L^{-u}(f)} & (B^{\fS_{n+u}}_r,-u)\ar[u]^-{\cong}_-{\langle\id_{B^{\fS_{n+u}}}\rangle}}\]
\end{lem}
\begin{proof}
	It is a straightforward computation.
\end{proof}

\section{Long exact sequences associated to extensions}\label{sec:long}

\begin{lem}\label{lem:longexact}
	Let $\scrE: A\overset{f}\to B\overset{g}\to C$ be an extension. Then, for every $\ell$-algebra $D$, there is a natural long exact sequence of abelian groups:
	\[\xymatrix@C=1.3em{\cdots\ar[r] & \fk(D,(A,n))\ar[r]^-{f_*} & \fk(D,(B,n))\ar[r]^-{g_*} & \fk(D,(C,n))\ar[r]^-{\partial} & \fk(D,(A,n-1))\ar[r] & \cdots}\]
\end{lem}
\begin{proof}
	By Theorem \ref{thm:excision} applied to the extension $\scrE^{\fS_n}_0$ we have an exact sequence:
	\[\xymatrix@C=1.5em@R=1.5em{\fk(D,(B^{\fS_n}_0)^{\fS_1}_0)\ar[r]^-{g_*} & \fk(D,(C^{\fS_n}_0)^{\fS_1}_0)\ar[r]^-{\partial} & \fk(D,A^{\fS_n}_0)\ar[r]^-{f_*} & \fk(D,B^{\fS_n}_0)\ar[r]^-{g_*} & \fk(D,C^{\fS_n}_0)}\]
	Under the identifications described in Remark \ref{rem:natisos}, the latter becomes:
	\[\xymatrix@C=1.5em@R=1.5em{\fk(D,(B,n+1))\ar[r]^-{g_*} & \fk(D,(C,n+1))\ar[r]^-{\partial} & \fk(D,(A,n))\ar[r]^-{f_*} & \fk(D,(B,n))\ar[r]^-{g_*} & \fk(D,(C,n))}\]
	For varying $n\geq 0$, these sequences assemble into an exact sequence, infinite to the left, ending in $\fk(D,(C,0))$. It remains to show how to extend this sequence to the right. Upon applying $\fk(D^{\fS_n}_0,?)$ to $\scrE$, we get an exact sequence:
	\[\xymatrix@C=1.5em@R=1.5em{\fk(D^{\fS_n}_0,(B,1))\ar[r]^-{g_*} & \fk(D^{\fS_n}_0,(C,1))\ar[r]^-{\partial} & \fk(D^{\fS_n}_0,A)\ar[r]^-{f_*} & \fk(D^{\fS_n}_0,B)\ar[r]^-{g_*} & \fk(D^{\fS_n}_0,C)}\]
	After identifying $\fk(D^{\fS_n}_0,?)\cong\fk((D,n),?)\cong\fk(D,L^{-n}(?))$, this sequence becomes:
	\[\xymatrix@C=1em@R=1.5em{\fk(D,(B,1-n))\ar[r]^-{g_*} & \fk(D,(C,1-n))\ar[r]^-{\partial} & \fk(D,(A,-n))\ar[r]^-{f_*} & \fk(D,(B,-n))\ar[r]^-{g_*} & \fk(D,(C,-n))}\]
	Now glue these for varying $n\geq 0$ to extend the exact sequence to the right.
\end{proof}

\begin{lem}\label{lem:mayerviet}
	Let $D\in\Algl$. For every pullback square in $\Algl$	\[\xymatrix{B'\ar[d]\ar[r] & C'\ar[d] \\ B\ar[r]^-{g} & C}\]
	where $g$ is a split surjection in $\Mod_\ell$, there is a long exact Mayer-Vietoris sequence:
	\[\xymatrix@C=1em{\cdots\fk(D, (B',n))\ar[r] & \fk(D,(B,n))\oplus\fk(D,(C',n))\ar[r] & \fk(D,(C,n))\ar[r]^-{\partial} & \fk(D,(B',n-1))\cdots}\]
\end{lem}
\begin{proof}
	It follows from Lemma \ref{lem:longexact} and  the argument explained in \cite{ralf}*{Theorem 2.41}.
\end{proof}

\begin{coro}
	Let $f$ be any morphism in $\Algl$. Then the morphisms $(P_f)_r\to (P_f)_{r+1}$ are $\fk$-equivalences for all $r$.
\end{coro}
\begin{proof}
	Pulling back the path extension $\scrP_{0,B}$ along $f:A\to B$, we get a long exact Mayer-Vietoris sequence that takes the following form, since $(PB)_r$ is contractible:
	\[\label{eq:mayer1}\xymatrix@C=1.2em{\cdots\ar[r] & \fk(D,(A,1))\ar[r] & \fk(D,(B,1))\ar[r]^-{\partial} & \fk(D,(P_f)_r))\ar[r] & \fk(D,A)\ar[r] & \fk(D,B)\ar[r] & \cdots}\]
	As this sequence is natural in $r$, the result follows from the five lemma and Yoneda.
\end{proof}

\section{Triangulated structure}\label{sec:triangu} We will use the definition of triangulated category given in \cite{neeman}. As explained in \cite{cortho}*{Section 6.5}, the definition of triangulated category is self dual, and we will actually prove that $\fk^\op$ is a triangulated category, since this is more natural in our context. Recall the definitions of $\pi_f$ and $\iota_f$ from \eqref{eq:mappath} and \eqref{eq:defiiotaf}.

\begin{defi}\label{defi:trianglesinkk}
	We call \emph{mapping path triangle} to a diagram in $\fk$ of the form
	\[\triangle_{f,n}:\xymatrix@C=4em{L(B,n)\ar[r]^-{\partial_{f,n}} & ((P_f)_0,n)\ar[r]^-{L^n(\pi_f)} & (A,n)\ar[r]^-{L^n(f)} & (B,n) },\]
	where $f:A\to B$ is a morphism in $\Algl$, $n\in\Z$ and $\partial_{f,n}$ equals the composite:
	\[\xymatrix@C=6em{(B,n+1) & (B^{\fS_1}_0,n)\ar[l]_-{\langle\id_{B^{\fS_1}}\rangle}^-{\cong}\ar[r]^-{(-1)^{n+1}L^n(\iota_f)} & ((P_f)_0,n)}\]
	A \emph{distinguished triangle} in $\fk$ is a triangle isomorphic (as a triangle) to some $\triangle_{f,n}$.
\end{defi}

We are ready to verify that $\fk$ satisfies the axioms of a triangulated category with the translation functor $L$ and the distinguished triangles defined above.

\begin{ax}[TR0]
	Any triangle which is isomorphic to a distinguished triangle is itself distinguished. For any $B\in\Algl$ and any $n\in \Z$, the following triangle is distinguished:
	\[\xymatrix{L(B,n)\ar[r] & 0\ar[r] & (B,n)\ar[r]^-{\id} & (B,n)}\]
\end{ax}
\begin{proof}
	It follows from the fact that the mapping path $(P_{\id_B})_0\cong (PB)_0$ is contractible.
\end{proof}

\begin{ax}[TR1]
	Every morphism $\alpha$ in $\fk$ fits into a distinguished triangle of the form:
	\[\xymatrix{L(Y)\ar[r] & Z\ar[r] & X\ar[r]^-{\alpha} & Y}\]
\end{ax}
\begin{proof}
	By Lemma \ref{lem:arrowiso} we can assume that $X=(C,k)$, $Y=(B,k)$ and $\alpha=L^{k}(f)$ with $f:C\to D$ a morphism in $\Algl$. In this case $\alpha$ fits into the mapping path triangle $\triangle_{f,k}$.
\end{proof}

\begin{defi}\label{defi:rotatedtriangle}Consider a triangle $\triangle$ in $\fk$:
	\begin{equation}\label{eq:triangle}\triangle: \xymatrix{L(Z)\ar[r]^-{\alpha} & X\ar[r]^-{\beta} & Y\ar[r]^-{\gamma} & Z}\end{equation}
	Define the \emph{rotated triangle} $R(\triangle)$ by:
	\[R(\triangle): \xymatrix{L(Y)\ar[r]^-{-L\gamma} & L(Z)\ar[r]^-{-\alpha} & X\ar[r]^-{-\beta} & Y}\]
\end{defi}

\begin{rem} As explained in \cite{ralf}*{Definition 6.51}, we have an isomorphism:
	\[R(\triangle)\cong(\xymatrix{L(Y)\ar[r]^-{-L\gamma} & L(Z)\ar[r]^-{\alpha} & X\ar[r]^-{\beta} & Y})\]
\end{rem}

\begin{ax}[TR2]A triangle $\triangle$ is distinguished if and only if $R(\triangle)$ is.
\end{ax}
\begin{proof}
	Let us first show that if $\triangle$ is distinguished, then $R(\triangle)$ is distinguished as well. It suffices to prove that the rotation of a mapping path triangle is distinguished. Let $f:A\to B$ be a morphism in $\Algl$ and consider the following mapping path triangles:
	\[\triangle_{f,n}:\xymatrix@C=4em{L(B,n)\ar[r]^-{\partial_{f,n}} & ((P_f)_0,n)\ar[r]^-{L^n(\pi_f)} & (A,n)\ar[r]^-{L^n(f)} & (B,n) }\]
	\[\triangle_{\pi_f,n}:\xymatrix@C=3.5em{L(A,n)\ar[r]^-{\partial_{{\pi_f},n}} & ((P_{\pi_f})_0,n)\ar[r]^-{L^n(\pi_{\pi_f})} & ((P_f)_0,n)\ar[r]^-{L^n(\pi_f)} & (A,n) }\]
	Let $\varepsilon:L(B,n)\to ((P_{\pi_f})_0,n)$ be the following composite, where $\phi_f$ is the morphism defined in Corollary \ref{coro:phi}:
	\[\xymatrix@C=6em{(B,n+1) & (B^{\fS_1}_0,n)\ar[l]_-{\langle\id_{B^{\fS_1}}\rangle}^-{\cong}\ar[r]^-{(-1)^{n+1}L^n(\phi_f)} & ((P_{\pi_f})_0,n)}\]
	Notice that $\varepsilon$ is an isomorphism by Corollary \ref{coro:phi}. It follows from Lemma \ref{lem:tr2} that we have an isomorphism $R(\triangle_{f,n})\cong\triangle_{\pi_f,n}$ as follows:
	\[\xymatrix@C=3em{R(\triangle_{f,n})\ar[d]^-{\cong} & L(A,n)\ar[r]^-{-L^{n+1}(f)}\ar[d]^-{\id} & L(B,n)\ar[r]^-{-\partial_{f,n}}\ar[d]^-{\varepsilon} & ((P_f)_0,n)\ar[d]^-{-\id}\ar[r]^-{-L^n(\pi_f)} & (A,n)\ar[d]^-{\id} \\
		\triangle_{\pi_f,n} & L(A,n)\ar[r]^-{\partial_{{\pi_f},n}} & ((P_{\pi_f})_0,n)\ar[r]^-{L^n(\pi_{\pi_f})} & ((P_f)_0,n)\ar[r]^-{L^n(\pi_f)} & (A,n)}\]
	This shows that the rotation of a mapping path triangle is distinguished.

	We still have to prove that if $R(\triangle)$ is distinguished, then $\triangle$ is distinguished. We claim that if $R^3(\triangle)$ is distinguished, then $\triangle$ is distinguished; suppose for a moment that this claim is proved. If $R(\triangle)$ is distinguished then $R^3(\triangle)$ is distinguished ---because $R$ preserves distinguished triangles--- and so $\triangle$ is distinguished ---by the claim. Thus, the proof will be finished if we prove the claim. Let $\triangle$ be the triangle in \eqref{eq:triangle}. Then:
	\[R^3(\triangle)\cong (\xymatrix{L^2(Z)\ar[r]^-{-L\alpha} & L(X)\ar[r]^-{L\beta} & L(Y)\ar[r]^-{L\gamma} & L(Z)})\]
	Suppose that $R^3(\triangle)$ is distinguished. Then there exists a morphism $f:A\to B$ in $\Algl$ that fits into an isomorphism of triangles as follows:
	\begin{equation}\label{eq:isotria}\begin{gathered}\xymatrix@C=4em{L^2(Z)\ar[r]^-{-L\alpha}\ar[d]^-{\cong} & L(X)\ar[d]^-{\cong}\ar[r]^-{L\beta} & L(Y)\ar[d]^-{\cong}\ar[r]^-{L\gamma} & L(Z)\ar[d]^-{\cong} \\
		L(B,n)\ar[r]^-{\partial_{f,n}} & ((P_f)_0,n)\ar[r]^-{L^n(\pi_f)} & (A,n)\ar[r]^-{L^n(f)} & (B,n)}\end{gathered}\end{equation}
	Upon applying $L^{-1}$ to \eqref{eq:isotria} we get a commutative diagram as follows:
	\[\xymatrix@C=4em{L(Z)\ar[r]^-{-\alpha}\ar[d]^-{\cong} & X\ar[d]^-{\cong}\ar[r]^-{\beta} & Y\ar[d]^-{\cong}\ar[r]^-{\gamma} & Z\ar[d]^-{\cong} \\
		(B,n)\ar[r]^-{-\partial_{f,n-1}} & ((P_f)_0,n-1)\ar[r]^-{L^{n-1}(\pi_f)} & (A,n-1)\ar[r]^-{L^{n-1}(f)} & (B,n-1)}\]
	Thus, the vertical morphisms in the latter diagram assemble into an isomorphism of triangles $\triangle\cong \triangle_{f,n-1}$. Then $\triangle$ is distinguished.
\end{proof}

\begin{lem}\label{lem:triaisoJ}
	Let $f:A\to B$ be a morphism in $\Algl$, let $k\in \Z$ and let $n\geq 0$. Then there is a morphism of triangles:
	\[\xymatrix{\triangle_{f,k+n}\ar[d] & L(B,k+n)\ar[r]\ar[d]^-{\cong} & ((P_f)_0,k+n)\ar[r]\ar[d] & (A,k+n)\ar[r]\ar[d]_-{\langle\id_{J^nA}\rangle}^-{\cong} & (B,k+n)\ar[d]_-{\langle\id_{J^nB}\rangle}^-{\cong} \\
		\triangle_{J^n(f),k} & L(J^nB,k)\ar[r] & ((P_{J^n(f)})_0,k)\ar[r] & (J^nA,k)\ar[r] & (J^nB,k)}\]
\end{lem}
\begin{proof}
	It is enough to construct the morphism for $n=1$ and then consider the composite:
	\[\xymatrix@R=2em{\triangle_{f,k+n}\ar[r] & \triangle_{J(f),k+n-1}\ar[r] & \triangle_{J^2(f),k+n-2}\ar[r] & \cdots\ar[r] & \triangle_{J^n(f),k}}\]
	Let $c:J(P_f)_0\to (P_{J(f)})_0$ be the morphism defined by the following diagram in $\Algl$:
	\[\xymatrix{J(P_f)_0\ar@{..>}[dr]|-{\exists !c}\ar@/_1pc/[dddr]\ar@/^1pc/[drr]^-{J(\pi_f)} & & \\
		& (P_{J(f)})_0\ar[r]^-{\pi_{J(f)}}\ar[d] & JA\ar[d]^-{J(f)} \\
		& P(JB)_0\ar[r]^-{d_1} & JB \\
		& J(PB)_0\ar[u]\ar@/_/[ur]_-{J(d_1)} & }\]
	It is easily verified that the following diagram commutes, where the unlabelled vertical morphisms are induced by the natural isomorphism $j\circ J\cong L\circ j:\Algl\to\fk$:
	\[\xymatrix@C=4.5em@R=2em{(B^{\fS_1}_0,k+1)\ar[d]^-{\cong}\ar[r]^-{L^{k+1}(\iota_f)} & ((P_f)_0,k+1)\ar[d]^-{\cong}\ar[r]^-{L^{k+1}(\pi_f)} & (A,k+1)\ar[d]^-{\cong}\ar[r]^-{L^{k+1}(f)} & (B,k+1)\ar[d]^-{\cong} \\
		(J(B^{\fS_1}_0),k)\ar[r]^-{L^k(J(\iota_f))}\ar[d]^-{L^k(\kappa^{1,1}_B)} & (J(P_f)_0,k)\ar[d]^-{L^k(c)}\ar[r]^-{L^k(J(\pi_f))} & (JA,k)\ar[d]^-{\id}\ar[r]^-{L^k(J(f))} & (JB,k)\ar[d]^-{\id} \\
		((JB)^{\fS_1}_0,k)\ar[r]^-{L^k(\iota_{J(f)})} & ((P_{J(f)})_0,k)\ar[r]^-{L^k(\pi_{J(f)})} & (JA,k)\ar[r]^-{L^k(J(f))} & (JB,k) }\]
	Another straightforward computation shows that the following square also commutes:
	\[\xymatrix@C=5em@R=2em{(B,k+2)\ar[dd]_-{\langle\id_{JB}\rangle}^-{\cong} & (B^{\fS_1}_0,k+1)\ar[l]_-{\langle\id_{B^{\fS_1}}\rangle}^-{\cong}\ar[d]^-{\langle\id_{J(B^{\fS_1})}\rangle}_-{\cong} \\
		& (J(B^{\fS_1}_0),k)\ar[d]^-{L^k(\kappa^{1,1}_B)} \\
		(JB,k+1) & ((JB)^{\fS_1}_0,k)\ar[l]_-{-\langle\id_{(JB)^{\fS_1}}\rangle}^-{\cong}}\]
	To get the desired morphism of triangles $\triangle_{f,k+1}\to\triangle_{J(f),k}$, join both diagrams above.
\end{proof}

\begin{lem}\label{lem:triaisofS}
	Let $f:A\to B$ be a morphism in $\Algl$, let $k\in \Z$ and let $n\geq 0$. Then there is an isomorphism of triangles:
	\[\xymatrix@R=2em{\triangle_{f^{\fS_n},k}\ar[d]^-{\cong} & L(B^{\fS_n}_r,k)\ar[r]\ar[d]^-{\cong} & ((P_{f^{\fS_n}})_0,k)\ar[r]\ar[d]^-{\cong} & (A^{\fS_n}_r,k)\ar[r]\ar[d]^-{\cong}_-{\langle\id_{A^{\fS_n}}\rangle} & (B^{\fS_n}_r,k)\ar[d]^-{\cong}_-{\langle\id_{B^{\fS_n}}\rangle} \\
		\triangle_{f,n+k} & L(B,n+k)\ar[r] & ((P_f)_0,n+k)\ar[r] & (A,n+k)\ar[r] & (B,n+k)}\]
\end{lem}
\begin{proof}
	It is easily verified that there is an isomorphism  $((P_f)_0)^{\fS_n}_r\cong (P_{f^{\fS_n}})_0$.
	We have a commutative diagram as follows, where the vertical morphisms from the second row to the first one are induced by the natural isomorphism $j\circ (?)^{\fS_n}_r\cong L^n\circ j:\Algl\to\fk$:
	\[\xymatrix@C=4.5em{(B^{\fS_1}_0,n+k)\ar[r]^-{L^{n+k}(\iota_f)} & ((P_f)_0,n+k)\ar[r]^-{L^{n+k}(\pi_f)} & (A,n+k)\ar[r]^-{L^{n+k}(f)} & (B,n+k) \\
		((B^{\fS_1}_0)^{\fS_n}_r,k)\ar[u]_-{\cong}\ar[r]^-{L^k\left(\iota_f^{\fS_n}\right)}\ar[d]^-{\cong} & (((P_f)_0)^{\fS_n}_r,k)\ar[d]^-{\cong}\ar[r]^-{L^k\left(\pi_f^{\fS_n}\right)}\ar[u]_-{\cong} & (A^{\fS_n}_r,k)\ar[u]_-{\cong}\ar[d]^-{\id}\ar[r]^-{L^k\left(f^{\fS_n}\right)} & (B^{\fS_n}_r,k)\ar[d]^-{\id}\ar[u]_-{\cong} \\
		((B^{\fS_n}_r)^{\fS_1}_0,k)\ar[r]^-{L^k\left(\iota_{f^{\fS_n}}\right)} & ((P_{f^{\fS_n}})_0,k)\ar[r]^-{L^k\left(\pi_{f^{\fS_n}}\right)} & (A^{\fS_n}_r,k)\ar[r]^-{L^k\left(f^{\fS_n}\right)} & (B^{\fS_n}_r,k) }\]
	A straightforward computation shows that the following square commutes:
	\[\xymatrix@C=7em@R=2em{(B,n+k+1) & (B^{\fS_1}_0,n+k)\ar[l]_-{\langle\id_{B^{\fS_1}}\rangle}^-{\cong} \\
		& ((B^{\fS_1}_0)^{\fS_n}_r,k)\ar[u]_-{\langle\id_{(B^{\fS_1})^{\fS_n}}\rangle}^-{\cong}\ar[d]^-{\cong} \\
		(B^{\fS_n}_r,k+1)\ar[uu]_-{\cong}^-{\langle\id_{B^{\fS_n}}\rangle} & ((B^{\fS_n}_r)^{\fS_1}_0,k)\ar[l]^-{\cong}_-{(-1)^n\langle\id_{(B^{\fS_n})^{\fS_1}}\rangle}}\]
	To get the desired isomorphism of triangles $\triangle_{f^{\fS_n},k}\cong\triangle_{f,n+k}$, join both diagrams above.
\end{proof}

\begin{ax}[TR3]
	For every diagram of solid arrows as follows, in which the rows are distinguished triangles, there exists a dotted arrow that makes the whole diagram commute.
	\begin{equation}\label{eq:tr3}\begin{gathered}\xymatrix@C=4em{L(Z')\ar[r]\ar[d] & X'\ar@{..>}[d]\ar[r] & Y'\ar[d]\ar[r] & Z'\ar[d] \\
		L(Z)\ar[r] & X\ar[r] & Y\ar[r] & Z}\end{gathered}\end{equation}
\end{ax}
\begin{proof}
	We follow \cite{ralf}*{Axiom 6.53}, making appropiate changes. Let us begin with a special case. Consider a commutative square in $[\Algl]$:
	\[\xymatrix{A'\ar[d]_-{a}\ar[r]^-{f'} & B'\ar[d]^-{b} \\ A\ar[r]^-{f} & B}\]
	Suppose that \eqref{eq:tr3} takes the following form, where the rows are mapping path triangles:
	\begin{equation}\label{eq:tr31}\begin{gathered}\xymatrix@C=4em{L(B',n)\ar[r]^-{\partial_{f',n}}\ar[d]_-{L^{n+1}(b)} & ((P_{f'})_0,n)\ar@{..>}[d]\ar[r]^-{L^n(\pi_{f'})} & (A',n)\ar[d]_-{L^n(a)}\ar[r]^-{L^n(f')} & (B',n)\ar[d]_-{L^n(b)} \\
		L(B,n)\ar[r]^-{\partial_{f,n}} & ((P_f)_0,n)\ar[r]^-{L^n(\pi_f)} & (A,n)\ar[r]^-{L^n(f)} & (B,n)}\end{gathered}\end{equation}
	We want to show that a dotted arrow exists in this case. Let $H:A'\to B^{\sd^rI}$ be a homotopy such that $d_1\circ H=f\circ a$ and $d_0\circ H=b\circ f'$; we may assume $r\geq 1$. We have:
	\[(P_{f'})_r=\left\{(x,y)\in A'\times (PB')_r\mid f'(x)=d_1(y)\right\}\]
	\[(P_f)_{r+1}=\left\{(x,y)\in A\times (PB)_{r+1}\mid f(x)=d_1(y)\right\}\]
	Define a morphism $c:(P_{f'})_r\to (P_f)_{r+1}$ by the formula:
	\[c(x,y)=\left(a(x),H(x)\bullet P(b)(y)\right)\]
	Here, $\bullet$ stands for \emph{concatenation} of paths; see Example \ref{exa:concat}. Note that this concatenation makes sense since $d_0(H(x))=b(f'(x))=b(d_1(y))=d_1(P(b)(y))$. Moreover, $c(x,y)$ is indeed an element of $(P_f)_{r+1}$ since we have:
	\[d_1(H(x)\bullet P(b)(y))=d_1(H(x))=f(a(x))\]
	Let $\chi:((P_{f'})_0,n)\to ((P_f)_0,n)$ be the composite:
	\[\xymatrix@C=4em{((P_{f'})_0,n)\ar[r]^-{\cong} & ((P_{f'})_r,n)\ar[r]^-{L^n(c)} & ((P_{f})_{r+1},n) & ((P_{f})_0,n)\ar[l]_-{\cong}}\]
	We claim that taking the dotted arrow in \eqref{eq:tr31} equal to $\chi$ makes the whole diagram commute. It is easily verified that the following square commutes in $\Algl$, and this implies that the middle square in \eqref{eq:tr31} commutes:
	\[\xymatrix{(P_{f'})_r\ar[r]^-{\pi_{f'}}\ar[d]_-{c} & A'\ar[d]^-{a} \\
		(P_f)_{r+1}\ar[r]^-{\pi_f} & A}\]
	Another straightforward computation shows that the following diagram commutes in $[\Algl]$, and this implies that the left square in \eqref{eq:tr31} commutes:
	\[\xymatrix{(B')^{\fS_1}_0\ar[r]\ar[d]_-{b^{\fS_1}} & (B')^{\fS_1}_r\ar[r]^-{\iota_{f'}} & (P_{f'})_r\ar[d]^-{c} \\
		B^{\fS_1}_0\ar[r] & B^{\fS_1}_{r+1}\ar[r]^-{\iota_f} & (P_f)_{r+1}}\]
	This finishes the proof of the axiom in this special case.

	In the general case, we may suppose that both triangles are mapping path triangles, so that \eqref{eq:tr3} equals the following diagram, for some morphisms $f:A\to B$ and $f':A'\to B'$:
	\begin{equation}\label{eq:tr33}\begin{gathered}\xymatrix{\triangle_{f',k'}\ar@{..>}[d] & L(B',k')\ar[r]\ar[d] & ((P_{f'})_0,k')\ar@{..>}[d]\ar[r] & (A',k')\ar[d]^-{\alpha}\ar[r] & (B',k')\ar[d]^-{\beta} \\
		\triangle_{f,k} & L(B,k)\ar[r] & ((P_f)_0,k)\ar[r] & (A,k)\ar[r] & (B,k)}\end{gathered}\end{equation}
	We may choose $l$ and $r$ large enough so that $\alpha$ is represented by $a:J^{k'+l}A'\to A^{\fS_{k+l}}_r$, $\beta$ is represented by $b:J^{k'+l}B'\to B^{\fS_{k+l}}_r$, and the following square in $[\Algl]$ commutes:
	\[\xymatrix@C=5em{J^{k'+l}A'\ar[d]_-{a}\ar[r]^-{J^{k'+l}(f')} & J^{k'+l}B'\ar[d]^-{b} \\ A^{\fS_{k+l}}_r\ar[r]^-{f^{\fS_{k+l}}} & B^{\fS_{k+l}}_r}\]
	By the special case we have already proven, we can extend $a$ and $b$ to a morphism of triangles $\triangle_{J^{k'+l}(f'),-l}\to \triangle_{f^{\fS_{k+l}},-l}$. Then the composite
	\[\xymatrix@C=6em{\triangle_{f',k'}\ar[r]^-{\mbox{\scriptsize Lemma \ref{lem:triaisoJ}}} & \triangle_{J^{k'+l}(f'),-l}\ar[r] & \triangle_{f^{\fS_{k+l}},-l}\ar[r]^-{\mbox{\scriptsize Lemma \ref{lem:triaisofS}}} & \triangle_{f,k}}\]
	is a morphism of triangles that extends the diagram of solid arrows in \eqref{eq:tr33}.
\end{proof}

\begin{ax}[TR4]\label{ax:tr4}
	Let $\alpha:X\to X'$ and $\pi':X'\to Y$ be composable morphisms in $\fk$ and put $\pi:=\pi'\circ \alpha$. Then there exist commutative diagrams as follow, where the rows and columns of the diagram on the left are distinguished triangles.
	\[\xymatrix{L^2Y\ar[d]\ar[r]^-{-L\iota'} & LZ'\ar[d]\ar[r]^-{L\varepsilon'} & LX'\ar[d]^-{\varphi}\ar[r]^-{L\pi'} & LY\ar[d] & & \\
		0\ar[d]\ar[r] & Z''\ar[d]^-{\psi}\ar[r]^-{\id} & Z''\ar[r]\ar[d] & 0\ar[d] & LX'\ar[d]^-{\varphi}\ar[r]^-{L\pi'} & LY\ar[d]^-{\iota} \\
		LY\ar[d]^-{\id}\ar[r]^-{\iota} & Z\ar[d]\ar[r]^-{\varepsilon} & X\ar[d]^-{\alpha}\ar[r]^-{\pi} & Y\ar[d]^-{\id} & Z''\ar[r]^-{\psi} & Z \\
		LY\ar[r]^-{\iota'} & Z'\ar[r]^-{\varepsilon'} & X'\ar[r]^-{\pi'} & Y & &}\]
\end{ax}
\begin{proof}
	Let $\alpha:X\to X'$ and $\pi':X'\to Y$ be composable morphisms and consider the following diagram in $\fk$:
	\begin{equation}\label{eq:diagramtr4}\xymatrix{X\ar[r]^-{\alpha} & X'\ar[r]^-{\pi'} & Y}\end{equation}
	We say that \eqref{eq:diagramtr4} satisfies (TR4) if the axiom holds for this particular pair of morphisms. The following remarks are straightforward:
	\begin{enumerate}[label=(\roman*)]
		\item\label{item:tr41} If two diagrams like \ref{eq:diagramtr4} are isomorphic, then one satisfies (TR4) if and only if the other does.
		\item\label{item:tr42} Any diagram like \eqref{eq:diagramtr4} is isomorphic to the diagram below, for some $n\in\Z$ and some morphisms $a:A\to B$ and $b:B\to C$ in $\Algl$:
		\[\xymatrix{(A,n)\ar[r]^-{L^n(a)} & (B,n)\ar[r]^-{L^n(b)} & (C,n)}\]
		(Apply Lemma \ref{lem:arrowiso} twice, once for $\alpha$ and  once for $\pi'$.)
		\item\label{item:tr43} The diagram \eqref{eq:diagramtr4} satisfies (TR4) if and only if the following diagram does:
		\[\xymatrix{LX\ar[r]^-{L\alpha} & LX'\ar[r]^-{L\pi'} & LY}\]
	\end{enumerate}
	By \ref{item:tr41}, \ref{item:tr42} and \ref{item:tr43}, we may assume that \eqref{eq:diagramtr4} is of the form
	\[\xymatrix{(A,0)\ar[r]^-{a} & (B,0)\ar[r]^-{b} & (C,0)}\]
	where $a:A\to B$ and $b:B\to C$ are morphisms in $\Algl$. We will proceed as explained in \cite{cortho}*{Axiom 6.5.7} but we will provide some more details. Put $c:=b\circ a:A\to C$. We will use the identifications in Lemma \ref{lem:tr2} so that we have, for example:
	\[(PC)_0=(t-1)C[t]\]
	\[C^{\fS_1}_0=(t^2-t)C[t]\]
	\[(P_b)_0=\{(p(t),y)\in C[t]\times B: p(0)=b(y)\mbox{ and }p(1)=0\}\]
	\[(P_c)_0=\{(q(t),z)\in C[t]\times A: q(0)=c(z)\mbox{ and }q(1)=0\}\]
	Recall from \eqref{eq:mappath} and \eqref{eq:defiiotaf} the definitions of $\pi_b:(P_b)_0\to C$ and $\iota_b:C^{\fS_1}_0\to (P_b)_0$. For example, the morphism $\pi_b:(P_b)_0\to B$ is defined by $\pi_b(p(t),y)=y$. The morphism $\eta:(P_c)_0\to (P_b)_0$, $\eta(q(t),z)=(q(t),a(z))$,  makes the following diagram in $\Algl$ commute:
	\[\xymatrix{C^{\fS_1}_0\ar[r]^-{\iota_c}\ar[d]^-{\id} & (P_c)_0\ar[d]^-{\eta}\ar[r]^-{\pi_c} & A\ar[d]^-{a}\ar[r]^-{c} & C\ar[d]^-{\id} \\
		C^{\fS_1}_0\ar[r]^-{\iota_b} & (P_b)_0\ar[r]^-{\pi_b} & B\ar[r]^-{b} & C}\]
	By functoriality of the mapping path construction, there is a morphism $\theta$ making the following diagram commute:
	\begin{equation}\label{eq:tr4cat}\begin{gathered}\xymatrix{& ((P_b)_0)_0^{\fS_1}\ar[r]^-{(\pi_b)^{\fS_1}}\ar[d]^-{\iota_\eta} & B^{\fS_1}_0\ar[d]^-{\iota_a} & \\
		& (P_\eta)_0\ar[d]^-{\pi_\eta}\ar[r]^-{\theta} & (P_a)_0\ar[d]^-{\pi_a} & \\
		C^{\fS_1}_0\ar[r]^-{\iota_c}\ar[d]^-{\id} & (P_c)_0\ar[d]^-{\eta}\ar[r]^-{\pi_c} & A\ar[d]^-{a}\ar[r]^-{c} & C\ar[d]^-{\id} \\
		C^{\fS_1}_0\ar[r]^-{\iota_b} & (P_b)_0\ar[r]^-{\pi_b} & B\ar[r]^-{b} & C}\end{gathered}\end{equation}
	We claim that $\theta$ is a $\fk$-equivalence; indeed, it is a split surjection with contractible kernel, as we proceed to explain. We have:
	\[(P(P_b)_0)_0=\left\{(p(t,s),y(s))\in C[t,s]\times B[s] : \begin{array}{c}p(0,s)=b(y)(s),\hspace{0.3em} p(1,s)=0,\\
	p(t,1)=0\mbox{ and }y(1)=0\end{array}\right\}\]
	\[(P_\eta)_0=\{(p(t,s),y(s),q(t),z)\in (P(P_b)_0)_0\times (P_c)_0: (p(t,0),y(0))=(q(t),a(z))\}\]
	In the description of $(P_\eta)_0$ above, $(p(t,s),y(s),q(t),z)$ satisfies $q(t)=p(t,0)$ so that we can get rid of $q$ as long as we keep $p$. Hence, we have:
	\[(P_\eta)_0=\left\{(p(t,s),y(s),z)\in C[t,s]\times B[s]\times A : \begin{array}{c}p(0,s)=b(y)(s),\\
	p(1,s)=0,\hspace{0.3em} p(t,1)=0,\\
	y(1)=0,\hspace{0.3em} p(0,0)=c(z)\\
	\mbox{and }y(0)=a(z)\end{array}\right\}\]
	It is easily seen that, using the latter description of $(P_\eta)_0$, the morphism $\theta:(P_\eta)_0\to (P_a)_0$ is given by $\theta(p(t,s),y(s),z)=(y(t),z)$. We have:
	\begin{align*}\ker\theta&=\{(p(t,s),0,0)\in(P_\eta)_0\}\\
	&\cong\{p(t,s)\in C[t,s]: p(0,s)=0, p(1,s)=0\mbox{ and }p(t,1)=0\}\end{align*}
	It is easily verified that $\ker\theta$ is contractible. Moreover, $\theta$ is a split surjection with section:
	\[(P_a)_0\ni(y(t),z)\mapsto (b(y)(1-(1-s)(1-t)),y(s),z)\in(P_\eta)_0\]

	Upon applying $j$ to \eqref{eq:tr4cat} and identifying $(B^{\fS_1}_0,0)\cong (B,1)$, we get the following diagram in $\fk$ whose rows and columns are mapping path triangles; the diagram clearly commutes, except maybe for the squares $*$ and $\star$:
	\[\xymatrix@C=4em{(C,2)\ar[d]\ar[r]^-{-\partial_{b,0}}\ar@{}[dr]|-{\displaystyle *} & ((P_b)_0,1)\ar[r]^-{L(\pi_b)}\ar[d]^-{\partial_{\eta,0}} & (B,1)\ar[d]|-{\theta^{-1}\circ \partial_{a,0}}\ar[r]^-{L(b)} & (C,1)\ar[d]\\
		0\ar[d]\ar[r] & ((P_\eta)_0,0)\ar[d]^-{\pi_\eta}\ar[r]^-{\id} & ((P_\eta)_0,0)\ar[d]|-{\pi_a\circ \theta}\ar[r]\ar@{}[dr]|-{\displaystyle\star} & 0\ar[d]\\
		(C,1)\ar[r]^-{\partial_{c,0}}\ar[d]^-{\id} & ((P_c)_0,0)\ar[d]^-{\eta}\ar[r]^-{\pi_c} & (A,0)\ar[d]^-{a}\ar[r]^-{c} & (C,0)\ar[d]^-{\id} \\
		(C,1)\ar[r]^-{\partial_{b,0}} & ((P_b)_0,0)\ar[r]^-{\pi_b} & (B,0)\ar[r]^-{b} & (C,0)}\]
	The composite $c\circ\pi_a:(P_a)_0\to C$ is easily seen to be nullhomotopic, so that the square $\star$ commutes. The composite
	\[\xymatrix{(C^{\fS_1}_0)^{\fS_1}_0\ar[r]^-{(\iota_b)^{\fS_1}} & ((P_b)_0)^{\fS_1}_0\ar[r]^-{\iota_\eta} & (P_\eta)_0}\]
	is easily seen to factor through $\ker\theta$, which is contractible; this implies that the square $*$ commutes too.

	We still have to show that the following diagram commutes:
	\[\xymatrix{(B,1)\ar[d]_-{\theta^{-1}\circ \partial_{a,0}}\ar[r]^-{L(b)} & (C,1)\ar[d]^-{\partial_{c,0}}\\
		((P_\eta)_0,0)\ar[r]^-{\pi_\eta} & ((P_c)_0,0)}\]
	It is easily seen that the commutativity of the square above is implied by the commutativity of the following diagram in $[\Algl]$:
	\begin{equation}\label{eq:tr4smallsqbis}\begin{gathered}\xymatrix@C=4em@R=1.2em{B^{\fS_1}_0\ar[d]_-{\iota_a}\ar[r]^-{b^{\fS_1}} & C^{\fS_1}_0\ar[dd]^-{\iota_c} \\
		(P_a)_0\ar@{}[ur]\ar[dr]^-{\xi} & \\
		(P_\eta)_0\ar[u]^-{\theta}\ar[r]_-{\pi_\eta} & (P_c)_0}\end{gathered}\end{equation}
	Here the morphism $\xi$ is given by $\xi(y(t),z)=(b(y)(t),z)$. The square in \eqref{eq:tr4smallsqbis} commutes on the nose. The triangle in \eqref{eq:tr4smallsqbis} commutes in $[\Algl]$, as we proceed to explain. Consider the following elementary homotopies $H_1,H_2:(P_\eta)_0\to (P_c)_0[u]$:
	\[H_1(p(t,s),y(s),z)=(p(tu,t),z)\]
	\[H_2(p(t,s),y(s),z)=(p(t,tu),z)\]
	Then $\ev_{u=0}\circ H_1=\xi\circ\theta$, $\ev_{u=1}\circ H_1=\ev_{u=1}\circ H_2$ and $\ev_{u=0}\circ H_2=\pi_\eta$, showing that $\xi\circ\theta=\pi_\eta$ in $[\Algl]$. This finishes the proof of (TR4).
\end{proof}

We showed that $\fk$ is a triangulated category with the distinguished triangles being those triangles isomorphic to mapping path triangles. As in the topological setting \cite{ralf}*{Section 6.6}, the distinguished triangles could also be defined using extension triangles; we proceed to give the details of this.

\begin{defi}\label{defi:extensiontriangles}
	Let $\scrE: A\overset{f}\to B\overset{g}\to C$ be an extension in $\Algl$ with classifying map $\xi:JC\to A$. Let $n\in\Z$ and let $\partial_{\scrE,n}$ be the composite:
	\[\xymatrix@C=5em{(C,n+1)\ar[r]^-{\langle\id_{JC}\rangle} & (JC,n)\ar[r]^-{(-1)^nL^n(\xi)} & (A,n)}\]
	We call \emph{extension triangle} to a diagram in $\fk$ of the form:
	\[\triangle_{\scrE,n}:\xymatrix@C=3em{L(C,n)\ar[r]^-{\partial_{\scrE,n}} & (A,n)\ar[r]^-{L^n(f)} & (B,n)\ar[r]^-{L^n(g)} & (C,n)}\]
\end{defi}

\begin{prop}[\cite{ralf}*{Section 6.6}]\label{prop:trianglesinkk}
	A triangle in $\fk$ is distinguished if and only if it is isomorphic to an extension triangle.
\end{prop}
\begin{proof}
	Let us show first that every mapping path triangle is isomorphic to an extension triangle. Let $g:B\to C$ be any morphism in $\Algl$. Consider the mapping cylinder:
	\[Z_g:=\{(p,b)\in C[t]\times B: p(0)=g(b)\}\]
	Using the identifications in Lemma \ref{lem:tr2}, we have:
	\[(P_g)_0:=\{(p,b)\in (t-1)C[t]\times B: p(0)=g(b)\}\]
	It is easily verified that the following diagram is an extension in $\Algl$:
	\[\scrZ_g:\xymatrix@R=-0.3em{(P_g)_0\ar[r]^-{\inc} & Z_g\ar[r]^-{\varepsilon} & C\\
		& (p,b)\ar@{|->}[r] & p(1)}\]
	Let $\pr:Z_g\to B$ be the natural projection; $\pr$ is easily seen to be a homotopy equivalence inverse to $b\mapsto (g(b),b)$. We claim that there is an isomorphism of triangles as follows:
	\[\xymatrix{\triangle_{\scrZ_g,0}\ar[d]_-{\cong} & (B,1)\ar[d]_-{\id}\ar[r]^-{\partial_{\scrZ_g,0}} & ((P_g)_0,0)\ar[d]_{\id}\ar[r]^-{\inc} & (Z_g,0)\ar[d]_-{\cong}^-{\pr}\ar[r]^-{\varepsilon} & (C,0)\ar[d]^-{\id} \\
		\triangle_{g,0} & (B,1)\ar[r]^-{\partial_{g,0}} & ((P_g)_0,0)\ar[r]^-{\pi_g} & (B,0)\ar[r]^-{g} & (C,0)}\]
	The middle and right squares clearly commute but we still have to show that $\partial_{\scrZ_g,0}=\partial_{g,0}$. Let $\omega:C^{\fS_1}_0\to C^{\fS_1}_0$ be the automorphism defined by $\omega(p(t))=p(1-t)$. Consider the following morphism of extensions, where the vertical map in the middle is defined by $p(t)\mapsto (p(1-t),0)$:
	\[\xymatrix{\scrP_{0,C}\ar[d] & C^{\fS_1}_0\ar[d]_-{\iota_g\circ \omega}\ar[r] & (PC)_0\ar[d]\ar[r] & C\ar[d]^-{\id} \\
		\scrZ_g &(P_g)_0\ar[r] & Z_g\ar[r] & C}\]
	By Proposition \ref{lem:classcommute}, the classifying map of $\scrZ_g$ equals $\iota_g\circ\omega\circ\lambda_C$; this is easily seen to imply that $\partial_{\scrZ_g,0}=\partial_{g,0}$.

	Let us now show that every extension triangle is isomorphic to a mapping path triangle. Let $\scrE:A\overset{f}\to B\overset{g}\to C$ be an extension in $\Algl$. Let $h:A\to (P_g)_0$ be the natural morphism, that is a $\fk$-equivalence by Lemma \ref{lem:kerkkequiv}. We claim that there is an isomorphism of triangles:
	\[\xymatrix{\triangle_{\scrE,0}\ar[d]_-{\cong} & (C,1)\ar[d]_-{\id}\ar[r]^-{\partial_{\scrE,0}} & (A,0)\ar[d]_-{\cong}^-{h}\ar[r]^-{f} & (B,0)\ar[d]^-{\id}\ar[r]^-{g} & (C,0)\ar[d]^-{\id} \\
		\triangle_{g,0} & (C,1)\ar[r]^-{\partial_{g,0}} & ((P_g)_0,0)\ar[r]^-{\pi_g} & (B,0)\ar[r]^-{g} & (C,0)}\]
	The middle and right squares clearly commute but we still have to show that $j(h)\circ\partial_{\scrE,0}=\partial_{g,0}$. Above we proved that the classifying map of $\scrZ_g$ equals $\iota_g\circ\omega\circ\lambda_C$ in $[\Algl]$. Now let $\xi:JC\to A$ be the classifying map of $\scrE$ and consider the following morphism of extensions, where the middle vertical map is $b\mapsto (g(b),b)$:
	\[\xymatrix{\scrE:\ar[d] & A\ar[d]_-{h}\ar[r] & B\ar[d]\ar[r] & C\ar[d]^-{\id} \\
		\scrZ_g: &(P_g)_0\ar[r]^-{\inc} & Z_g\ar[r]^-{\varepsilon} & C}\]
	By Proposition \ref{lem:classcommute}, the classifying map of $\scrZ_g$ equals $h\circ\xi$ in $[\Algl]$. Then $\iota_g\circ\omega\circ\lambda_C=h\circ\xi$ in $[\Algl]$, and this is easily seen to imply that $j(h)\circ\partial_{\scrE,0}=\partial_{g,0}$.
\end{proof}

\section{Universal property}\label{sec:uniprop} Recall from \cite{cortho}*{Section 6.6} the definition of an excisive homology theory with values in a triangulated category.

\begin{defi}
	Let $(\tria,L)$ be a triangulated category. An \emph{excisive homology theory} with values in $\tria$ consists of the following data:
	\begin{enumerate}[label=(\roman*)]
		\item a functor $X:\Algl\to \tria$;
		\item a morphism $\delta_\scrE\in \Hom_{\tria}(LX(C),X(A))$ for every extension $\scrE$ in $\Algl$:
		\begin{equation}\label{eq:extuni}\scrE:\xymatrix{A\ar[r]^-{f} & B\ar[r]^-{g} & C}\end{equation}
	\end{enumerate}
	These morphisms $\delta_\scrE$ are subject to the following conditions:
	\begin{enumerate}[label=(\alph*)]
		\item For every extension \eqref{eq:extuni}, the following triangle is distinguished:
		\[\triangle_\scrE:\xymatrix{LX(C)\ar[r]^-{\delta_\scrE} & X(A)\ar[r]^-{X(f)} & X(B)\ar[r]^-{X(g)} & X(C)}\]
		\item The triangles $\triangle_\scrE$ are natural with respect to morphisms of extensions.
	\end{enumerate}
\end{defi}

\begin{exa}\label{exa:kkuni}
	Let $\scrE: A\overset{f}\to B\overset{g}\to C$ be an extension in $\Algl$. Recall from Proposition \ref{prop:trianglesinkk} that we have a distinguished triangle in $\fk$:
	\[\triangle_{\scrE,0}:\xymatrix{(C,1)\ar[r]^-{\partial_{\scrE,0}} & (A,0)\ar[r]^-{j(f)} & (B,0)\ar[r]^-{j(g)} & (C,0)}\]
	Moreover, it follows from Proposition \ref{lem:classcommute} that $\partial_{\scrE,0}$ is natural with respect to morphisms of extensions. Then the functor $j:\Algl\to\fk$ together with the morphisms $\partial_{\scrE,0}$ is an excisive homology theory.
\end{exa}

A \emph{graded category} is a pair $(\gra,L)$ where $\gra$ is an additive category and $L$ is an automorphism of $\gra$. It $(\gra,L)$ is a graded category and $X$ is an object of $\gra$, we will often write $(X,n)$ instead of $L^n(X)$. A \emph{graded functor} $F:(\gra,L)\to (\gra',L')$ is an additive functor $F:\gra\to\gra'$ such that $F\circ L=L'\circ F$. Let $F,G:(\gra,L)\to (\gra',L')$ be graded functors. A \emph{graded natural transformation} $\nu:F\to G$ is a natural transformation $\nu$ such that $L'(\nu_X)=\nu_{L(X)}:L'F(X)\to L'G(X)$ for all $X\in\gra$.

\begin{exa}
	A triangulated category is a graded category.
\end{exa}

\begin{exa}\label{exa:gradedtriaI}
	Let $(\tria,L)$ be a triangulated category. Put $\gra:=\tria^I$ where $I=\{0\to 1\}$ is the interval category; then $\gra$ is an additive category. It is easily seen that $L$ induces a translation functor in $\gra$ that makes $\gra$ into a graded category.
\end{exa}

\begin{exa}\label{exa:GrAb}
	Let $\GrAb$ be the category whose objects are $\Z$-graded abelian groups and whose morphisms graded morphisms of degree zero. Then $\gra=\GrAb$ is a graded category with the translation functor $L$ defined by $L(M)_n=M_{n+1}$, $n\in\Z$, $M\in\GrAb$.
\end{exa}

\begin{defi}
	Let $(\gra,L)$ be a graded category. A \emph{$\delta$-functor} with values in $\gra$ consists of the following data:
	\begin{enumerate}[label=(\roman*)]
		\item a functor $X:\Algl\to \gra$ that preserves finite products;
		\item a morphism $\delta_\scrE\in \Hom_{\gra}(LX(C),X(A))$ for every extension $\scrE$ in $\Algl$:
		\[\scrE:\xymatrix{A\ar[r] & B\ar[r] & C}\]
	\end{enumerate}
	These morphisms $\delta_\scrE$ are subject to the following conditions:
	\begin{enumerate}[label=(\alph*)]
		\item $\delta_\scrE:LX(C)\to X(A)$ is an isomorphism if $X(B)=0$;
		\item The morphisms $\delta_\scrE$ are natural with respect to morphisms of extensions.
	\end{enumerate}
\end{defi}

\begin{exa}
	An excisive homology theory $X:\Algl\to\tria$ is a $\delta$-functor.
\end{exa}

\begin{exa}\label{exa:inducedprehomo}
	Let $X,Y:\Algl\to\tria$ be excisive homology theories and let $\nu:X\to Y$ be a natural transformation such that, for every extension \eqref{eq:extuni}, the following diagram commutes:
	\begin{equation}\label{eq:inducedprehomo}\begin{gathered}\xymatrix{LX(C)\ar[r]^-{\delta^X_\cE}\ar[d]_-{L(\nu_C)} & X(A)\ar[d]^-{\nu_A} \\
		LY(C)\ar[r]^-{\delta^Y_\cE} & Y(A)}\end{gathered}\end{equation}
	Let $\gra=\tria^I$ be the graded category of Example \ref{exa:gradedtriaI}. Then the natural transformation $\nu$ induces a $\delta$-functor $\Algl\to\gra$, that we still denote $\nu$. Explicitely, the functor $\nu:\Algl\to\gra$ is defined as follows:
	\[A\mapsto (\nu_A:X(A)\to Y(A))\]
	\[f\in\Hom_\Algl(A,B)\mapsto (X(f),Y(f))\in\Hom_\gra(\nu_A,\nu_B)\]
	For an extension \eqref{eq:extuni}, the morphism $\delta_\scrE\in\Hom_\gra(L(\nu_C),\nu_A)$ is defined by:
	\[\delta_\scrE:=(\delta^X_\scrE,\delta^Y_\scrE)\in \Hom_\gra(L(\nu_C),\nu_A)\]
\end{exa}

We will show that $j:\Algl\to\fk$ is the universal excisive and homotopy invariant homology theory, in the sense of \cite{cortho}*{Section 6.6}. In order to deal with natural transformations, we will work in the slightly more general setting of $\delta$-functors. From now on, fix a homotopy invariant $\delta$-functor $X$ with values in a graded category $(\gra,L)$. A morphism in $\Algl$ will be called an \emph{$X$-equivalence} if it becomes invertible upon applying $X$. For example, the following morphisms are $X$-equivalences:
\begin{enumerate}[label=\arabic*)]
	\item The morphisms $B^{\fS_n}_r\to B^{\fS_n}_{r+1}$ for any $B\in\Algl$, $n\in\Zo$ and $r\geq 0$. This follows by induction on $n$. For the inductive step, apply $X$ to the morphism of extensions \eqref{eq:lastvertexkkequiv} and use that $P(n,B)_r$ is contractible.
	\item The morphisms $\mu^{m,n}:(B^{\fS_m}_r)^{\fS_n}_s\to B^{\fS_{m+n}}_{r+s}$ for $B\in\Algl$ and $m,n,r,s\geq 0$. We will only use this fact in the special case $n=1$ and $r=s=0$. Proceed by induction on $m$. For the inductive step, apply $X$ to the morphism of extensions \eqref{eq:mukkequiv} and then use that the middle terms are contractible.
\end{enumerate}

Since $X$ is homotopy invariant, $X$ induces a functor $[\Algl]\to \gra$ that we still denote $X$. Let $[\Algl]^\ind_X$ be the full subcategory of $[\Algl]^\ind$ whose objects are ind-objects $(A,I)$ such that:
\begin{enumerate}[label=(\roman*)]
	\item $I$ has an initial object $i_0$;
	\item all the transition morphisms $A_i\to A_j$ are $X$-equivalences.
\end{enumerate}
Note that, for any $B\in\Algl$, the ind-object $B^{\fS_n}_\bul$ is in $[\Algl]^\ind_X$.

It is easily verified that $[\Algl]$ has finite products and that the functor $\Algl\to[\Algl]$ commutes with finite products. Then $[\Algl]^\ind$ has finite products too; explicitely, the product of $(A_\bul, I)$ and $(B_\bul, J)$ is the object $(A_\bul\times B_\bul , I\times J)$ with the obvious projections. Since $X$ is a $\delta$-functor, we have natural isomorphisms:
\[X(A_i\times B_j)\cong X(A_i)\oplus X(B_j)\]
Using this, it is easily seen that the product of two objects of $[\Algl]^\ind_X$ is again an object of $[\Algl]^\ind_X$. This shows that $[\Algl]^\ind_X$ has finite products.

Let $(A,I), (B,J)\in[\Algl]^\ind$. A morphism $f\in[A_\bul, B_\bul]$ is a collection $\{[f_i]\}_{i\in I}$ of homotopy classes of morphisms $f_i:A_i\to B_{\theta(i)}$ subject to certain compatibility relations.

\begin{lem}\label{lem:xind}
	There is a functor $\tilde{X}:[\Algl]^\ind_X\to \gra$ such that $\tilde{X}(A,I)=X(A_{i_0})$ and such that $\tilde{X}(f)$ is the composite
	\[\xymatrix{X(A_{i_0})\ar[r]^-{X(f_{i_0})} & X(B_{\theta(i_0)}) & X(B_{j_0})\ar[l]_-{\cong} }\]
	for any $f\in[A_\bul, B_\bul]$. Moreover, $\tilde{X}$ preserves finite products.
\end{lem}
\begin{proof}
	It is easily verified that $\tilde{X}$ is indeed a well-defined functor. The fact that $\tilde{X}$ preserves finite products follows from the fact that $X:[\Algl]\to\gra$ does.\end{proof}

If $B\in\Algl$ and $n\geq 1$, then $B^{\fS_n}_\bul$ is a group object in $[\Algl]^\ind_X$; see Lemma \ref{lem:groupstr}. It follows from Lemma \ref{lem:xind} that $X(B^{\fS_n}_0)$ has a group object structure induced by that of $B^{\fS_n}_\bul$. Since $\gra$ is an additive category, every object of $\gra$ is naturally an abelian group object. Thus, $X(B^{\fS_n}_0)$ has two group object structures: the one coming from $B^{\fS_n}_\bul$ and the other from being an object of $\gra$. By the Eckmann-Hilton argument, both group structures coincide and the function
\[\xymatrix{\tilde{X}:[A_\bul,B^{\fS_n}_\bul]\ar[r] & \Hom_\gra(X(A_{i_0}),X(B^{\fS_n}_0))}\]
is a group homomorphism for every $(A,I)$ in $[\Algl]^\ind_X$.

Let $A\in\Algl$ and let $\scrU_A$ be the universal extension of $A$. Since $TA$ is contractible, there is an isomorphism:
\[\xymatrix{\delta_{\scrU_A}:(X(A),1)\ar[r]^-{\cong} & (X(JA),0)}\]
Put $i^{J,1}_A:=\delta_{\scrU_A}$ and define inductively $i^{J,n+1}_A$ as the composite:
\[\xymatrix@C=4em{(X(A),n+1)\ar[r]^-{L(i^{J,n}_A)} & (X(J^nA),1)\ar[r]^-{i^{J,1}_{J^nA}} & (X(J^{n+1}A),0)}\]
Let $i^{J,0}_A$ be the identity of $(X(A),0)$. It is easily verified by induction on $n=p+q$ that the following equality holds for $p,q\geq 0$:
\[i^{J,p+q}_A=i^{J,q}_{J^pA}\circ L^q(i^{J,p}_A)\]
The morphisms $i^{J,n}_?$ assemble into a natural isomorphism $L^n\circ X\cong X\circ J^n(?):\Algl\to\gra$.

Let $A\in \Algl$ and let $\scrP_{0,A}$ be the path extension of $A$. Since $(PA)_0$ is contractible, there is an isomorphism:
\[\xymatrix{\delta_{\scrP_A}:(X(A),1)\ar[r]^-{\cong} & (X(A^{\fS_1}_0),0)}\]
Put $i^{\fS,1}_A:=\delta_{\scrP_A}$ and define inductively $i^{\fS,n+1}_A$ as the composite:
\[\xymatrix@C=2.5em{(X(A),n+1)\ar[r]^-{L(i^{\fS,n}_A)} & (X(A^{\fS_n}_0),1)\ar[r]^-{i^{\fS,1}_{A^{\fS_n}_0}} & (X((A^{\fS_n}_0)^{\fS_1}_0),0)\ar[r]_-{\cong}^-{X(\mu^{n,1}_A)} & (X(A^{\fS_{n+1}}_0),0)}\]
Let $i^{\fS,0}_A$ be the identity of $(X(A),0)$. It is easily verified by induction on $n=p+q$ that the following equality holds for $p,q\geq 0$:
\[i^{\fS,p+q}_A=X(\mu^{p,q}_A)\circ i^{\fS,q}_{A^{\fS_p}_0}\circ L^q(i^{\fS,p}_A)\]
The morphisms $i^{\fS,n}_?$ assemble into a natural isomorphism $L^n\circ X\cong X\circ (?)^{\fS_n}_0:\Algl\to\gra$.

\begin{lem}\label{lem:signT}
	Let $N_2, N_3\geq 0$ and let $B\in\Algl$. Then the following diagram in $\gra$ commutes up to the sign $(-1)^{N_2N_3}$:
	\[\xymatrix@C=4em{(X(B),N_2+N_3)\ar[r]^-{L^{N_3}(i^{\fS,N_2}_B)}\ar[d]_-{L^{N_2}(i^{J,N_3}_B)} & (X(B^{\fS_{N_2}}_0),N_3)\ar[r]^-{i^{J,N_3}_{B^{\fS_{N_2}}}} & (X(J^{N_3}(B^{\fS_{N_2}}_0)),0)\ar[d]^-{X(\kappa^{N_3,N_2}_B)} \\
		(X(J^{N_3}B),N_2)\ar[rr]^-{i^{\fS,N_2}_{J^{N_3}B}} & & (X((J^{N_3}B)^{\fS_{N_2}}_0),0)}\]
\end{lem}
\begin{proof}
	If $N_2=0$ or $N_3=0$ there is nothing to prove. Upon applying the functor $\tilde{X}$ to the diagram \eqref{eq:clascon} we get that the followig diagram in $\gra$ commutes up to the sign $-1$; the case $N_2=N_3=1$ follows easily from this:
	\[\xymatrix@C=4em{X(J^2B)\ar[r]^-{X(J(\lambda_B))}\ar@/_1pc/[dr]_-{X(\lambda_{JB})} & X(J(B^{\fS_1}_0))\ar[d]^-{X(\kappa^{1,1}_B)} \\
		& X((JB)^{\fS_1}_0) }\]
	Once we know the case $N_2=N_3=1$, the case $N_3=1$ with arbitrary $N_2$ follows by an easy induction on $N_2$. Once we know the result for $N_3=1$ and arbitrary $N_2$, the general case follows by induction on $N_3$.
\end{proof}

\begin{thm}\label{thm:unigra}
	Let $(\gra,L)$ be a graded category and let $X:\Algl\to\gra$ be a homotopy invariant $\delta$-functor. Then there exists a unique graded functor $\bar{X}:\fk\to\gra$ such that $\bar{X}(\partial_{\scrE,0})=\delta_\scrE$ for every extension $\scrE$ and such that the following diagram commutes:
	\[\xymatrix@C=4em{\Algl\ar[r]^{j}\ar@/_/[dr]_-{X} & \fk\ar@{..>}[d]^-{\exists ! \bar{X}} \\
		& \gra}\]
\end{thm}
\begin{proof}
	Define $\bar{X}$ on objects by $\bar{X}(A,m):=(X(A),m)$. To define $\bar{X}$ on morphisms we must define, for every pair of objects of $\fk$, a group homomorphism:
	\begin{equation}\label{eq:Xbarmor}\bar{X}_{(A,m),(B,n)}:\xymatrix{\fk((A,m),(B,n))\ar[r] & \Hom_\gra((X(A),m),(X(B),n))}\end{equation}
	Let $\bar{X}^v$ be the dotted composite:
	\[\xymatrix{[J^{m+v}A,B^{\fS_{n+v}}_\bul]\ar@{..>}@/_3pc/[ddr]_-{\bar{X}^v}\ar[r]^-{\tilde{X}} & \Hom_\gra((X(J^{m+v}A),0),(X(B^{\fS_{n+v}}_0),0))\ar[d]_-{\cong}^-{(i^{\fS,n+v}_B)^{-1}\circ (?)\circ i^{J,m+v}_A} \\
		& \Hom_\gra((X(A),m+v),(X(B),n+v))\ar[d]_-{\cong}^-{L^{-v}} \\
		& \Hom_\gra((X(A),m),(X(B),n))}\]
	The function $\bar{X}^v$ is a group homomorphism by the discussion following Lemma \ref{lem:xind}. Moreover, it is easily verified that this  diagram commutes:
	\[\xymatrix{[J^{m+v}A,B^{\fS_{n+v}}_\bul]\ar[r]^-{\bar{X}^v}\ar[d]_-{\Lambda^{n+v}} & \Hom_{\gra}((X(A),m),(X(B),n)) \\
		[J^{m+v+1}A,B^{\fS_{n+v+1}}_\bul]\ar@/_/[ur]_-{\bar{X}^{v+1}} &}\]
	Thus, the morphisms $\bar{X}^v$ induce the desired group homomorphism $\bar{X}_{(A,m),(B,n)}$ in \eqref{eq:Xbarmor}; this defines $\bar{X}$ on morphisms. It is straightforward but tedious to verify that the definitions above indeed give rise to an additive functor $\bar{X}:\fk\to\gra$. When verifying that $\tilde{X}$ preserves composition, Lemma \ref{lem:signT} is needed to show that the signs in Definition \ref{defi:kkc} work out. We clearly have $X=\bar{X}\circ j$ and $L\circ \bar{X}=\bar{X}\circ L$.

	Let us now show that $\bar{X}(\partial_{\scrE,0})=\delta_\scrE$ for every extension $\scrE$ in $\Algl$. Consider an extension as follows, with classifying map $\xi:JC\to A$:
	\[\scrE:\xymatrix{A\ar[r]^-{f} & B\ar[r]^-{g} & C}\]
	Recall from Definition \eqref{defi:extensiontriangles} that $\partial_{\scrE,0}$ equals the composite:
	\[\xymatrix@C=3em{(C,1)\ar[r]^-{\langle\id_{JC}\rangle} & (JC,0)\ar[r]^-{\xi} & (A,0)}\]
	Upon applying $\bar{X}$ we get:
	\[\xymatrix@C=3em{(X(C),1)\ar[r]^-{i^{J,1}_C} & (X(JC),0)\ar[r]^-{X(\xi)} & (X(A),0)}\]
	By naturality of $\delta$, we have:
	\[\bar{X}(\partial_{\scrE,0})=X(\xi)\circ i^{J,1}_C=X(\xi)\circ\delta_{\scrU_C}=\delta_{\scrE}\]

	It remains to check the uniqueness of $\bar{X}$. Let $\bar{X}:\fk\to \gra$ be any graded functor with the properties described in the statement of this theorem. Let $\alpha\in\fk((A,m),(B,n))$ be represented by $f:J^{m+v}A\to B^{\fS_{n+v}}_r$ and let $\gamma^r:B^{\fS_{n+v}}_0\to B^{\fS_{n+v}}_r$ be the morphism induced by the iterated last vertex map. By Lemma \ref{lem:arrowiso}, the following diagram in $\fk$ commutes:
	\[\xymatrix@C=5em{(A,m)\ar[d]_-{\cong}\ar[r]^-{\alpha} & (B,n) & \\
		(J^{m+v}A,-v)\ar[r]^-{L^{-v}(f)} & (B^{\fS_{n+v}}_r,-v)\ar[u]_-{\cong} & (B^{\fS_{n+v}}_0,-v)\ar@/_/[ul]_-{\cong}\ar[l]^-{\cong}_-{L^{-v}(\gamma^r)}}\]
	Upon applying $\bar{X}$ we get the following commutative diagram in $\gra$:
	\[\xymatrix@C=5em{(X(A),m)\ar[d]_-{L^{-v}(i^{J,m+v}_A)}^-{\cong}\ar[r]^-{X(\alpha)} & (X(B),n) & \\
		(X(J^{m+v}A),-v)\ar[r]^-{L^{-v}X(f)} & (X(B^{\fS_{n+v}}_r),-v)\ar[u]_-{\cong} & (X(B^{\fS_{n+v}}_0),-v)\ar@/_/[ul]_-{L^{-v}(i^{\fS,n+v}_B)}^-{\cong}\ar[l]^-{\cong}_-{L^{-v}X(\gamma^r)}}\]
	It follows that $\bar{X}$ is the functor defined above.
\end{proof}

As a corollary we get:

\begin{thm}[\cite{garku}*{Comparison Theorem B}]\label{thm:unitria}
	Let $(\tria,L)$ be a triangulated category and let $X:\Algl\to\tria$ be an excisive and homotopy invariant homology theory. Then there exists a unique triangulated functor $\bar{X}:\fk\to\tria$ such that $\bar{X}(\partial_\scrE)=\delta_\scrE$ for every extension $\scrE$, and such that the following diagram commutes:
	\[\xymatrix@C=4em{\Algl\ar[r]^{j}\ar@/_/[dr]_-{X} & \fk\ar@{..>}[d]^-{\exists ! \bar{X}} \\
		& \tria}\]
\end{thm}
\begin{proof}
	By Theorem \ref{thm:unigra}, there exists a unique graded functor $\bar{X}$ making the diagram commute and such that $\bar{X}(\partial_{\scrE,0})=\delta_\scrE$ for every extension $\scrE$ in $\Algl$. It remains to check that $\bar{X}$ sends distinguished triangles in $\fk$ to distinguished triangles in $\tria$. By Proposition \ref{prop:trianglesinkk} it suffices to show that $\bar{X}$ sends extension triangles $\triangle_{\scrE,0}$ to distinguished triangles in $\tria$, but this follows immediately from the fact that $\bar{X}(\partial_{\scrE,0})=\delta_\scrE$.
\end{proof}

\begin{rem}
	One way to summarize Theorem \ref{thm:unitria} is to say that $j:\Algl\to\fk$ is the universal excisive and homotopy invariant homology theory with values in a triangulated category. As explained in the introduction, such a theory was already constructed by Garkusha \cite{garkuuni}*{Theorem 2.6 (2)} using completely different methods. Both constructions are, of course, equivalent since they satisfy the same universal property.
\end{rem}

\begin{thm}\label{thm:deriving}
	Let $F:\Algl\to \Algl$ be a functor satisfying the following properties:
	\begin{enumerate}[label=\arabic*)]
		\item $F$ preserves homotopic morphisms;
		\item $F$ preserves extensions.
	\end{enumerate}
	Then there exists a unique triangulated functor $\bar{F}$ making the following diagram commute:
	\[\xymatrix{\Algl\ar[r]^{j}\ar[d]_-{F} & \fk\ar@{..>}[d]^-{\exists !\bar{F}} \\
		\Algl\ar[r]^-{j} & \fk}\]
	Let $F_1,F_2:\Algl\to\Algl$ be two functors with the properties above and let $\eta:F_1\to F_2$ be a natural transformation. Then there exists a unique (graded) natural transformation $\bar{\eta}:\bar{F}_1\to\bar{F}_2$ such that $\bar{\eta}_{j(A)}= j(\eta_A)$ for all $A\in\Algl$.
\end{thm}
\begin{proof}
	The existence and uniqueness of $\bar{F}$ follow from Theorem \ref{thm:unitria} once we notice that $j\circ F:\Algl\to\fk$ is an excisive and homotopy invariant homology theory. Let us show now the existence of $\bar{\eta}$. For every $A\in\Algl$ put:
	\[\nu_A:=j(\eta_A)\in\fk(F_1(A), F_2(A))\]
	The $\nu_A$ assemble into a natural transformation $\nu:j\circ F_1\to j\circ F_2:\Algl\to\fk$. Let $\gra:=(\fk)^I$ where $I$ is the interval category. Recall from Example \ref{exa:inducedprehomo} that $\nu$ induces a homotopy invariant $\delta$-functor $\Algl\to\gra$ if we show that the diagram \eqref{eq:inducedprehomo} commutes. Let
	\[\scrE:\xymatrix{A\ar[r]^-{f} & B\ar[r]^-{g} & C}\]
	be an extension in $\Algl$. Then we have a morphism of extensions in $\Algl$:
	\[\xymatrix{F_1(A)\ar[d]_-{\eta_A}\ar[r]^-{F_1(f)} & F_1(B)\ar[d]^-{\eta_B}\ar[r]^-{F_1(g)} & F_1(C)\ar[d]^-{\eta_C} \\
		F_2(A)\ar[r]^-{F_2(f)} & F_2(B)\ar[r]^-{F_2(g)} & F_2(C)}\]
	Since $j$ sends extensions to triangles in a natural way, the following diagram in $\fk$ commutes:
	\[\xymatrix{(F_1(C),1)\ar[d]_-{L(\nu_C)}\ar[r]^-{\delta^{j\circ F_1}} & (F_1(A),0)\ar[d]^-{\nu_A} \\
		(F_2(C),1)\ar[r]^-{\delta^{j\circ F_2}} & (F_2(A),0)}\]
	Thus, $\nu$ induces a homotopy invariant $\delta$-functor $\Algl\to\gra$, which in turn induces a graded functor $\bar{\nu}:\fk\to\gra$ by Theorem \ref{thm:unigra}. It is easily verified that this graded functor $\bar{\nu}$ corresponds to the desired natural transformation $\bar{\eta}:\bar{F}_1\to\bar{F}_2$.
\end{proof}

\begin{rem}\label{rem:derivingnat}
	Let $F,F',F'':\Algl\to\Algl$ be functors satisfying the hypothesis of Theorem \ref{thm:deriving} and let $\eta:F\to F'$ and $\eta':F'\to F''$ be natural transformations. Then $\overline{\eta'\circ\eta}=\bar{\eta'}\circ\bar{\eta}$.
\end{rem}

\appendix

\section{}\label{sec:app}

In this appendix we prove Proposition \ref{prop:bundle}.

\begin{prop}\label{prop:bundleapp}
Let $A,B\in\Algl$, let $K$ be a filtering poset, let $C_\bul\in(\Alg_\Z)^K$ and let $m,n\geq 1$. Then the following composite functions are group homomorphisms:
\begin{enumerate}[label=(\roman*)]
	\item\label{eq:Jkappa} $\xymatrix@C=3em{[A,B^{\fS_m}_\bul] \ar[r]^-{J^n} & [J^nA, J^n(B^{\fS_m}_\bul)]\ar[r]^-{\left(\kappa^{n,m}_{B}\right)_*} & [J^nA, (J^nB)^{\fS_m}_\bul]}$
	\item\label{eq:tensorgrouphomo} $\xymatrix@C=4em{[A,B^{\fS_m}_\bul]\ar[r]^-{?\otimes C_\bul} & [A\otimes C_\bul,B^{\fS_m}_\bul\otimes C_\bul]\cong[A\otimes C_\bul,(B\otimes C_\bul)^{\fS_m}_\bul]}$
	\item\label{eq:mugrouphomo} $\xymatrix@C=4em{[A,(B^{\fS_m}_\bul)^{\fS_n}_\bul]\ar[r]^-{\left(\mu^{m,n}_{B}\right)_*} & [A,B^{\fS_{m+n}}_\bul]}$
	\item\label{eq:mu2grouphomo} $\xymatrix@C=3em{[A,B^{\fS_m}_\bul]\ar[r]^-{(?)^{\fS_n}} & [A^{\fS_n}_\bul,(B^{\fS_m}_\bul)^{\fS_n}_\bul]\ar[r]^-{\left(\mu^{m,n}_{B}\right)_*} & [A^{\fS_n}_\bul,B^{\fS_{m+n}}_\bul]}$
\end{enumerate}
In \ref{eq:tensorgrouphomo}, the bijection on the right is induced by the obvious isomorphism of $K\times\Zo$-diagrams $B^{\fS_m}_\bul\otimes C_\bul\cong(B\otimes C_\bul)^{\fS_m}_\bul$.
\end{prop}
\begin{proof}
Write $\Phi^{n,m}_{A,B}$ for the composite function in \ref{eq:Jkappa}; let us prove it is a group homomorphism. We proceed by induction on $n$. In the case $n=1$, we claim that there is a morphism of simplicial sets $\varphi:\Hom_\Algl(A,B^\Delta)\to \Hom_\Algl(JA, (JB)^\Delta)$ inducing $\Phi^{1,m}_{A,B}$ under the identifications of Theorem \ref{thm:bij}. Indeed, for $f\in \Hom_\Algl(A,B^{\Delta^p})$, let $\varphi(f)\in\Hom_\Algl(JA,(JB)^{\Delta^p})$ be the classifying map of $f$ with respect to $(\scrU_B)^{\Delta^p}$. It is easily verified that the following diagram commutes, proving the case $n=1$.
\[\xymatrix@C=3em{\left(\Omega^m\Ex^\infty\Hom_\Algl(A,B^\Delta)\right)_0\ar[r]^-{\cong}\ar[d]_-{\left(\Omega^m\Ex^\infty\varphi\right)_0} & \Hom_{\Algl^\ind}(A, B^{\fS_m}_\bul)\ar[d]^-{\kappa^{1,m}_B\circ J(?)} \\
\left(\Omega^m\Ex^\infty\Hom_\Algl(JA,(JB)^\Delta)\right)_0\ar[r]^-{\cong} & \Hom_{\Algl^\ind}(A, (JB)^{\fS_m}_\bul) \\}\]
The inductive step is straightforward once we notice that $\Phi^{n+1,m}_{A,B}=\Phi^{1,m}_{J^nA, J^nB}\circ \Phi^{n,m}_{A,B}$.

Write $\Psi_{C_\bul}$ for the composite function in \ref{eq:tensorgrouphomo}; let us prove it is a group homomorphism. We can easily reduce to the case $C_\bul=C\in\Alg_\Z$. Again, we claim that there is a morphism of simplicial sets $\psi:\Hom_\Algl(A,B^\Delta)\to \Hom_\Algl(A\otimes C,(B\otimes C)^\Delta)$ inducing $\Psi_C$ upon taking $\pi_m$ and making the identifications of Theorem \ref{thm:bij}. For $p\geq 0$, let $\psi_p$ be the composite:
\[\xymatrix@C=3em{\Hom_\Algl(A,B^{\Delta^p})\ar[r]^-{?\otimes C} & \Hom_\Algl(A\otimes C,B^{\Delta^p}\otimes C)\cong\Hom_\Algl(A\otimes C,(B\otimes C)^{\Delta^p})}\]
It is straightforward to show that this defines a morphism $\psi$ with the desired properties.

Let us now prove that the function $(\mu^{m,n}_{B})_*$ in \ref{eq:mugrouphomo} is a group homomorphism. Note that
\[[A,(B^{\fS_m}_\bul)^{\fS_n}_\bul]=\colim_r\,[A,(B^{\fS_m}_r)^{\fS_n}_\bul]\]
and write $\iota_r:[A,(B^{\fS_m}_r)^{\fS_n}_\bul]\to[A,(B^{\fS_m}_\bul)^{\fS_n}_\bul]$ for the natural morphism into the colimit. It suffices to show that $(\mu^{m,n}_B)_*\circ\iota_r$ is a group homomorphism for all $r$. For fixed $r$, $s$ and $p$, consider the dotted function that makes the following diagram commute:
\[\xymatrix{\left(\Ex^s\Hom_\Algl(A,(B^{\fS_m}_r)^\Delta)\right)_p\ar[r]^-{\cong}\ar@{..>}[d] & \Hom_\Algl\left(A,(B^{\fS_m}_r)^{\Delta^p}_s\right)\ar[d]^-{\left(\mu_{B}^{\fS_m,\,\Delta^p}\right)_*}\\
\left(\Omega^m\Ex^{r+s}\Hom_\Algl(A,B^\Delta)\right)_p\ar[r]^-{\cong} & \Hom_\Algl\left(A,B^{\fS_n\square\Delta^p}_{r+s}\right)}\]
This dotted function is natural in $p$ and thus induces a morphism of simplicial sets:
\[\xymatrix{\Ex^s\Hom_\Algl\left(A,(B^{\fS_m}_r)^\Delta\right)\ar[r] & \Omega^m\Ex^{r+s}\Hom_\Algl\left(A,B^\Delta\right)}\] The latter is in turn natural in $s$ and thus induces a morphism $\Theta_r$ upon taking colimit:
\[\xymatrix{\Theta_r:\Ex^\infty\Hom_\Algl\left(A,(B^{\fS_m}_r)^\Delta\right)\ar[r] & \Omega^m\Ex^\infty\Hom_\Algl\left(A,B^\Delta\right)}\]
We claim that $(\mu^{m,n}_B)_*\circ\iota_r$ equals the group homomorphism
\[\xymatrix{\pi_n\Theta_r:\pi_n\Ex^\infty\Hom(A,(B^{\fS_m}_r)^\Delta)\ar[r] & \pi_n\Omega^m\Ex^\infty\Hom(A,B^\Delta)\cong \pi_{m+n}\Ex^\infty\Hom(A,B^\Delta)}\]
upon making the identifications of Theorem \ref{thm:bij}. Indeed, this claim follows from the commutativity of the diagram below, which ultimately reduces to the naturality of $\mu$.
\[\xymatrix@R=1em@C=3em{\Hom_\Algl\left(A,(B^{\fS_m}_r)^{\fS_n}_\bul\right)\ar[dd]_-{\cong}\ar[r]^-{\left(\mu^{m,n}_{B}\right)_*\circ\iota_r} & \Hom_\Algl\left(A,B^{\fS_{m+n}}_\bul\right)\ar[d]^-{\cong} \\
 & \left(\Omega^{m+n}\Ex^\infty\Hom_\Algl(A,B^\Delta)\right)_0\ar[d]^-{\cong} \\
\left(\Omega^n\Ex^\infty\Hom_\Algl(A,(B^{\fS_m}_r)^\Delta)\right)_0\ar[r]^-{\Omega^n(\Theta_r)} & \left(\Omega^n\Omega^m\Ex^\infty\Hom_\Algl(A,B^\Delta)\right)_0 }\]

It remains to show that the composite function in \ref{eq:mu2grouphomo} is a group homomorphism. The functors $?\otimes \Z^{\fS_n}_\bul$ and $(?)^{\fS_n}_\bul$ are naturally isomorphic. Then, by Proposition \ref{prop:bundleapp} \ref{eq:tensorgrouphomo} applied to $C_\bul=\Z^{\fS_n}_\bul$, we have a group homomorphism:
\[\xymatrix{\Psi_{\Z^{\fS_n}_\bul}:[A,B^{\fS_m}_\bul]\ar[r] & [A^{\fS_n}_\bul,(B^{\fS_n}_\bul)^{\fS_m}_\bul]}\]
Let $c:I^m\times I^n\overset{\cong}\to I^n\times I^m$ be the commutativity isomorphism. It is easily verified that the following diagram commutes:
\[\xymatrix{[A,B^{\fS_m}_\bul]\ar[r]^-{(?)^{\fS_n}}\ar[d]_-{\Psi_{\Z^{\fS_n}_\bul}} & [A^{\fS_n}_\bul,(B^{\fS_m}_\bul)^{\fS_n}_\bul]\ar[r]^-{\left(\mu^{m,n}_B\right)_*} & [A^{\fS_n}_\bul,B^{\fS_{m+n}}_\bul] \\
[A^{\fS_n}_\bul,(B^{\fS_n}_\bul)^{\fS_m}_\bul]\ar[rr]^-{\left(\mu_B^{n,m}\right)_*} & & [A^{\fS_n}_\bul,B^{\fS_{n+m}}_\bul]\ar[u]_-{c^*}}\]
The function $(\mu_B^{n,m})_*$ is a group homomorphism by Proposition \ref{prop:bundleapp} \ref{eq:mugrouphomo} and $c^*$ is multiplication by $(-1)^{mn}$. The result follows.
\end{proof}

\end{document}